\documentclass[11pt]{article}
\usepackage{amssymb,amsmath,amsfonts,amsthm,mathtools,enumerate,color}
\usepackage{mathrsfs}
\usepackage[toc,page,title,titletoc,header]{appendix}
\usepackage{graphicx,subfig}

\usepackage{booktabs}
\usepackage{paralist}

\usepackage{tikz}
\usetikzlibrary{arrows,positioning,shapes.geometric}

\usepackage{algorithm,algorithmic}
\usepackage{etoolbox}
\preto\subequations{\ifhmode\unskip\fi}

\graphicspath{ {./figures/} }

\usepackage{indentfirst}
\usepackage{multicol}
\usepackage{booktabs}
\usepackage{url}
\usepackage{epstopdf} 

\usepackage{hyperref}
\hypersetup{
    colorlinks=true, 
    linktoc=all,     
    linkcolor=blue,  
}
\numberwithin{equation}{section}

\newtheorem{Theorem}{Theorem}[section]
\newtheorem{Lemma}[Theorem]{Lemma}
\newtheorem{Proposition}[Theorem]{Proposition}
\newtheorem{Corollary}[Theorem]{Corollary}
\newtheorem{Assumption}{H.\!\!}

\theoremstyle{definition}
\newtheorem{Definition}{Definition}[section]

\theoremstyle{remark}
\newtheorem{Remark}{Remark}[section]

 \def\p{\partial} \def\nb{\nonumber}
\def \Vh0{\stackrel{\circ}{V}_h} \def\to{\rightarrow}
     \def\ul{\underline}
\def\Om{\Omega}   
 
\newcommand{\q}{\quad}

\def\l{\label}  \def\f{\frac}  \def\fa{\forall}
\def\b{\beta}  \def\a{\alpha} 
 
\def\eps{\varepsilon}

 \def\t{\times}  
\def\ms{\medskip}  

\def\p{\partial}
\def \la{\langle} \def\ra{\rangle}

\def\u{{\bf u}}   
   \def\A{{\bf A}}

\def\cA{\mathcal{A}}
\def\cB{\mathcal{B}}
\def\cC{\mathcal{C}}

\def\cE{\mathcal{E}}
\def\cF{\mathcal{F}}

\def\cL{\mathcal{L}}
\def\cM{\mathcal{M}}
\def\cN{\mathcal{N}}

\def\bA{{\textbf{A}}}
\def\bB{{\textbf{B}}}

\def\N{{\mathbb{N}}}
\def\bP{\mathbb{P}}
\def\Q{{\mathbb{Q}}}
\def\R{{\mathbb R}}

\newcommand{\ex}{\mathbb{E}}

\DeclareMathOperator*{\argmax}{arg\,max}
\DeclareMathOperator*{\argmin}{arg\,min}
\DeclareMathOperator*{\esssup}{ess\,sup}

\newcommand{\lc}
{\mathrel{\raise2pt\hbox{${\mathop<\limits_{\raise1pt\hbox
{\mbox{$\sim$}}}}$}}}

\newcommand{\gc}
{\mathrel{\raise2pt\hbox{${\mathop>\limits_{\raise1pt\hbox{\mbox{$\sim$}}}}$}}}

\newcommand{\ec}
{\mathrel{\raise2pt\hbox{${\mathop=\limits_{\raise1pt\hbox{\mbox{$\sim$}}}}$}}}

\def\bb{\begin{equation}} \def\ee{\end{equation}}
\def\bbn{\begin{equation*}} \def\een{\end{equation*}}

\def\beqn{\begin{eqnarray}}  \def\eqn{\end{eqnarray}}

\def\beqnx{\begin{eqnarray*}} \def\eqnx{\end{eqnarray*}}

\def\bn{\begin{enumerate}} \def\en{\end{enumerate}}

\def\bd{\begin{description}} \def\ed{\end{description}}

\makeatletter

\newenvironment{figurehere}
  {\def\@captype{figure}}
  {}
\makeatother

\begin{document}

\title{A neural network based policy iteration algorithm with global $H^2$-superlinear convergence for stochastic games on domains
}
\author{
Kazufumi Ito\thanks{Department of Mathematics, North Carolina State University, Raleigh, NC 27607,
United States of America, {\tt  kito@ncsu.edu}}
\and
Christoph Reisinger\thanks{Mathematical Institute, University of Oxford, United Kingdom ({\tt christoph.reisinger@maths.ox.ac.uk, yufei.zhang@maths.ox.ac.uk})}
\and
Yufei Zhang\footnotemark[2]
}
\date{}

\maketitle


\noindent\textbf{Abstract.} 
In this work, we propose a class of numerical schemes for solving semilinear Hamilton-Jacobi-Bellman-Isaacs (HJBI) boundary value problems which arise naturally from exit time problems of diffusion processes with controlled drift. We exploit  policy iteration to reduce the semilinear problem into a sequence of linear Dirichlet problems, which are subsequently approximated by a multilayer feedforward neural network ansatz. We establish that the numerical solutions converge globally in the $H^2$-norm, and further demonstrate that this convergence is superlinear, by interpreting the algorithm as an inexact Newton iteration for the HJBI equation. Moreover, we construct the optimal feedback controls from the numerical value functions and deduce convergence. The numerical schemes and convergence results are then extended to 
oblique derivative boundary conditions. 
Numerical experiments on the stochastic Zermelo navigation problem are presented to illustrate the theoretical results and to demonstrate the effectiveness of the method.


\ms


%
%

\medskip
 

\section{Introduction}
In this article, we propose a class of  numerical schemes for solving  Hamilton-Jacobi-Bellman-Isaacs (HJBI) boundary value problems of the following form:
\bb\l{eq:intro}
 -a^{ij}(x)\p_{ij}u+G(x,u,\nabla u)=0, \q\textnormal{in $\Om\subset \R^n$}; \q Bu=g, \q \textnormal{on $\p\Om$,} 
\ee
where $\Om$ is an open bounded domain, $G$ is the (nonconvex) Hamiltonian defined as
\bb\l{eq:G_intro}
G(x,u,\nabla u)=\max_{\a\in \bA}\min_{\b\in\bB}\big(b^i(x,\a,\b)\p_i u(x)+c(x,\a,\b)u(x)-f(x,\a,\b)\big),
\ee
{with given nonempty compact sets $\bA, \bB$},
and  $B$ is a  boundary operator, i.e., if $B$ is the identity operator, \eqref{eq:intro} is an HJBI Dirichlet problem, while if $Bu=\gamma^i \p_iu +\gamma^0  u$ with  some functions $\{\gamma^i\}_{i=0}^n$, \eqref{eq:intro} is an HJBI oblique derivative problem.  Above and hereafter,  when there is no ambiguity,  we shall adopt the summation convention as in \cite{gilbarg1983}, i.e., repeated equal dummy indices indicate summation from $1$ to $n$. 

It is well-known that the value function of zero-sum stochastic differential games in domains satisfies the HJBI equation \eqref{eq:intro}, and the optimal feedback controls can be constructed from  the derivatives of the solutions (see e.g.~\cite{krylov2014} and references within; see also Section \ref{sec:numerical} for a concrete example). In particular, the HJBI Dirichlet problem corresponds to exit time problems of diffusion processes with controlled drift (see e.g.~\cite{krylov2014,buckdahn2016,mohajerin2016}), while the HJBI oblique derivative problem corresponds to 
state constraints (see e.g.~\cite{lions1984,kushner1991}). 
A nonconvex HJBI equation as above also arises from a penalty approximation of hybrid control problems involving continous controls, optimal stopping and impulse controls, where the HJB (quasi-)variational inequality can be reduced to an HJBI equation by penalizing the difference between the value function and the obstacles (see e.g.~\cite{ito2003,witte2012,reisinger2018,reisinger2019}).
As  \eqref{eq:intro}  in general cannot be solved analytically, it is important to construct effective numerical schemes to find the solution of \eqref{eq:intro} and its derivatives.

The standard approach to solving \eqref{eq:intro} is to 
first discretize the operators in \eqref{eq:intro} by finite difference or finite element methods, and then solve the resulting nonlinear discretized equations by using policy iteration, also known as Howard's algorithm, or generally (finite-dimensional) semismooth Newton methods (see e.g.~\cite{forsyth2007,bokanowski2009,smears2014,reisinger2018}). However, this approach has the following  drawbacks, as do most mesh-based methods: (1) it can be difficult to generate meshes and to construct consistent numerical schemes for problems in domains with complicated geometries; (2) the   number of unknowns in general grows exponentially with the dimension $n$, i.e.,  it suffers from Bellman's \emph{curse of dimensionality}, and hence this approach is infeasible for solving high-dimensional control problems. Moreover, since  policy iteration is applied to a fixed finite-dimensional equation resulting from a particular discretization, it is difficult to infer whether the same convergence rate of policy iteration remains valid as the mesh size tends to zero (\cite{santos2004,bokanowski2009}).
We further remark that,
for a given discrete HJBI equation,
it can be difficult to determine a good initialization of policy iteration
to ensure fast convergence of the algorithm;
see \cite{alla2015} and references therein on possible  accelerated methods.
\color{black}

Recently, numerical methods based on deep neural networks have been designed to solve high-dimensional partial differential equations (PDEs) (see e.g.~\cite{lagaris1998,e2017,berg2018,e2018,hure2018,sirignano2018}). Most of these methods reformulate \eqref{eq:intro} into a nonlinear least-squares problem:
\bb\l{eq:direct}
\inf_{u\in \cF}\| -a^{ij}\p_{ij}u+G(\cdot,u,\nabla u)\|^2_{L^2(\Om)}+\| Bu-g\|^2_{L^2(\p\Om)},
\ee
where $\cF$ is a collection of neural networks with a smooth activation function. Based on  collocation points chosen randomly from the domain,   \eqref{eq:direct} is then reduced into an  empirical risk minimization problem, which is subsequently solved by using stochastic optimization algorithms, in particular the Stochastic Gradient Descent (SGD) algorithm or its variants.
Since these methods avoid mesh generation, they can be adapted to solve PDEs in high-dimensional  domains with complex geometries. Moreover, the choice of smooth activation functions leads to smooth numerical solutions, whose values  can be evaluated everywhere without interpolations. In the following, we shall  refer to these methods as the Direct Method, due to the fact that there is no policy iteration involved. 
 
We observe, however, that the Direct Method also has several serious drawbacks, especially for solving nonlinear  nonsmooth equations including \eqref{eq:intro}. 
Firstly, the  nonconvexity of both the deep neural networks and the Hamiltonian $G$  leads to a  nonconvex empirical minimization  problem, for which there is no theoretical guarantee on the convergence of SGD to a minimizer (see e.g.~\cite{shamir2013}). In practice, training a network with a desired accuracy could take  hours or days (with  hundreds of thousands of  iterations) due to the slow convergence of SGD. Secondly, each SGD iteration requires the evaluation of  $\nabla G$ (with respect to $u$ and $\nabla u$) on sample points,
but $\nabla G$ is not necessarily defined everywhere due to the nonsmoothness of $G$.
Moreover, evaluating the function $G$ (again on a large set of sample points) can be expensive, especially when the sets $\bA$ and $\bB$ are of high dimensions,
as we do not require more regularity than continuity of the coefficients with respect to the controls, so that approximate optimization may only be achieved by exhaustive search over a discrete coverage of the compact control set.
Finally, as we shall see in Remark \ref{rmk:H3/2}, merely including an $L^2(\p\Om)$-norm of the boundary data in the loss function \eqref{eq:direct} does not generally lead to convergence of  the derivatives of numerical solutions or the corresponding feedback control laws.


In this work, we  propose an efficient neural network based policy iteration algorithm for solving \eqref{eq:intro}.
At  the $(k+1)$th iteration, $k\ge 0$, we shall update the control laws $(\a^k,\b^k)$ by performing pointwise maximization/minimization of the Hamiltonian $G$ based on the previous iterate $u^k$,  and obtain the next iterate $u^{k+1}$ by solving a linear boundary value problem, whose coefficients involve the  control laws  $(\a^k,\b^k)$.
This reduces the (nonconvex) semilinear problem into a sequence of \textit{linear} boundary value problems, which are subsequently approximated by a multilayer  neural network ansatz.  
Note that compared to Algorithm Ho-3 in  \cite{bokanowski2009} for discrete HJBI equations, which  requires  to  solve  a nonlinear HJB subproblem (involving minimization over the set $\bB$) for each iteration, our algorithm only requires to solve a linear subproblem for each iteration, 
hence it is in general   more efficient, especially when the dimension of $\bB$ is high.

Policy iteration (or Successive Galerkin Approximation) was employed  in \cite{beard1997, beard1998, kalise2018,kerimkulov2018} to solve \textit{convex HJB equations} on the whole space $\R^n$. Specifically, \cite{beard1997, beard1998, kalise2018} approximate the solution to each linear equation via a separable polynomial ansatz (without concluding any convergence rate), while \cite{kerimkulov2018} assumes each linear equation is solved sufficiently accurately (without specifying a numerical method), and deduces pointwise \textit{linear} convergence. 
The continuous policy iteration in \cite{kalise2018}
has also been applied to solve HJBI equations on $\R^n$ in \cite{kalise2019},
which 
is a direct  extension of Algorithm Ho-3 in  \cite{bokanowski2009}
and
still requires to solve a nonlinear  HJB subproblem at each iteration.
\color{black}
In this paper, we propose an easily implementable accuracy criterion for the numerical solutions  of the \textit{linear} PDEs which ensures the numerical solutions  converge  superlinearly in a suitable function space for nonconvex HJBI equations from an arbitrary initial guess.

Our algorithm enjoys the main advantage of the Direct Method, i.e., it is a mesh-free method and can be applied to solve high-dimensional stochastic games. Moreover, by utilizing the superlinear convergence of policy iteration, our algorithm effectively reduces the number of pointwise maximization/minimization over the  sets $\bA$ and $\bB$, and significantly reduces the computational cost of the Direct Method, especially for high dimensional control sets. The superlinear convergence of policy iteration also helps  eliminate the oscillation caused by SGD, which leads to smoother and more rapidly decaying loss curves in both the training and validation processes (see Figure \ref{fig:hjbi}). Our algorithm further allows training of the feedback controls on a separate network architecture from that representing the value function, or adaptively adjusting the architecture of networks for each policy iteration.

A major theoretical contribution of this work is the proof of global superlinear convergence of the 
policy iteration algorithm  for the HJBI equation \eqref{eq:intro} in $H^2(\Om)$, which is novel even for  HJB equations (i.e., one of the sets $\bA$ and $\bB$ is singleton). Although the (local) superlinear convergence of policy iteration for discrete equations has been proved in various works (e.g.~\cite{puterman1979,santos2004,forsyth2007,bokanowski2009,witte2012,smears2014,reisinger2018}),  to the best of our knowledge, there is no published work on the superlinear convergence of policy iteration for HJB PDEs in a function space, nor on the  global convergence of policy iteration for solving nonconvex HJBI equations. 

Moreover, this is the first paper which demonstrates the  convergence of  neural network based  methods  for the solutions and their (first and second order) derivatives of nonlinear PDEs with merely measurable coefficients (cf.~\cite{han2016, han2018,hure2018,sirignano2018}).
We will also prove the pointwise  convergence of the numerical solutions and their derivatives, which subsequently enables us to  construct the optimal feedback controls from the numerical value functions and deduce convergence.

Let us briefly comment on the  main difficulties encountered in studying  the convergence of policy iteration for HJBI equations.  
Recall that at  the $(k+1)$th iteration, we 
need to solve a linear boundary value problem, whose coefficients involve the  control laws  $(\a^k,\b^k)$,
obtained by performing pointwise maximization/minimization of the Hamiltonian $G$.
The uncountability of the state space $\Om$ and the nonconvexity of the Hamiltonian require us to exploit several technical measurable selection arguments to ensure the measurability of the controls  $(\a^k,\b^k)$, which is essential for the well-definedness  of the linear boundary value problems and the algorithm.

Moreover,  the nonconvexity of the Hamiltonian prevents us from following the  arguments in \cite{santos2004,forsyth2007,bokanowski2009}  for discrete HJB equations to establish  the \emph{global} convergence of our inexact policy iteration algorithm for HJBI equations.
In fact,  a crucial step in the arguments for discrete HJB equations is to use the discrete maximum principle and show 
the iterates generated by policy iteration converge monotonically with an arbitrary initial guess, which subsequently implies the global convergence of the iterates. However, this monotone convergence is in general false for the iterates generated by the inexact policy iteration algorithm, due to the nonconvexity of the Hamiltonian and the fact that each linear equation is only solved approximately. We shall present a novel analysis technique for 
establishing the  global convergence of our inexact policy iteration algorithm, by 
{interpreting it as a fixed point iteration in $H^2(\Om)$}.

Finally, we remark that the proof of \emph{superlinear} convergence of our algorithm is significantly different from the arguments for discrete equations. Instead of working with the sup-norm for (finite-dimensional) discrete equations as in  \cite{santos2004,forsyth2007,bokanowski2009,witte2012,reisinger2018}, we employ a two-norm framework to establish the generalized differentiability of HJBI operators, where the norm gap is essential as has already been pointed out in \cite{hintermuller2002,ulbrich2011,smears2014}.  Moreover, by taking advantage of the  fact that  the Hamiltonian only involves   low order terms, we further demonstrate that the inverse of the generalized derivative is uniformly bounded. Furthermore, we include a suitable  fractional Sobolev norm of the boundary data in the loss functions used in the training process, which is crucial for the $H^2(\Om)$-superlinear convergence of the neural network based policy iteration algorithm.

We organize this paper as follows. Section \ref{sec:D} states the main assumptions and recalls basic results for HJBI Dirichlet problems. In Section \ref{sec:pi_D} we  propose a policy iteration scheme for HJBI Dirichlet problems and establish its global superlinear convergence. Then in Section \ref{sec:dpi_D}, we shall introduce the neural network based policy iteration algorithm, establish its various convergence properties, and construct convergent approximations to optimal feedback controls. We extend the algorithm and convergence results to HJBI oblique derivative problems in Section \ref{sec:oblique}. 
Numerical examples for two-dimensional stochastic Zermelo navigation problems are presented in Section \ref{sec:numerical} to confirm the theoretical findings and to illustrate the effectiveness of our algorithms. 
The Appendix  collects some basic results which are used in this article, and gives a proof for the main result on the HJBI oblique derivative problem.

\section{HJBI Dirichlet problems}\l{sec:D}
In this section, we introduce the HJBI Dirichlet boundary value problems of our interest,   recall the appropriate notion of solutions, and state the main assumptions on its coefficients.  We start with several important  spaces  used frequently throughout this work. 

Let $n\in \N$ and $\Om$ be a  bounded $C^{1,1}$ domain in $\R^n$, i.e., a bounded open connected subset of $\R^n$  with a $C^{1,1}$ boundary. For each  integer $k\ge 0$ and real $p$ with $1\le p<\infty$,  we denote by $W^{k,p}(\Om)$  the  standard Sobolev space of real functions with their weak derivatives of order up to $k$ in the Lebesgue space $L^p(\Om)$. When $p=2$, we use $H^k(\Om)$ to denote $W^{k,2}(\Om)$. We further  denote by $H^{1/2}(\p\Om)$ and $H^{3/2}(\p\Om)$ the spaces of traces from $H^1(\Om)$ and  $H^{2}(\Om)$, respectively  (see  \cite[Proposition 1.1.17]{grisvard1985}),  
which can   be equivalently defined by using the surface measure  $\sigma$ on the boundaries  $\p\Om$ as follows (see e.g.~\cite{garroni2002}):
\begin{align}
  \|g\|_{H^{1/2}(\p\Om)}
 &=
 \textstyle
 \big[\int_{\p\Om}
 |g|^2\,d\sigma
+\iint_{\p\Om \t \p\Om}\f{|g(x)-g(y)|^2}{|x-y|^{n}}\,d\sigma(x)d\sigma(y)
\big]^{1/2},
\l{eq:H1/23/2}
\\
  \|g\|_{H^{3/2}(\p\Om)}
 &=
 \textstyle
\big[ \int_{\p\Om}  
 \big(|g|^2+\sum_{i=1}^n|\p_i g|^2\big)\,d\sigma
 +\sum_{i=1}^n\iint_{\p\Om \t \p\Om}\f{|\p_i g(x)-\p_ig(y)|^2}{|x-y|^{n}}\,d\sigma(x)d\sigma(y)
 \big]^{1/2}
 \nb.
\end{align}

We shall consider the following HJBI equation with  nonhomogeneous Dirichlet boundary data:
\begin{subequations}\l{eq:D}
\begin{align}
F(u)&\coloneqq -a^{ij}(x)\p_{ij}u+G(x,u,\nabla u)=0, \q\textnormal{a.e.~$\Om$}, \l{eq:hjb}\\
\tau u&=g, \q \textnormal{on $\p\Om$.} \l{eq:bc}
\end{align}
\end{subequations}
where  the nonlinear Hamiltonian is given as in \eqref{eq:intro}:
\bb\l{eq:G}
G(x,u,\nabla u)=\max_{\a\in \bA}\min_{\b\in\bB}\big(b^i(x,\a,\b)\p_i u(x)+c(x,\a,\b)u(x)-f(x,\a,\b)\big).
\ee
Throughout this paper, we shall focus on the strong solution to \eqref{eq:D}, i.e., a twice weakly differentiable function $u\in H^2(\Om)$ satisfying the HJBI equation \eqref{eq:hjb} almost everywhere in $\Om$, and the  boundary values on $\p \Om$ will be 
interpreted as traces of the corresponding Sobolev space. For instance, $\tau u = g$ 
 on $\p\Om$ in \eqref{eq:bc} means that the trace of $u$ is equal to $g$ in $H^{3/2}(\p\Om)$, where $\tau\in \cL(H^2(\Om), H^{3/2}(\p\Om))$ denotes the trace operator  (see \cite[Proposition 1.1.17]{garroni2002}). See Section \ref{sec:oblique} for boundary conditions involving the derivatives of solutions.


We now list the main assumptions on the coefficients of \eqref{eq:D}.
\begin{Assumption}\l{assum:D}
Let $n\in \N$,  $\Om\subset \R^n$ be a  bounded $C^{1,1}$ domain, 
  $\bA$ be a nonempty finite set, and $\bB$ be a nonempty compact  metric space.
Let $g\in H^{3/2}(\p\Om)$,  $\{a^{ij}\}_{i,j=1}^n\subseteq  {C(\bar{\Om})}$ satisfy the following ellipticity condition with a constant  $\lambda>0$:
\begin{align*}
&\sum_{i,j=1}^na^{ij}(x)\xi_i\xi_j\ge \lambda\sum_{i=1}^n \xi_i^2, \q    \textnormal{for all $\xi\in \R^n$ and $x\in {\Om}$},
\end{align*}
and $\{b^i\}_{i=1}^n,c,f\in L^\infty({\Om}\t \bA\t \bB)$ satisfy that $c\ge 0$ on ${\Om}\t \bA\t \bB$, and  that $\phi(x,\a,\cdot):\bB\to \R$ is continuous, for all $\phi=b^i,c,f$ and $(x,\a)\in \Om\t \bA$.

\end{Assumption}

As we shall see in Theorem \ref{thm:semismooth_F}
and Corollary \ref{cor:semismooth_F}, the finiteness of the set $\bA$ enables us to establish the semismoothness of the HJBI operator \eqref{eq:hjb}, whose coefficients involve a general nonlinear dependence on the parameters $\a$ and $\b$. If all coefficients of \eqref{eq:hjb} are in a separable form, i.e., it holds for all $\phi=b^i,c,f$ that $\phi(x,\a,\b)=\phi_1(x,\a)+\phi_2(x,\b)$ for some functions $\phi_1,\phi_2$ (e.g.~the penalized equation for variational inequalities with bilateral obstacles in \cite{ito2003}), then we can relax the finiteness of $\bA$ to the same conditions on $\bB$.

Finally,  in this work we focus on boundary value problems in a $C^{1,1}$ domain   to simplify the presentation, but  the numerical schemes and  their convergence analysis can be extended to problems in nonsmooth convex domains with sufficiently regular coefficients (see e.g.~\cite{grisvard1985,smears2014}).


%

We end this section by proving   the uniqueness of solutions to the Dirichlet problem  \eqref{eq:D} in $H^2(\Om)$. The existence of strong solutions shall be established constructively via policy iteration below (see Theorem \ref{thm:pi_global}). 
\begin{Proposition}\l{prop:wp_d}
Suppose (H.\ref{assum:D}) holds. Then the Dirichlet problem  \eqref{eq:D} admits at most one strong solution $u^*\in H^2(\Om)$. 
\end{Proposition}
\begin{proof}
Let $u,v\in H^2(\Om)$ be two strong solutions to  \eqref{eq:D},  we  consider the following linear homogeneous Dirichlet problem:
\bb\l{eq:d_u-v}
-a^{ij}(x)\p_{ij}w+\tilde{b}^{i}(x)\p_i w+\tilde{c}(x)w=0, \q \textnormal{a.e. in $\Om$}; \q \tau w=0, \q \textnormal{on $\p\Om$},
\ee
where we define the following measurable functions: for each $i=1,\ldots, n$,
\begin{align*}
\tilde{b}^i(x)&=\begin{cases}
\f{G\big(x,v,((\p_jv)_{1\le j<i},\p_i u,(\p_ju)_{i< j\le n})\big)-G\big(x,v,((\p_jv)_{1\le j<i},\p_i v,(\p_ju)_{i< j\le n})\big)}{(\p_i u-\p_i v)(x)}, &\textnormal{on $\{x\in \Om\mid \p_i(u-v)(x)\not =0\}$,}\\
0,&\textnormal{otherwise,}
\end{cases}\\
\tilde{c}(x)&=\begin{cases}
\f{G(x,u,\nabla u)-G(x,v,\nabla u)}{(u-v)(x)}, & \textnormal{on $\{x\in \Om\mid (u-v)(x)\not =0\}$,}\\
0, &\textnormal{otherwise,}
\end{cases}
\end{align*}
with the Hamiltonian $G$ defined as in \eqref{eq:G}. Note that the boundedness of coefficients implies that 
$\{\tilde{b}^i\}_{i=1}^n\subseteq L^\infty(\Om)$, and $\tilde{c}\in L^\infty(\Om)$. Moreover, one can directly verify that the following inequality holds  for all parametrized functions $(f^{\a,\b},g^{\a,\b})_{\a\in \bA,\b\in \bB}$: for all $ x\in \R^n$,
$$\inf_{(\a,\b)\in \bA\t \bB}f^{\a,\b}(x)-g^{\a,\b}(x)\le \inf_{\a\in \bA}\sup_{\b\in \bB} f^{\a,\b}(x)-\inf_{\a\in \bA}\sup_{\b\in \bB} g^{\a,\b}(x)\le \sup_{(\a,\b)\in \bA\t \bB}f^{\a,\b}(x)-g^{\a,\b}(x),$$
which together with   (H.\ref{assum:D}) leads to the estimate  that $\tilde{c}(x)\ge \inf_{(\a,\b)\in \bA\t\bB}c(x,\a,\b)\ge 0$ 
on the set $\{x\in \Om\mid (u-v)(x)\not =0\}$, and hence we have $\tilde{c}\ge 0$ a.e.~$\Om$. Then, we can deduce from  Theorem \ref{thm:D_regularity}  that the Dirichlet problem \eqref{eq:d_u-v} admits a unique strong solution $w^*\in H^2(\Om)$ and $w^*=0$. Since $w=u-v\in H^2(\Om)$ satisfies \eqref{eq:d_u-v} a.e.~in $\Om$ and $\tau w=0$, we see that $w=u-v$ is a strong solution to \eqref{eq:d_u-v} and hence $u-v=w^*=0$, which subsequently implies the uniqueness of strong solutions to the Dirichlet problem  \eqref{eq:D}.
\end{proof}

\section{Policy iteration for HJBI Dirichlet problems}\l{sec:pi_D}

In this section, we  propose a policy iteration algorithm for solving the Dirichlet problem \eqref{eq:D}. We shall also establish the global superlinear convergence of the algorithm, which subsequently gives a constructive proof for the existence of a strong solution to the Dirichlet problem \eqref{eq:D}.

We start by presenting the  policy iteration scheme for the HJBI equations in Algorithm \ref{alg:pi}, which extends the policy iteration algorithm (or Howard's algorithm) for discrete HJB equations (see e.g.~\cite{forsyth2007,bokanowski2009,reisinger2018}) to the continuous setting. 


%
%
%

\begin{algorithm}[!h]
\caption{Policy iteration algorithm for Dirichlet problems.}
\label{alg:pi}
\vspace{-3mm}
\bn 
\item Choose an initial guess $u^0$ in $H^2(\Om)$, and set $k=0$.
\item Given the  iterate  $u^{k}\in H^2(\Om)$, update the following control laws: for all $ \a\in \bA,x\in \Om $,
\begin{align}
\a^{k}(x)&\in \arg \max_{\a\in \A}\bigg[\min_{\b\in \bB}\big(b^i(x,\a,\b)\p_i u^k(x)+c(x,\a,\b)u^k(x)-f(x,\a,\b)\big)\bigg],\l{eq:a_k}\\
\b^{k}(x) &\in \arg\min_{\b\in \bB}\big(b^i(x,\a^k(x),\b)\p_i u^k(x)+c(x,\a^k(x),\b)u^k(x)-f(x,\a^k(x),\b)\big). \l{eq:b_k}
\end{align}
\item Solve the \emph{linear} Dirichlet problem for $u^{k+1}\in H^2(\Om)$: 
\bb\l{eq:linear}
-a^{ij}\p_{ij}u+b^{i}_{k}\p_i u+c_{k}u-f_{k}=0, \q \textnormal{in $\Om$}; \q \tau u=g, \q \textnormal{on $\p\Om$},
\ee
where $\phi_{k}(x)\coloneqq \phi(x,\a^{k}(x),\b^{k}(x))$ for $\phi=b^i,c,f$.
\item If $\|u^{k+1}-u^k\|_{H^2(\Om)}=0$, then terminate with outputs $u^{k+1},\a^k$ and $\b^k$, otherwise increment $k$ by one and go to step 2.
\en
\vspace{-3mm}
\end{algorithm}

The remaining part of this section is devoted to the  convergence analysis of Algorithm \ref{alg:pi}. For notational simplicity,
we first introduce  two auxiliary  functions: 
for each $(x, \u,\a,\b) \in \Om\t \R^{n+1}\t \bA\t \bB$ with $\u=(z,p)\in \R\t \R^{n}$, we shall define the following functions
\begin{align}
\ell(x,\u,\a,\b)&\coloneqq b^i(x,\a,\b)p_i+c(x,\a,\b)z-f(x,\a,\b), \l{eq:l}\\
 h(x,\u,\a)&\coloneqq \min_{\b\in \bB} \ell(x,\u,\a,\b). \l{eq:h}
\end{align}
Note that for all $k\ge 0$ and $x\in\Om$, by setting $\u^k(x)=(u^k(x),\nabla u^k(x))$, we can see from \eqref{eq:a_k} and \eqref{eq:b_k} that
\begin{align}\l{eq:maxmin}
\ell(x,\u^k(x),\a^k(x),\b^k(x))
=
\min_{\b\in \bB}\ell(x,\u^k(x),\a^k(x),\b)
=\max_{\a\in \bA}\min_{\b\in \bB}\ell(x,\u^k(x),\a,\b).
\end{align}

We then recall several important concepts, which play a pivotal role  in our subsequent  analysis.
 The first  concept ensures the existence of measurable feedback controls and the well-posedness of Algorithm \ref{alg:pi}.
\begin{Definition}
Let $(S, \Sigma)$ be a measurable space, and let $X$ and $Y$ be topological spaces. A function $\psi : S \t X \to Y$ is a Carath\'{e}odory function if:
\bn
\item for each $x\in X$, the function $\psi_x = \psi(\cdot, x): S \to Y$ is $(\Sigma, \cB_Y )$-measurable,
where $\cB_Y$ is the Borel $\sigma$-algebra of the topological space $Y$; and
\item for each $s\in S$, the function $\psi_s = \psi(s,\cdot): X \to Y$ is continuous.
\en
\end{Definition}
\begin{Remark}\l{rmk:cara}
It is well-known that if $X,Y$ are two complete separable metric spaces, and $\psi: S \t X \to Y$ is a Carath\'{e}odory function, then for any given measurable function $f:S\to  X$, the composition function $s\to \psi(s,f(s))$ is measurable (see e.g.~\cite[Lemma 8.2.3]{aubin1990}). 
Since any compact metric space is complete and separable,
it is clear that (H.\ref{assum:D}) implies that the coefficients $b^i,c,f$ are Carath\'{e}odory functions (with $S=\Om$ and $X=\bA\t\bB$). Moreover, one can easily check that  both $\ell$ and $h$ are Carath\'{e}odory functions, i.e., $\ell$ (resp.~$h$) is  continuous  in $(\u,\a,\b)$ (resp.~$(\u,\a)$) and measurable in $x$ (see Theorem \ref{thm:measurable_selection} for the measurability of $h$ in $x$).
\end{Remark}

We now  recall a generalized differentiability concept for nonsmooth operators between Banach spaces, which is referred as  semismoothness  in \cite{ulbrich2011} and  slant differentiability in   \cite{chen2000,hintermuller2002}. It is well-known (see e.g.~\cite{bokanowski2009,smears2014,reisinger2018}) that 
 the HJBI operator in \eqref{eq:hjb}  is in general non-Fr\'{e}chet-differentiable, and 
 this generalized differentiability is essential for showing the superlinear convergence of policy iteration applied to HJBI equations.
\begin{Definition}
Let $F:V\subset Y\mapsto Z$ be defined on a open subset $V$ of the Banach space $Y$ with images in the Banach space $Z$. In addition, let $\p^*F: V \rightrightarrows \cL(Y ,Z)$ be a given a set-valued mapping with nonempty images, i.e., $\p^* F(y)\not =\emptyset$ for all $y \in V$. We say $F$ is $\p^*F$-semismooth in $V$ if for any given $y\in V$, we have that $F$ is continuous near $y$, and
$$
\sup_{M\in\p^* F(y+s)}\|F(y+s)-F(y)-Ms\|_Z=o(\|s\|_Y), \q \textnormal{as $\|s\|_Y\to 0$.}
$$ 
The set-valued mapping $\p^*F$  is called a generalized differential of $F$ in $V$.
\end{Definition}
\begin{Remark}
As in \cite{ulbrich2011}, we always require that $\p^*F$ has a nonempty image, and hence the $\p^*F$-semismooth of $F$ in $V$ shall automatically imply that the image of  $\p^*F$ is nonempty on $V$.
\end{Remark}

Now we are ready to  analyze Algorithm \ref{alg:pi}.
We  first prove the  semismoothness of the  Hamiltonian $G$ defined as in \eqref{eq:G},  by viewing it as the composition of  a pointwise maximum operator and a  family of HJB operators parameterized by the control $\a$. Moreover, we shall simultaneously establish   that, for each iteration, one can select  measurable control  laws $\a^k,\b^k$ to ensure the measurability of the controlled coefficients $b^i_k,c_k,f_k$ in  the linear problem  \eqref{eq:linear}, which is essential for the well-posedness of strong solutions to \eqref{eq:linear}, and the well-definedness of Algorithm \ref{alg:pi}.

The following proposition establishes the semismoothness of a parameterized family of first-order HJB operators, which extends the result for scalar-valued HJB operators in \cite{smears2014}. Moreover, by taking advantage of the  fact that  the operators involve  only first-order derivatives, we are able to establish that they are  semismooth from $H^2(\Om)$ to  $L^p(\Om)$ for some $p>2$ (cf.~\cite[Theorem 13]{smears2014}), which is essential for  the superlinear convergence of Algorithm \ref{alg:pi}.

\begin{Proposition}\l{prop:F1}
Suppose (H.\ref{assum:D}) holds. Let $p$ be a given constant satisfying  $p\ge 1$ if $n\le 2$ and $p\in [1,\f{2n}{n-2})$ if $n>2$, and let 
$F_1:H^2(\Om)\to (L^p(\Om))^{|\bA|}$ be the HJB operator defined by
$$
F_1(u)\coloneqq \bigg(\min_{\b\in\bB}\big(b^i(x,\a,\b)\p_i u+c(x,\a,\b)u-f(x,\a,\b)\big)\bigg)_{\a\in \bA}, \q \fa u\in H^2(\Om).
$$
Then $F_1$  is Lipschitz continuous and $\p^* F_1$-semismooth  in $H^2(\Om)$ with a generalized differential
\begin{align*}
\p^* F_1:H^2(\Om)&\rightarrow \cL(H^2(\Om), (L^p(\Om))^{|\bA|})
\end{align*}
defined  as follows: for any $u\in H^2(\Om)$,  we have
\bb\l{eq:p_F1}
\p^*F_1(u)\coloneqq \bigg(b^{i}(\cdot,\a,\b^{u}(\cdot,\a))\p_i +c(\cdot,\a,\b^{u}(\cdot,\a))\bigg)_{\a\in \bA},
\ee
where  $\b^{u}:\Om\t \bA\to \bB$ is any jointly measurable function such that for all $\a\in \bA$ and $x\in \Om$,
\bb\l{eq:b^u}
\b^u(x,\a)\in  \arg\min_{\b\in \bB}\left(b^i(x,\a,\b)\p_i u(x)+c(x,\a,\b)u(x)-f(x,\a,\b)\right).
\ee
\end{Proposition}
\begin{proof}
Since $\bA$ is a finite set, we shall assume without  loss of generality that, the Banach space $(L^p(\Om))^{|\bA|}$ is endowed with the usual product norm $\|\cdot\|_{p,\bA}$, i.e., for all  ${u}\in (L^p(\Om))^{|\bA|}$, $\|{u}\|_{p,\bA}=\sum_{\a\in \bA}\|u(\cdot,\a)\|_{L^p(\Om)}$.
Note that the Sobolev embedding theorem shows that the following  injections  are continuous: $H^2(\Om)\hookrightarrow W^{1,q}(\Om)$, for all $q\ge 2, n\le 2$, and $H^2(\Om)\hookrightarrow W^{1,2n/(n-2)}(\Om)$, for all $n>2$. Thus for any given $p$ satisfying the conditions in Proposition \ref{prop:F1}, we can find $r\in (p,\infty)$ such that the injection $H^2(\Om)\hookrightarrow W^{1,r}(\Om)$ is continuous.  Then, the boundedness of $b^i, c,f$ implies that the mappings $F_1$ and $\p^* F_1$ are well-defined, and $F_1:H^2(\Om)\to (L^p(\Om))^{|\bA|}$ is Lipschitz continuous.

Now we show that  the mapping $\p^*F_1$ has a nonempty image from $W^{1,r}(\Om)$ to $(L^p(\Om))^{|\bA|}$, where we 
choose $r\in (p,\infty)$ such that the injection $H^2(\Om)\hookrightarrow W^{1,r}(\Om)$ is continuous, and 
naturally extend the operators $F_1$ and $\p^*F_1$ from $H^2(\Om)$ to $W^{1,r}(\Om)$.
For each $u\in  W^{1,r}(\Om)$, we consider the Carath\'{e}odory function $g:\Om\t\bA\t \bB\to \R$ such that $g(x,\a,\b)\coloneqq \ell(x,(u,\nabla u)(x),\a,\b)$ for all $(x,\a,\b)\in \Om\t \bA\t \bB$, where $\ell$ is defined by \eqref{eq:h}.
Theorem \ref{thm:measurable_selection} shows there exists a function  $\b^u:\Om\t \bA\to \bB$ satisfying \eqref{eq:b^u}, i.e.,
$$\b^u(x,\a)\in \argmin_{\b\in\bB}\ell\big(x,\big(u(x),\nabla u(x)\big),\a,\b\big),\q \fa (x,\a)\in \Om\t \bA,$$
and $\b^u$ is jointly measurable with respect to the product $\sigma$-algebra on $ \Om\t \bA$. 
Hence $\p^*F_1(u)$ is nonempty for all $u\in W^{1,r}(\Om)$.


We proceed to show that  the operator $F_1$ is in fact $\p^*F_1$-semismooth from $W^{1,r}(\Om)$ to $(L^p(\Om))^{|\bA|}$,
which implies the desired conclusion due to the continuous embedding $H^2(\Om)\hookrightarrow W^{1,r}(\Om)$.
For each $\a\in \bA$, we denote by $F_{1,\a}:W^{1,r}(\Om)\to L^p(\Om)$ the $\a$-th component of $F_1$, and by $\p^*F_{1,\a}$  the $\a$-th component of $\p^*F_1$.
Theorem \ref{thm:uhc} and the continuity of $\ell$ in $\u$ show that for each $(x,\a)\in \Om\t \bA$, the set-valued mapping 
$$
 \u\in \R^{n+1} \rightrightarrows \arg \min_{\b\in \bB}\ell(x,\u,\a,\b)\subseteq	 \bB,
$$
 is upper hemicontinuous, from which,  by  following  precisely the steps in the arguments for \cite[Theorem 13]{smears2014}, 
we can  prove  that   $F_{1,\a}:W^{1,r}(\Om)\to L^p(\Om)$ is  $\p^*F_{1,\a}$-semismooth.
Then, by using the fact that a direct product of semismooth operators is again semismooth with respect to the direct product of the generalized differentials of the components (see \cite[Proposition 3.6]{ulbrich2011}), we can deduce that $F_1:W^{1,r}(\Om)\to (L^p(\Om))^{|\bA|}$ is semismooth with respect to the generalized differential $\p^*F_1$ and finishes the proof.
\end{proof}

We then establish the semismoothness of  a general  pointwise maximum operator, by extending the result in \cite{hintermuller2002} for the max-function $f:x\in \R\to \max(x,0)$.
\begin{Proposition}\l{prop:F2}
Let $p\in (2,\infty)$ be a  given constant,  $\bA$ be a finite set, and $\Om$ be a bounded subset of $\R^n$. Let 
$F_2:(L^p(\Om))^{|\bA|}\to L^2(\Om)$ be the pointwise maximum operator such that for each ${u}=(u(\cdot,\a))_{\a\in \bA}\in (L^p(\Om))^{|\bA|}$, 
\bb\l{eq:p_F2}
F_2({u})(x)\coloneqq \max_{\a\in \bA}u(x, \a), \q \textnormal{for a.e. $x\in \Om$.}
\ee
Then $F_2$  is $\p^* F_2$-semismooth  in $(L^p(\Om))^{|\bA|}$ with a generalized differential
\begin{align*}
\p^* F_2:(L^p(\Om))^{|\bA|}&\rightarrow \cL((L^p(\Om))^{|\bA|}, L^2(\Om))
\end{align*}
defined  as follows: for any ${u}=(u(\cdot,\a))_{\a\in \bA},{v}=(v(\cdot,\a))_{\a\in \bA}\in (L^p(\Om))^{|\bA|}$,  we have
$$
\big(\p^*F_2({u}){v}\big)(x)\coloneqq v(x,\a^{{u}}(x)), \q \textnormal{for  $x\in \Om$,}
$$
where  $\a^{{u}}:\Om\to \bA$ is any measurable function such that 
\bb\l{eq:a^u}
\a^{{u}}(x)\in  \arg\max_{\a\in \bA}\left(u(x,\a)\right), \q \textnormal{for  $x\in \Om$.}
\ee
Moreover, $\p^*F_2({u})$ is uniformly bounded (in the operator norm) for all ${u}\in  (L^p(\Om))^{|\bA|}$.
\end{Proposition}
\begin{proof}

Let the Banach space $(L^p(\Om))^{|\bA|}$ be endowed with the  product norm $\|\cdot\|_{p,\bA}$ defined as in the proof of Proposition \ref{prop:F1}.
We first show the mappings $F_2$ and $\p^*F_2$ are well-defined, $\p^*F_2$ has nonempty images, and $\p^*F_2(u)$ is uniformly bounded for $u\in (L^p(\Om))^{|\bA|}$.

The finiteness of $\bA$  implies that any  ${u}\in (L^p(\Om))^{|\bA|}$ 
  can also  be viewed as  a Carath\'{e}odory function $u:\Om\t \bA\to \R$. Hence for any given ${u}\in (L^p(\Om))^{|\bA|}$, we can deduce from Theorem \ref{thm:measurable_selection} the existence of a measurable function $\a^{{u}}:\Om\to \bA$ satisfying \eqref{eq:a^u}. Moreover, for  any given measurable function $\a^{{u}}:\Om\to \bA$ and ${v}\in (L^p(\Om))^{|\bA|}$, the function $\p^*F_2({u}){v}$ remains  Lebesgue measurable (see Remark \ref{rmk:cara}). 
Then, for any given ${u}\in (L^p(\Om))^{|\bA|}$ with $p>2$, one can easily  check that  $F_2({u})\in L^2(\Om)$, and $\p^*F_2({u})\in \cL((L^p(\Om))^{|\bA|}, L^2(\Om))$, which subsequently implies that $F_2$ and $\p^*F_2$ are well-defined, and the image of $\p^*F_2$ is nonempty on $(L^p(\Om))^{|\bA|}$. Moreover, for any $u,v\in(L^p(\Om))^{|\bA|}$, H\"{o}lder's inequality leads to the following estimate:
$$
\int_\Om |v(x,\a^{{u}}(x))|^2\,dx\le \int_\Om \sum_{\a\in \bA} |v(x,\a)|^2\,dx\le \sum_{\a\in \bA}|\Om|^{(p-2)/p}   \|v(\cdot,\a)\|_{L^p(\Om)}^2,
$$
which shows that $\|\p^*F_2({u})\|_{\cL((L^p(\Om))^{|\bA|}, L^2(\Om))}\le |\Om|^{(p-2)/(2p)}$ for all ${u}\in  (L^p(\Om))^{|\bA|}$.

Now  we prove by contradiction that the operator $F_2$ is $\p^*F_2$-semismooth. Suppose there exists a constant $\delta>0$ and functions ${u}, \{{v}_k\}_{k=1}^\infty\in  (L^p(\Om))^{|\bA|}$ such that $\|{v}_k\|_{p,\bA}\to 0$ as $k\to \infty$, and 
\bb\l{eq:F2_contradiction}
\|F_2({u}+{v}_k)-F_2({u})-\p^*F_2({u}+{v}_k){v}_k\|_{L^2(\Om)}/\|{v}_k\|_{p,\bA}\ge \delta>0,\q k\in \N,
\ee
where for each $k\in \N$, $\p^*F_2({u}+{v}_k)$ is defined with some measurable function $\a^{{u+v_k}}:\Om\to \bA$. 
Then, by  passing to a subsequence, we may assume that for all $\a\in \bA$, the sequence $\{{v}_k(\cdot,\a)\}_{k\in \N}$ converges to zero pointwise a.e.~in $ \Om$, as $k\to \infty$.

For notational simplicity, we define $ \Sigma(x,u)\coloneqq \arg\max_{\a\in \bA}\left(u(x,\a)\right)$ for all $u\in (L^p(\Om))^{|\bA|}$ and  $x\in \Om$. Then
for a.e.~$x\in \Om$, we have $\lim_{k\to \infty}{v}_k(x,\a)= 0$ for all $\a\in \bA$,  $\a^{{u+v_k}}(x)\in \Sigma(x,u+v_k)$ for all $k\in\N$. By using the finiteness of  $\bA$ and the convergence of  $\{{v}_k(\cdot,\a)\}_{k\in \N}$, it is straightforward to prove by contradiction  that  for all such $x\in \Om$, $\a^{{u+v_k}}(x)\in \Sigma(x,u)$ for all  large  enough $k$.

We now   derive an upper bound of the left-hand side of \eqref{eq:F2_contradiction}. For a.e.~$x\in \Om$, we have
\begin{align*}
F_2({u}&+{v}_k)(x)-F_2({u})(x)-\big(\p^*F_2({u}+{v}_k){v}_k\big)(x)\\
&\le ({u}+{v}_k)(x,\a^{{u+v_k}}(x))-u(x,\a^{{u+v_k}}(x))-v_k(x,\a^{{u+v_k}}(x))= 0,\\
F_2({u}&+{v}_k)(x)-F_2({u})(x)-\big(\p^*F_2({u}+{v}_k){v}_k\big)(x)\\
&\ge ({u}+{v}_k)(x,{\a}^{u}(x))-u(x,{\a}^{u}(x))-v_k(x,\a^{{u+v_k}}(x))
= {v}_k(x,{\a}^{u}(x))-v_k(x,\a^{{u+v_k}}(x)),
\end{align*}
from any $\a^u(x)\in  \Sigma(x,u)$.
Thus, for each $k\in \N$, we have for a.e.~$x\in \Om$ that,
\begin{align*}
|F_2({u}+{v}_k)(x)-F_2({u})(x)-\big(\p^*F_2({u}+{v}_k){v}_k\big)(x)|\le \phi_k(x)\coloneqq \inf_{\a^u\in  \Sigma(x,u)}| {v}_k(x,{\a}^{u})-v_k(x,\a^{{u+v_k}}(x))|,
\end{align*}
where, by applying Theorem \ref{thm:measurable_selection} twice, we can see that both the set-valued mapping $x\rightrightarrows \Sigma(x,u)$ and the function $\phi_k$ are measurable. 

We then introduce  the  set $\Om_k=\{x\in \Om\mid \a^{{u+v_k}}(x)\not\in \Sigma(x,u)\}$ for each $k\in \N$. The measurability of the set-valued mapping $x\rightrightarrows \Sigma(x,u)$ implies the associated distance function $\rho(x,\a)\coloneqq \textrm{dist}(\a,\Sigma(x,u))$ is a Carath\'{e}odory function (see \cite[Theorem 18.5]{aliprantis1999}), which subsequently leads to the measurability of $\Om_k$ for all $k$. Hence we can deduce   that 
\begin{align*}
&\|F_2({u}+{v}_k)-F_2({u})-\p^*F_2({u}+{v}_k){v}_k\|^2_{L^2(\Om)}\le \int_{\Om_k}  \inf_{\a^u\in  \Sigma(x,u)} |{v}_k(x,{\a}^{u})-v_k(x,\a^{{u+v_k}}(x))|^2\,dx\\
&\le2\int_{\Om_k} \sum_{\a\in \bA}|{v}_k(x,\a)|^2\,dx\le 
2 \sum_{\a\in \bA}|\Om_k|^{(p-2)/p}   \|v_k(\cdot,\a)\|_{L^p(\Om)}^2,
\end{align*}
which leads to the following estimate:
$$
\|F_2({u}+{v}_k)-F_2({u})-\p^*F_2({u}+{v}_k){v}_k\|_{L^2(\Om)}/\|v_k\|_{p,\bA}\le \sqrt{2}|\Om_k|^{(p-2)/(2p)} \to 0, \q \textnormal{as $k\to \infty$},   
$$
where we have used  the bounded convergence theorem and the fact that for a.e.~$x\in \Om$, $1_{\Om_k}(x)=0$ for all large enough $k$. This contradicts to the hypothesis \eqref{eq:F2_contradiction}, and hence finishes our proof.
\end{proof}

Now we are ready to conclude the semismoothness of the HJBI operator.
Note that the argument in \cite{smears2014}  does not apply directly  to the HJBI operator, due to the nonconvexity of the Hamiltonian $G$ defined as in \eqref{eq:G}.

\begin{Theorem}\l{thm:semismooth_F}
Suppose (H.\ref{assum:D}) holds, and let $F:H^2(\Om)\to L^2(\Om)$ be the HJBI operator defined as in \eqref{eq:hjb}.
Then 
$F$ is semismooth in $H^2(\Om)$, with a generalized differential $\p^*F:H^2(\Om)\to \cL(H^2(\Om), L^2(\Om))$ defined  as follows: for any $u\in H^2(\Om)$, 
\bb\l{eq:p_F}
\p^*F(u)\coloneqq -a^{ij}(\cdot)\p_{ij}+b^{i}(\cdot,\a(\cdot),\b^u(\cdot,\a(\cdot)))\p_i +c(\cdot,\a(\cdot),\b^u(\cdot,\a(\cdot))),
\ee
where $\b^u:\Om\t \bA\to \bB$ is any jointly measurable function satisfying \eqref{eq:b^u}, and $\a:\Om\to \bA$  is any  measurable function such that 
\bb\l{eq:measurable_a}
\a(x)\in  \arg \max_{\a\in \A}\bigg[\min_{\b\in \bB}\left(b^i(x,\a,\b)\p_i u(x)+c(x,\a,\b)u(x)-f(x,\a,\b)\right)\bigg], \q \textnormal{for a.e.~$x\in \Om$.}
\ee
\end{Theorem}
\begin{proof}
Note that we can decompose the HJBI operator $F:H^2(\Om)\to L^2(\Om)$ into $F=F_0+F_2\circ F_1$, where $F_0:H^2(\Om)\to L^2(\Om)$ is the linear operator $u\mapsto -a^{ij}\p_{ij}u$, $F_1:H^2(\Om)\to (L^p(\Om))^{|\bA|}$ is the HJB operator defined in Proposition \ref{prop:F1}, $F_2:(L^p(\Om))^{|\bA|}\to L^2(\Om)$ is the pointwise maximum operator defined in Proposition \ref{prop:F2}, and $p$ is a constant satisfying  $p> 2$ if $n\le 2$, and $p\in (2,2n/(n-2))$ if $n>2$.

 Proposition \ref{prop:F1} shows that $F_1$ is Lipschitz continuous and semismooth with respect to the generalized differential $\p^*F_1$ defined by \eqref{eq:p_F1}, while Proposition \ref{prop:F2} 
shows that $F_2$ is  semismooth with respect to the uniformly bounded generalized differential $\p^*F_2$ defined by \eqref{eq:p_F2}. Hence, we know the composed operator $F_2\circ F_1$ is  semismooth with respect to the composition of the generalized differentials (see \cite[Proposition 3.8]{ulbrich2011}), i.e., $\p^*(F_2\circ F_1)(u)=\p^*F_2(F_1(u))\circ \p^*F_1(u)$ for all $u\in H^2(\Om)$. Consequently, by using the fact that $F_0$ is  Fr\'{e}chet differentiable with the  derivative  $-a^{ij}\p_{ij}\in \cL(H^2(\Om),L^2(\Om))$,  we can conclude from Propositions \ref{prop:F1} and \ref{prop:F2} that   $F:H^2(\Om)\to L^2(\Om)$ is  semismooth on $H^2(\Om)$, and  that \eqref{eq:p_F} is a desired generalized differential of $F$ at $u$.
\end{proof}

Note that  the above  characterization of the generalized differential of the HJBI operator
involves a jointly measurable function $\b^u:\Om\t \bA\to \bB$,
satisfying \eqref{eq:b^u} for all $(x,\a)\in \Om\t \bA$.
We now present a technical lemma, which allows us to  view the control law $\b^k$ in \eqref{eq:b_k} as 
such a feedback control on $x\in \Om$ and $\a\in \bA$.

\begin{Lemma}\l{lemma:b_ext}
Suppose (H.\ref{assum:D}) holds. Let $h,\{h_i\}_{i=1}^n:\Om\to \R$, $\a^h:\Om\to \bA$ be given measurable functions,    and $\b^h:\Om\to \bB$ be a measurable function such that  for all $x\in \Om$,
\bb\l{eq:b_h}
{\b}^h(x)\in  \arg\min_{\b\in \bB}\left(b^i(x,\a^h(x),\b)h_i(x)+c(x,\a^h(x),\b)h(x)-f(x,\a^h(x),\b)\right).
\ee
%
%
Then there exists a jointly measurable function $\tilde{\b}^h:\Om\t \bA\to \bB$ such that $\b^h(x)=\tilde{\b}^h(x,\a^h(x))$ for all $x\in \Om$, and it holds for all $x\in \Om$  and  $\a\in \bA$ that 
\bb\l{eq:b_exd}
\tilde{\b}^h(x,\a)\in  \arg\min_{\b\in \bB}\left(b^i(x,\a,\b)h_i(x)+c(x,\a,\b)h(x)-f(x,\a,\b)\right).
\ee
%
\end{Lemma}
\begin{proof}
Let $\b^h:\Om\to \bB$ be a given measurable function satisfying \eqref{eq:b_h} for all $x\in \Om$
(see Remark \ref{rmk:cara} and Theorem \ref{thm:measurable_selection} for the existence of such a measurable function).
As shown in the proof of Proposition \ref{prop:F1}, there exists a jointly measurable function $\bar{\b}:\Om\t \bA\to \bB$ 
satisfying  the  property   \eqref{eq:b_exd}
for all $(x,\a)\in \Om\t \bA$. Now suppose that   $\bA=\{\a_i\}_{i=1}^{ |\bA|}$ with $|\bA|<\infty$ (see (H.\ref{assum:D})), we shall define the   function
$\tilde{\b}^h(x,\a):\Om\t \bA\to \bB$, such that for all $(x,\a)\in \Om\t \bA$,
$$
\tilde{\b}^h(x,\a)=
\begin{cases}
\b^h(x), & (x,\a)\in \cC\coloneqq \bigcup_{i=1}^{ |\bA|}
\big(\{x\in \Om\mid \a^h(x)=\a_i\}\t\{\a_i\}\big),\\
\bar{\b}(x,\a), &\textnormal{otherwise.}
\end{cases}
$$
The measurability of $\a^h$ and the finiteness of $\bA$ imply that  
the set $\cC$ is measurable in the product $\sigma$-algebra on $ \Om\t \bA$,
which along with the joint measurability of $\bar{\b}$
leads to the joint measurability of the function $\tilde{\b}^h$.

For any given $x\in \Om$, we have $(x,\a^h(x))\in \{y\in \Om\mid \a^h(y)=\a^h(x)\}\t\{\a^h(x)\}$,
from which we can deduce from 
the definition of $\tilde{\b}^h$  that $\tilde{\b}^h(x,\a^h(x))=\b^h(x)$ for all $x\in \Om$. 
Finally, for any given  $\a_i\in \bA$, we shall  verify \eqref{eq:b_exd}  for all $x\in \Om$ and $\a=\a_i$. 
Let $x\in \Om$ be fixed. If  $\a^h(x)=\a_i$, 
then
the fact that  $(x,\a_i)\in \cC$ and the definition of $\tilde{\b}^h$ imply that $\tilde{\b}^h(x,\a_i)=\b^h(x)$,
which along with  \eqref{eq:b_h} and $\a^h(x)=\a_i$
shows that \eqref{eq:b_exd} holds for the point $(x,\a_i)$. On the other hand, if 
$\a^h(x)\not=\a_i$, 
then 
$(x,\a_i)\not \in \cC$ and
$\tilde{\b}^h(x,\a_i)=\bar{\b}(x,\a_i)$ satisfies the condition \eqref{eq:b_exd} due to the selection of $\bar{\b}$.
\end{proof}

As a direct consequence of the above extension result, we now present 
an equivalent characterization of the generalized differential of the HJBI operator.

\begin{Corollary}\l{cor:semismooth_F}
Suppose (H.\ref{assum:D}) holds, and let $F:H^2(\Om)\to L^2(\Om)$ be the HJBI operator defined as in \eqref{eq:hjb}.
Then 
$F$ is semismooth in $H^2(\Om)$, with a generalized differential $\p^*F:H^2(\Om)\to \cL(H^2(\Om), L^2(\Om))$ defined  as follows: for any $u\in H^2(\Om)$, 
\bb\l{eq:p_F_alternative}
\p^*F(u)\coloneqq -a^{ij}(\cdot)\p_{ij}+b^{i}(\cdot,\a^u(\cdot),\b^u(\cdot))\p_i +c(\cdot,\a^u(\cdot),\b^u(\cdot)),
\ee
where  $\a^u:\Om\to \bA$
and $\b^u:\Om\to \bB$ are any  measurable functions satisfying for all $x\in \Om$ that
\begin{align}\l{eq:measurable_ab}
\begin{split}
\a^u(x)&\in  \arg \max_{\a\in \A}\bigg[\min_{\b\in \bB}\left(b^i(x,\a,\b)\p_i u(x)+c(x,\a,\b)u(x)-f(x,\a,\b)\right)\bigg], \\
\b^u(x)&\in \argmin_{\b\in \bB}\left(b^i(x,\a^u(x),\b)\p_i u(x)+c(x,\a^u(x),\b)u(x)-f(x,\a^u(x),\b)\right).
\end{split}
\end{align}
\end{Corollary}
\begin{proof}
Let $u\in H^2(\Om)$, and let $\a^u$ and $\b^u$ be given measurable functions satisfying \eqref{eq:measurable_ab}
(see Remark \ref{rmk:cara} and Theorem \ref{thm:measurable_selection} for the existence of such measurable functions). 
Then by using Lemma \ref{lemma:b_ext}, we know there exists a jointly measurable function 
$\tilde{\b}^u:\Om\t \bA\to \bB$ such that  
 $\tilde{\b}^u$ satisfies \eqref{eq:b^u} for  all $(x,\a)\in \Om\t \bA$,
  and 
 $\tilde{\b}^u(x,\a^u(x))=\b^u(x)$ for all $x\in \Om$.
Hence 
we  see the linear operator defined in \eqref{eq:p_F_alternative} is equal to the following operator
$$
-a^{ij}(\cdot)\p_{ij}+b^{i}(\cdot,\a^u(\cdot),\tilde{\b}^u(\cdot,\a^u(\cdot)))\p_i +c(\cdot,\a^u(\cdot),\tilde{\b}^u(\cdot,\a^u(\cdot)))\in \cL(H^2(\Om), L^2(\Om)),
$$
which is a generalized differential of the HJBI operator $F$ at $u$ due to Theorem \ref{thm:semismooth_F}.
\end{proof}
\color{black}

The above characterization of the generalized differential of the HJBI operator enables us to demonstrate the superlinear convergence of Algorithm \ref{alg:pi} by reformulating it as a semismooth Newton method for an operator equation.


\begin{Theorem}\l{thm:pi_superlinear}
Suppose  (H.\ref{assum:D}) holds and let  $u^*\in H^2(\Om)$ be a strong solution to the Dirichlet problem  \eqref{eq:D}. Then there exists a neighborhood $\cN$ of $u^*$, such that for all $u^0\in \cN$,   Algorithm  \ref{alg:pi} either terminates with $u^k=u^*$ for some $k\in \N$, or generates a sequence  $\{u^k\}_{k\in \N}$ that converges $q$-superlinearly to $u^*$ in  $H^2(\Om)$, i.e.,  $\lim_{k\to \infty}\|u^{k+1}-u^*\|_{H^2(\Om)}/\|u^{k}-u^*\|_{H^2(\Om)}=0$.
\end{Theorem}
\begin{proof}
Note that the Dirichlet problem  \eqref{eq:D}  can be written as an operator equation $\tilde{F}(u)=0$ with the following operator
$$
\tilde{F}:u\in H^2(\Om)\to  (F(u),\tau u-g)\in L^2(\Om)\t H^{3/2}(\p\Om),
$$
where $F$ is the HJBI operator defined as in \eqref{eq:hjb}, and $\tau:H^2(\Om)\to H^{3/2}(\p\Om)$ is the trace operator. Moreover, one can directly check that given an iterate $u^k\in H^2(\Om)$, $k\ge 0$, the next iterate $u^{k+1}$ solves
the following  Dirichlet problem:
$$
L_k(u^{k+1}-u^k)=-F(u^k),\q \textnormal{in $\Om$};\q \tau (u^{k+1}-u^k)=-(\tau u^k-g), \q \textnormal{on $\p\Om$.}
$$
with the differential operator  $L_k\in \p^*F(u^k)$ 
defined as in \eqref{eq:p_F_alternative}.
%
Since $F:H^2(\Om)\to L^2(\Om)$ is $ \p^*F$-semismooth (see Corollary \ref{cor:semismooth_F}) and $\tau\in \cL(H^2(\Om), H^{3/2}(\p\Om))$, we can conclude that  Algorithm  \ref{alg:pi} is in fact a semismooth Newton method for solving the operator equation $\tilde{F}(u)=0$.

Note that the boundedness of coefficients and the classical   theory of elliptic regularity (see Theorem \ref{thm:D_regularity}) imply that under condition (H.\ref{assum:D}), there exists a constant $C>0$, such that 
for any  $u\in H^2(\Om)$ and any  $L\in \p^*F(u)$, the inverse operator $(L,\tau)^{-1}:L^2(\Om)\t H^{3/2}(\p\Om) \to H^2(\Om)$ is well-defined, and the operator norm $\|(L,\tau)^{-1}\|$ is bounded by $C$, uniformly in $u\in H^2(\Om)$.
Hence one can conclude from \cite[Theorem 3.13]{ulbrich2011} (see also Theorem \ref{thm:quasi-Newton}) that the iterates $\{u^k\}_{k\in \N}$ converges superlinearly to $u^*$ in a neighborhood $\cN$ of $u^*$.
%
%
%
%
\end{proof}

The next theorem  strengthens Theorem \ref{thm:pi_superlinear}, and establishes a novel global convergence result of 
Algorithm \ref{alg:pi} applied to  the Dirichlet problem  \eqref{eq:D}, which subsequently provides a constructive proof for the existence of solutions to \eqref{eq:D}. The following additional condition is essential for our proof of the global convergence of  Algorithm \ref{alg:pi}:

\begin{Assumption}\l{assum:discount}
Let the function  $c$ in (H.\ref{assum:D}) be given as:  $c(x,\a,\b)=\bar{c}(x,\a,\b)+\ul{c}_0$, for all $(x,\a,\b)\in \Om\t \bA\t\bB$, where   $\ul{c}_0$ is a sufficiently large constant, depending on $\Om$, $\{a^{ij}\}_{i,j=1}^n$, $\{b^i\}_{i=1}^n$ and  $\|\bar{c}\|_{L^\infty(\Om\t \bA\t\bB)}$.
\end{Assumption}

In practice, (H.\ref{assum:discount}) can be satisfied  
if  \eqref{eq:D} arises from an infinite-horizon stochastic game with a large  discount factor (see e.g.~\cite{buckdahn2016}),
or 
if \eqref{eq:D} stems from an implicit (time-)discretization of parabolic HJBI equations  with a small time stepsize.

\color{black}

\begin{Theorem}\l{thm:pi_global}
Suppose  (H.\ref{assum:D}) and \textcolor{blue}{(H.\ref{assum:discount})} hold, then the Dirichlet problem  \eqref{eq:D} admits a unique strong solution $u^*\in H^2(\Om)$. Moreover, for any initial guess $u^0\in H^2(\Om)$, Algorithm  \ref{alg:pi} either terminates with $u^k=u^*$ for some $k\in \N$,  or generates a sequence  $\{u^k\}_{k\in \N}$ that converges $q$-superlinearly to $u^*$ in  $H^2(\Om)$, i.e., $\lim_{k\to \infty}\|u^{k+1}-u^*\|_{H^2(\Om)}/\|u^{k}-u^*\|_{H^2(\Om)}=0$.
\end{Theorem}

\begin{proof}
If Algorithm \ref{alg:pi} terminates  in iteration $k$, we have $F(u^k)=L_ku^k-f_k=0$ and $\tau u^k=g$, from which, we obtain from  the uniqueness of strong solutions to  \eqref{eq:D} (Proposition \ref{prop:wp_d}) that $u^k=u^*$ is the strong solution to the Dirichlet problem  \eqref{eq:D}. Hence we shall assume without loss of generality that Algorithm  \ref{alg:pi}  runs infinitely.

We now establish the global convergence of Algorithm \ref{alg:pi} by first showing the iterates $\{u^k\}_{k\in \N}$ form a Cauchy sequence in $H^2(\Om)$. 
For each $k\ge 0$, we deduce from 
\eqref{eq:maxmin}
and (H.\ref{assum:discount}) that
$\tau u^{k+1}=g$ on $\p\Om$ and 
\begin{align}\l{eq:u_{k+1}}
&-a^{ij}\p_{ij}u^{k+1}+b^{i}_{k}\p_i u^{k+1}+c_{k}u^{k+1}-f_{k}=-a^{ij}\p_{ij}u^{k+1}+b^{i}_{k}\p_i u^{k+1}+(\bar{c}_{k}+\ul{c}_0)u^{k+1}-f_{k}\nb\\
&\q= -a^{ij}\p_{ij}u^{k+1}+\ul{c}_0u^{k+1}+b^{i}_{k}\p_i (u^{k+1}-u^{k})+\bar{c}_{k}(u^{k+1}-u^{k})+\bar{G}(\cdot,u^k,\nabla u^k)=0,
\end{align}
for a.e.~$x\in \Om$, where the function $\bar{c}_{k}(x)\coloneqq \bar{c}(x,\a^{k}(x),\b^{k}(x))$ for all $x\in \Om$, and the modified Hamiltonian is defined as:
\bb\l{eq:G_m}
\bar{G}(x,u,\nabla u)=\max_{\a\in \bA}\min_{\b\in\bB}\big(b^i(x,\a,\b)\p_i u(x)+\bar{c}(x,\a,\b)u(x)-f(x,\a,\b)\big).
\ee
Hence, by taking the difference of  equations corresponding to the indices $k-1$ and $k$, one can obtain that 
\begin{align}
-&a^{ij}\p_{ij}(u^{k+1}-u^k)+\ul{c}_{0}(u^{k+1}-u^{k})=-b^{i}_{k}\p_i (u^{k+1}-u^{k})-\bar{c}_{k}(u^{k+1}-u^{k})\nb\\
&+b^{i}_{k-1}\p_i (u^{k}-u^{k-1})+\bar{c}_{k-1}(u^{k}-u^{k-1})-[\bar{G}(\cdot,u^k,\nabla u^k)-\bar{G}(\cdot,u^{k-1},\nabla u^{k-1})],
\l{eq:difference}
\end{align}
for $x\in\Om$, and $\tau(u^{k+1}-u^k)=0$ on $\p\Om$. 

It has been proved in Theorem 9.14 of \cite{gilbarg1983} that, there exist positive constants $C$ and $\gamma_{0}$, depending only on $\{a^{ij}\}_{i,j=1}^n$ and $\Om$, such that it holds for all $u\in H^2(\Om)$ with $\tau u=0$, and for all $\gamma\ge\gamma_{0}$ that
$$
\|u\|_{H^2(\Om)}\le C\|-a^{ij}\p_{ij}u+\gamma u\|_{L^2(\Om)},
$$
which, together with the identity that $\gamma u=(-a^{ij}\p_{ij}u +\gamma u)+a^{ij}\p_{ij}u$ and the boundedness of $\{a^{ij}\}_{ij}$, implies that the same estimate also holds for $\|u\|_{H^2(\Om)}+\gamma \|u\|_{L^2(\Om)}$: 
$$
\|u\|_{H^2(\Om)}+\gamma\|u\|_{L^2(\Om)}\le C\|-a^{ij}\p_{ij}u+\gamma u\|_{L^2(\Om)}.
$$

Thus, by assuming $\ul{c}_0\ge \gamma_0$ and using  the boundedness of the coefficients, we can deduce from \eqref{eq:difference} that
\begin{align}
\|u^{k+1}-u^k&\|_{H^2(\Om)}+\ul{c}_0\|u^{k+1}-u^k\|_{L^2(\Om)}\le C\bigg(\|-b^{i}_{k}\p_i (u^{k+1}-u^{k})-\bar{c}_{k}(u^{k+1}-u^{k})\nb\\
&+b^{i}_{k-1}\p_i (u^{k}-u^{k-1})+\bar{c}_{k-1}(u^{k}-u^{k-1})-[\bar{G}(\cdot,u^k,\nabla u^k)-\bar{G}(\cdot,u^{k-1},\nabla u^{k-1})]\|_{L^2(\Om)}\bigg)\nb\\
&\le C\big(\|u^{k+1}-u^{k}\|_{H^1(\Om)}+\|u^{k}-u^{k-1}\|_{H^1(\Om)}\big), 
\end{align}
for some constant $C$ independent of $\ul{c}_0$ and the index $k$.

Now we apply the following interpolation inequality (see \cite[Theorem 7.28]{gilbarg1983}): there exists a constant $C$, such that for all $u\in H^2(\Om)$ and $\eps>0$, we have $\|u\|_{H^1(\Om)}\le \eps\|u\|_{H^2(\Om)}+C\eps^{-1}\|u\|_{L^2(\Om)}$.
Hence, for any given $\eps_1\in(0,1),\eps_2>0$, we have 
\begin{align*}
(1-\eps_1)\|u^{k+1}-u^k\|_{H^2(\Om)}+\ul{c}_0\|u^{k+1}-u^k\|_{L^2(\Om)}&\le  \eps_2\|u^{k}-u^{k-1}\|_{H^2(\Om)}+C\eps_1^{-1}\|u^{k+1}-u^{k}\|_{L^2(\Om)}\\
&\q +C\eps_2^{-1}\|u^{k}-u^{k-1}\|_{L^2(\Om)}.
\end{align*}
Then, by taking $\eps_1\in(0,1)$, $\eps_2<1-\eps_1$, and assuming that $\ul{c}_0$ satisfies  $(\ul{c}_0-C/\eps_1)/(1-\eps_1)\ge C/\eps_2^2$, we can obtain  for $c'=C/\eps_2^2$ that 
$$
\|u^{k+1}-u^k\|_{H^2(\Om)}+c'\|u^{k+1}-u^k\|_{L^2(\Om)}\le \f{\eps_2}{1-\eps_1}\big(\|u^{k}-u^{k-1}\|_{H^2(\Om)}+c'\|u^{k}-u^{k-1}\|_{L^2(\Om)}\big),
$$
which implies that $\{u^k\}_{k\in \N}$ is a Cauchy sequence with the norm $\|\cdot\|_{c'}\coloneqq\|\cdot\|_{H^2(\Om)}+c'\|\cdot\|_{L^2(\Om)}$. 

Since $\|\cdot\|_{c'}$ is equivalent to $\|\cdot\|_{H^2(\Om)}$ on $H^2(\Om)$, we can deduce that $\{u^k\}_{k\in \N}$ converges to some $\bar{u}$ in $H^2(\Om)$.
By passing $k\to \infty$ in \eqref{eq:u_{k+1}} and using Proposition \ref{prop:wp_d}, we can deduce that $\bar{u}=u^*$ is the unique strong solution  of \eqref{eq:D}. 
Finally, for a sufficiently large $K_0\in \N$, we can conclude the superlinear convergence of $\{u^k\}_{k\ge K_0}$ from Theorem \ref{thm:pi_superlinear}. 
\end{proof}

\color{black}

We end this section with an important remark that, if one of  the sets $\bA$ and  $\bB$ is a singleton, 
and $a^{ij}\in C^{0,1}(\bar{\Om})$ for all $i,j$,
then  Algorithm \ref{alg:pi} applied to the Dirichlet problem  \eqref{eq:D} is in fact  monotonically  convergent with an arbitrary initial guess. Suppose, for instance, that $\bA$ is a singleton, 
then for each $k\in \N\cup\{0\}$, we have  that
 $$
 0=L_ku^{k+1}-f_k\ge F(u^{k+1})=-L_{k+1}(u^{k+2}-u^{k+1}),\q \textnormal{for a.e.~$x\in \Om$}.
 $$
Hence we can deduce that  $w^{k+1}\coloneqq u^{k+1}-u^{k+2}$ is a weak subsolution to $L_{k+1} w=0$, i.e.,
$$
\int_\Om \bigg[a^{ij}\p_jw^{k+1}\p_i \phi+\big((\p_i a^{ij}+b_{k+1}^i)\p_i w^{k+1}+c_{k+1}w^{k+1}\big)\phi\bigg]\,dx\le 0, \q \fa \phi\ge 0, \; \phi\in C_0^1(\Om).
$$
Thus, the  weak maximal principle (see \cite[Theorem 1.3.7]{garroni2002}) and the fact that  $u^{k+1}-u^{k+2}=0$ a.e.~$x\in \p\Om$ (with respect to the surface measure),   leads to the estimate $\esssup_{\Om} u^{k+1}-u^{k+2}\le 0$, which consequently implies that $u^k\le u^{k+1}$ for all $k\in \N$ and a.e.~$x\in \Om$.

\color{black}

\section{Inexact policy iteration for HJBI Dirichlet problems}\l{sec:dpi_D}

Note that  at each policy iteration, Algorithm \ref{alg:pi} requires us to  obtain an exact solution to  a linear Dirichlet boundary value problem,  which is generally infeasible. Moreover, an accurate computation of numerical solutions to  linear Dirichlet boundary value problems  could be expensive, especially in a high-dimensional setting. In this section, we shall propose an inexact policy iteration algorithm for \eqref{eq:D}, where we compute an approximate solution to \eqref{eq:linear} by solving an optimization problem over a family of trial functions, while maintaining the   superlinear convergence of policy iteration. 

We shall make the following assumption on the trial functions of the optimization problem. 
\begin{Assumption}\l{assum:cF}
The  collections of trial functions $\{\cF_M\}_{M\in \N}$ satisfies the following properties: $\cF_M\subset \cF_{M+1}$ for all $M\in \N$, and $\cF=\{\cF_M\}_{M\in \N}$ is dense in $H^{2}(\Om)$.
\end{Assumption} 

It is clear that  (H.\ref{assum:cF}) is satisfied by any reasonable $H^2$-conforming finite element spaces (see e.g.~\cite{brenner1994}) and high-order polynomial spaces or   kernel-function spaces  used in global spectral methods (see e.g.~\cite{beard1997, beard1998, cheung2018, kalise2018, kalise2019}).  
We now demonstrate that (H.\ref{assum:cF}) can also be easily satisfied by the sets of multi-layer feedforward neural networks,
which provides  effective trial functions for high-dimensional problems.  Let us first recall the definition of a feedforward neural network. 

\begin{Definition}[Artificial neural networks]\l{def:DNN}
Let $L,N_0,N_1,\ldots, N_L\in \N$ be given constants,  and $\varrho:\R\to \R$ be a given  function. For each $l=1,\ldots, L$, let $T_l:\R^{N_{l-1}}\to \R^{N_l}$ be an affine function given as $T_l(x)=W_lx+b_l$ for some $W_l\in \R^{N_l\t N_{l-1}}$ and $b_l\in \R^{N_{l}}$. A function $F:\R^{N_0}\to \R^{N_L}$ defined as 
$$
F(x)=T_L\circ (\varrho\circ T_{L-1})\circ\cdots (\varrho\circ T_{1}), \q x\in \R^{N_0},
$$
is called a feedforward neural network.  Here the activation function $\varrho$ is applied componentwise. We shall refer the quantity $L$ as the depth of $F$, $N_1,\ldots N_{L-1}$ as the dimensions of the hidden layers, and $N_0,N_L$  as the dimensions of the input and  output layers, respectively. We also refer to the number of entries of $\{W_l,b_l\}_{l=1}^N$ as the complexity of $F$.
\end{Definition}

Let $\{L^{(M)}\}_{M\in \N}$, $\{N^{(M)}_1\}_{M\in \N},\ldots, \{N^{(M)}_{L^{(M)}-1}\}_{M\in \N}$ be some nondecreasing sequences of natural numbers, we define for each $M$ the set  $\cF_M$ of all neural networks with  depth  $L^{(M)}$, input dimension being equal to $n$,  output dimension being equal to 1, and dimensions of hidden layers being equal to $\{N^{(M)}_1,\ldots, N^{(M)}_{L^{(M)}-1}\}_{M\in \N}$. It is clear that if $L^{(M)}\equiv L$ for all $M\in \N$, then we have $\cF_M\subset \cF_{M+1}$. The following proposition is proved in \cite[Corollary 3.8]{hornik1990}, which shows neural networks with  one hidden layer are dense in  $H^2(\Om)$.

\begin{Proposition}
Let $\Om\subset \R^n$ be an open bounded starshaped domain, and $\varrho\in C^2(\R)$ satisfying $0<|D^l\varrho|_{L^1(\Om)}<\infty$ for all $l=1,2$. Then the family of all neural networks with depth $L=2$ is dense in $H^2(\Om)$.

\end{Proposition}

Now we discuss how to approximate  the strong solutions of Dirichlet  problems by 
reformulating the equations into  optimization  problems over trial functions. The idea is similar to least squares finite-element methods (see e.g.~\cite{bochev2009}), and has been employed previously to develop numerical methods for PDEs based on neural networks (see e.g.~\cite{lagaris1998,berg2018,sirignano2018}). However, compared to \cite{berg2018,lagaris1998}, we do not impose additional constraints on the trial functions by requiring that the networks  exactly agree with the  boundary conditions, due to the lack of theoretical support that the constrained  neural networks are still dense in the solution space. Moreover, to ensure the convergence of solutions in the $H^2(\Om)$-norm, we include the $H^{3/2}(\p\Om)$-norm of the boundary data in the cost function, instead of the  $L^2(\p\Om)$-norm used in \cite{sirignano2018} (see Remark \ref{rmk:H3/2} for more details).

 For each $k\in \N\cup \{0\}$, let   $u^{k+1}\in H^2(\Om)$ be the unique  solution to the Dirichlet problem \eqref{eq:linear}: 
$$
L_k u-f_k=0, \; \textnormal{in $\Om$}; \q \tau u=g, \; \textnormal{on $\p\Om$},
$$
where $L_k$ and $f_k$ denote the linear elliptic operator and the source term in \eqref{eq:linear}, respectively. 
For each $M\in \N$, we shall consider the following optimization problems:
\bb\l{eq:J}
J_{k,M}\coloneqq \inf_{u\in \cF_M}J_k(u), \q \textnormal{with $J_k(u)=\|L_ku-f_k\|^2_{L^2(\Om)}+\|\tau u-g\|^2_{H^{3/2}(\p\Om)}$.}
\ee
The following result shows that the cost function $J_k$ provides  a computable indicator of the  error.

\begin{Proposition}\l{prop:J_conv}
Suppose (H.\ref{assum:D}) and (H.\ref{assum:cF}) hold. For each $k\in \N\cup\{0\}$ and $M\in \N$,  let $u^{k+1}\in H^2(\Om)$ be the unique solution to \eqref{eq:linear}, and $J_k, J_{k,M}$ be defined as in \eqref{eq:J}. Then there exist positive constants $C_1$ and $C_2$, such that we have for each $u\in H^2(\Om)$ and $k\in \N\cup\{0\}$ that
$$
C_1J_k(u)\le \|u-u^{k+1}\|^2_{H^2(\Om)}\le C_2 J_k(u).
$$
Consequently, it holds for each $k\in\N\cup\{0\}$ that $\lim_{M\to \infty} J_{k,M}=0$.
\end{Proposition}

\begin{proof}
Let  $k\in \N\cup\{0\}$ and $u\in H^2(\Om)$. The  definition of $J_k(u)$ implies that $L_ku-f_k=f^e\in L^2(\Om)$, $\tau u-g= g^e\in H^{3/2}(\p\Om)$ and $J(u)= \|f^e\|^2_{L^2(\Om)}+\|g^e\|^2_{H^{3/2}(\p\Om)}$. Then, by using the  assumption  that $u^{k+1}$ solves \eqref{eq:linear}, we deduce that the residual term satisfies the following Dirichlet problem:
$$
L_k(u-u^{k+1})=f^e, \q \textnormal{in $\Om$}; \q \tau(u-u^{k+1})=g^e, \q \textnormal{on $\p\Om$}.
$$
Hence the boundedness of coefficients and the regularity theory of elliptic operators (see Theorem \ref{thm:D_regularity})  lead to the estimate that 
$$C_1(\|f^e\|^2_{L^2(\Om)}+\|g^e\|^2_{H^{3/2}(\p\Om)})\le \|u-u^{k+1}\|^2_{H^2(\Om)}\le C_2(\|f^e\|^2_{L^2(\Om)}+\|g^e\|^2_{H^{3/2}(\p\Om)}),$$
where the constants $C_1,C_2>0$ depend only on the  $L^\infty(\Om)$-norms of $a^{ij}, b^i_k,c_k,f_k$, which are independent of $k$. The above estimate, together with the facts that $\{\cF_M\}_{M\in \N}$ is dense in $H^2(\Om)$ and $\cF_M\subset \cF_{M+1}$, leads to the desired result that $\lim_{M\to \infty} J_{k,M}=0$.
\end{proof}

We  now  present the inexact policy iteration algorithm for the HJBI problem \eqref{eq:D}, where at each policy iteration, we solve the linear Dirichlet problem  within  a given accuracy. 
%

\begin{algorithm}[H]
\caption{Inexact policy iteration algorithm for Dirichlet problems.}
\label{alg:dpi_d}
\vspace{-3mm}
\begin{enumerate}
\item Choose a family of trial functions $\cF=\{\cF_M\}_{M\in\N}\subset H^2(\Om)$, an initial guess $u^0$ in $\cF$, a sequence $\{\eta_k\}_{k\in \N\cup\{0\}}$ of positive scalars, and set $k=0$.
\item Given the  iterate  $u^{k}$, update the  control laws   $\a^{k}$ and 
$\b^{k}$ by  \eqref{eq:a_k} and \eqref{eq:b_k}, respectively.
\item Find $u^{k+1}\in \cF$ such that\footnotemark
\bb\l{eq:J_error}
J_k(u^{k+1})=\|L_ku^{k+1}\!-f_k\|^2_{L^2(\Om)}\!+\|\tau u^{k+1}-g\|^2_{H^{3/2}(\p\Om)}\le \eta_{k+1}\min(\|u^{k+1}-u^k\|^2_{H^2(\Om)},\eta_0),
\ee
where  $L_k$ and $f_k$ denote the linear  operator and the source term in \eqref{eq:linear}, respectively.
\item If $\|u^{k+1}-u^k\|_{H^2(\Om)}=0$, then terminate with outputs $u^{k+1},\a^k$ and $\b^k$,
otherwise increment $k$ by one and go to step 2.
\end{enumerate}
\vspace{-3mm}
\end{algorithm}
\footnotetext{With a slight abuse of notation, we denote by $u^{k+1}$ the inexact solution to the Dirichlet problem \eqref{eq:linear}.}

\begin{Remark}
In practice, the evaluation of the squared residuals  $J_k$ in \eqref{eq:J_error} depends on the choice of trial functions.
For trial functions with linear architecture, e.g.\ if
$\{\cF_M\}_{M\in \N}$ are finite element spaces,  high-order polynomial spaces, and  kernel-function  spaces 
 (see \cite{brenner1994,cheung2018, kalise2018,kalise2019}),  one may evaluate the norms 
 by applying high-order quadrature rules to the basis functions involved.

 For trial functions with nonlinear architecture, such as feedforward neural networks, 
 we can replace the integrations in  $J_k$
 by 
  the empirical mean 
over suitable collocation points in $\Om$ and on  $\p\Om$,
 such as pseudorandom points or quasi-Monte Carlo points  
 (see Section \ref{sec:numerical}; see also \cite{lagaris1998,berg2018,sirignano2018}).
 In particular,
 due to the existence of local coordinate charts of the boundaries,
 we can 
evaluate the double integral in the definition of the $H^{3/2}(\p\Om)$-norm (see \eqref{eq:H1/23/2})
by first generating points in $\R^{2(n-1)}$
and then mapping the samples  onto $\p\Om\t \p\Om$.
The resulting empirical least-squares problem for the $k+1$-th policy iteration step (cf.~\eqref{eq:J}) 
can then be solved by stochastic gradient descent  (SGD) algorithms; see Section \ref{sec:numerical}.
We remark that, 
instead of pre-generating all the collocation points in advance,
one can  perform  gradient descent based
on a sequence of mini-batches of points generated at each SGD iteration.
This is particularly useful in higher dimensions, where many collocation points may be needed to cover the boundary,
and using mini-batches avoids having to evaluate functions at all collocation points in each iteration.
\end{Remark}

\color{black}

It is well-known (see e.g.~\cite{ulbrich2011,dontchev2012}) that the  residual term $\|u^{k+1}-u^k\|_{H^2(\Om)}$ is crucial for the superlinear convergence of inexact Newton methods.
%
This next theorem  establishes the global superlinear convergence of Algorithm \ref{alg:dpi_d}. 
\color{black}

\begin{Theorem}\l{thm:dpi_global}
Suppose  (H.\ref{assum:D}), \textcolor{blue}{(H.\ref{assum:discount})} and (H.\ref{assum:cF}) hold,  and $\lim_{k\to \infty}\eta_k=0$ in Algorithm  \ref{alg:dpi_d}. Let $u^*\in H^2(\Om)$ be the solution to the Dirichlet problem  \eqref{eq:D}. Then for any initial guess $u^0\in \cF$, Algorithm  \ref{alg:dpi_d} either terminates with $u^k=u^*$ for some $k\in \N$, or generates a sequence  $\{u^k\}_{k\in \N}$ that converges $q$-superlinearly to $u^*$ in  $H^2(\Om)$, i.e., $\lim_{k\to \infty}\|u^{k+1}-u^*\|_{H^2(\Om)}/\|u^{k}-u^*\|_{H^2(\Om)}=0$. Consequently, we have $\lim_{k\to\infty}(u^k,\p_iu^k,\p_{ij}u^k)(x)= (u^*,\p_iu^*,\p_{ij}u^*)(x)$  for a.e.~$x\in \Om$, and for all $i,j=1,\ldots, n$.
\end{Theorem}

\begin{proof}
Let $u^0\in \cF$ be an arbitrary initial guess. 
We first show that Algorithm \ref{alg:dpi_d} is always well-defined.
For each $k\in \N\cup\{0\}$, if $u^k\in \cF$ is the strong solution to \eqref{eq:linear}, then we can choose $u^{k+1}=u^k$, which satisfies \eqref{eq:J_error} and terminates the algorithm. If $u^k$ does not solve \eqref{eq:linear}, the fact that  $\cF$ is dense in $H^2(\Om)$ enables us to find  $u^{k+1}\in \cF$  satisfying the criterion \eqref{eq:J_error}.

Moreover, one can clearly see from \eqref{eq:J_error}  that  if Algorithm \ref{alg:dpi_d}  terminates at  iteration $k$, then $u^k$ is the exact solution to the Dirichlet problem  \eqref{eq:D}. Hence in the sequel we shall assume without loss of generality that Algorithm  \ref{alg:dpi_d}  runs infinitely, i.e., $\|u^{k+1}-u^k\|_{H^2(\Om)}>0$ and $u^k\not =u^*$ for all $k\in \N\cup\{0\}$.
 
 We next show  the iterates  converge to $u^*$ in $H^2(\Om)$ by following similar arguments as those for Theorem \ref{thm:pi_global}. For each $k\ge 0$,  we can deduce from \eqref{eq:J_error} that there exists  $f^e_k\in L^2(\Om)$ and $g^e_k\in H^{3/2}(\p\Om)$ such that 
\bb\l{eq:linear_inexact}
L_ku^{k+1}-f_k=f^e_k, \q \textnormal{in $\Om$}; \q \tau u^{k+1}-g=g^e_k, \q \textnormal{on $\p\Om$},
\ee
and $J_k(u^{k+1})=\|f^e_k\|^2_{L^2(\Om)}+\|g^e_k\|^2_{H^{3/2}(\Om)}\le \eta_{k+1}(\|u^{k+1}-u^k\|^2_{H^2(\Om)})$ with $\lim_{k\to\infty}\eta_{k}= 0$. 
Then, by taking the difference between \eqref{eq:linear_inexact} and \eqref{eq:D}, we obtain that
\begin{align*}
-a^{ij}\p_{ij}(u^{k+1}-u^*)+\ul{c}_{0}(u^{k+1}-u^{*})&=-b^{i}_{k}\p_i (u^{k+1}-u^{k})-\bar{c}_{k}(u^{k+1}-u^{k})\nb\\
&-[\bar{G}(\cdot,u^k,\nabla u^k)-\bar{G}(\cdot,u^{*},\nabla u^{*})]+f_k^e, \q \textnormal{in $\Om$,}
\end{align*}
and $\tau (u^{k+1}-u^*)=g_k^e$ on $\p\Om$, where $\bar{G}$ is the modified Hamiltonian defined as in \eqref{eq:G_m}.
Then, by proceeding along the lines of  Theorem \ref{thm:pi_global}, we can obtain 
a positive constant $C$, independent of $\ul{c}_0$ and the  index $k$, such that 
\begin{align*}
&\|u^{k+1}-u^*\|_{H^2(\Om)}+\ul{c}_0\|u^{k+1}-u^*\|_{L^2(\Om)}\\
&\le  C\big(\|u^{k+1}-u^{*}\|_{H^1(\Om)}+\|u^{k+1}-u^{k}\|_{H^1(\Om)}+\|u^{k}-u^*\|_{H^1(\Om)}\big)+o(\|u^{k+1}-u^k\|_{H^2(\Om)})\\
&\le  C\big(\|u^{k+1}-u^*\|_{H^1(\Om)}+\|u^{k}-u^*\|_{H^1(\Om)}\big)+o(\|u^{k+1}-u^*\|_{H^2(\Om)}+\|u^{k}-u^*\|_{H^2(\Om)})
\end{align*}
as $k\to\infty$,  where the additional high-order terms are due to the residuals $f^e_k$ and $g^e_k$.
Then, by using the interpolation inequality and assuming $\ul{c}_0$ is sufficiently large, we can deduce that $\{u^k\}_{k\in \N}$ converge linearly to $u^*$ in  $H^2(\Om)$.

\color{black}

We  then  reformulate Algorithm \ref{alg:dpi_d} into a quasi-Newton method for the operator equation $\tilde{F}(u)=0$, with the operator $\tilde{F}: u\in H^2(\Om)\to (F(u),\tau u-g)\in L^2(\Om)\t H^{3/2}(\p\Om)$  defined  in the proof of Theorem \ref{thm:pi_superlinear}. 
Let   $H^2(\Om)^*$ denote the  strong dual space of $H^2(\Om)$, and  $\la \cdot, \cdot \ra$ denote the dual product on $H^2(\Om)^*\t H^2(\Om)$. For each $k\in \N\cup \{0\}$, by using the fact that $\|u^{k+1}-u^k\|_{H^2(\Om)}>0$, we can   choose  $w_k\in H^2(\Om)^*$ satisfying  $\la w_k, u^{k+1}-u^k\ra=-1$, and   introduce the following linear operators $\delta L_k\in \cL(H^2(\Om),  L^2(\Om))$ and $\delta \tau_k\in \cL(H^2(\Om),  H^{3/2}(\p\Om))$:
$$
\delta L_k: v\in H^2(\Om)\to \la w_k, v\ra f^e_k\in L^2(\Om)\q \textnormal{and}\q \delta \tau_k: v\in H^2(\Om)\to \la w_k, v\ra g^e_k\in H^{3/2}(\p\Om).
$$
Then, we can apply the identity $F(u^k)=L_ku^k-f_k$ and rewrite \eqref{eq:linear_inexact} as:
$$
(L_k+\delta L_k)(u^{k+1}-u^k)=-F(u^k), \q \textnormal{in $\Om$}; \q (\tau+\delta\tau_k) (u^{k+1}-u^k)=-(\tau u^k-g), \q \textnormal{on $\p\Om$},
$$
with $(L_k,\tau)\in \p^*\tilde{F}(u^k)$ as shown in Theorem \ref{thm:pi_superlinear}. Hence one can clearly see that \eqref{eq:linear_inexact} is precisely a Newton step with a perturbed operator for the  equation $\tilde{F}(u)=0$. 

Now we are ready to establish the superlinear convergence of $\{u^k\}_{k\in \N}$. For notational simplicity, in the subsequent analysis we shall denote by $Z\coloneqq L^2(\Om)\t  H^{3/2}(\p\Om)$ the Banach space with the usual product norm $\|z\|_Z\coloneqq \|z_1\|_{L^2(\Om)}+\|z_2\|_{H^{3/2}(\p\Om)}$ for each $z=(z_1,z_2)\in Z$.
By using the semismoothess of $\tilde{F}: H^2(\Om)\to  Z$ (see Theorem \ref{thm:pi_superlinear}) and  the strong convergence of $\{u^k\}_{k\in\N}$ in $H^2(\Om)$,  we can directly infer from Theorem \ref{thm:quasi-Newton} that it remains to show that there exists a open neighborhood $V$ of $u^*$, and a constant $L>0$, such that 
\begin{align}\l{eq:criterion1}
\|v - u^*\|_{H^2(\Om)}/L\le \|\tilde{F}(v) - \tilde{F}(u^*)\|_Z \le L\|v - u^*\|_{H^2(\Om)}, \; \fa v\in V,
\end{align}
and also 
\bb\l{eq:criterion2}
\lim_{k\to\infty}\|(\delta L_k s^k,\delta\tau_ks^k) \|_Z/\|s^k\|_{H^2(\Om)}=0, \q \textnormal{with $s^k=u^{k+1}-u^k$ for all $k\in \N$.}
\ee
The criterion \eqref{eq:J_error} and the definitions of $\delta L_k$ and $\delta \tau_k$ imply that \eqref{eq:criterion2} holds:
$$
\bigg(\f{\|(\delta L_k s^k,\delta\tau_ks^k) \|_Z}{\|s^k\|_{H^2(\Om)}}\bigg)^2= \bigg(\f{\|f^e_k\|_{L^2(\Om)}+\|g^e_k\|_{H^{3/2}(\p\Om)}}{\|s^k\|_{H^2(\Om)}}\bigg)^2\le \f{2J_k(u^{k+1})}{\|s^k\|^2_{H^2(\Om)}}\le 2\eta_0\eta_{k+1}\to 0,
$$
as $k\to \infty$. Moreover, the boundedness of the coefficients $a^{ij},b^i,c,f$ shows that $\tilde{F}$ is Lipschitz continuous.
Finally, the characterization of the generalized differential of $\tilde{F}$ in Theorem \ref{thm:pi_superlinear} and the regularity theory of elliptic operators (see Theorem \ref{thm:D_regularity}) show that for each $v\in H^2(\Om)$, we can choose an  invertible operator $M_v=(L_v,\tau)\in \p^*\tilde{F}(v)$ such that $\|M_v^{-1}\|_{\cL(Z,H^2(\Om))}\le C<\infty$, uniformly in $v$. Thus we can conclude from the semismoothness of $\tilde{F}$ at $u^*$ that
\begin{align*}
\|\tilde{F}(v)-\tilde{F}(u^*)\|_Z&=\|M_v(v-u^*)+o(\|v-u^*\|_{H^2(\Om)})\|_Z\ge \|M_v(v-u^*)\|_Z-o(\|v-u^*\|_{H^2(\Om)})\\
&\ge \|v-u^*\|_{H^2(\Om)}/C-o(\|v-u^*\|_{H^2(\Om)})\ge  \|v-u^*\|_{H^2(\Om)}/(2C),
\end{align*}
for all $v$ in some neighborhood $V$ of $u^*$, which completes our proof for $q$-superlinear convergence of $\{u^k\}_{k\in\N}$.

Finally, we establish the pointwise convergence of $\{u^k\}_{k=1}^\infty$ and their derivatives. For any given $\gamma\in (0,1)$, the superlinear convergence of $\{u^k\}_{k=1}^\infty$ implies that there exists a constant $C>0$, depending on $\gamma$, such that $\|u^{k}-u^*\|^2_{H^2(\Om)}\le C\gamma^{2k}$ for all $k\in \N$. Taking the summation over the index $k$, we have
$$
\int_\Om\sum_{k=1}^\infty \bigg(|u^{k}-u^*|^2+\sum_{i,j=1}^n[|\p_iu^{k}-\p_iu^*|^2+|\p_{ij}u^{k}-\p_{ij}u^*|^2]\bigg)\,dx=\sum_{k=1}^\infty\|u^{k}-u^*\|^2_{H^2(\Om)}\le \f{C\gamma^{2}}{1-\gamma^2}<\infty,
$$
where we  used the monotone convergence theorem in the first equality. Thus, we have 
$$\sum_{k=1}^\infty \bigg(|u^{k}-u^*|^2+\sum_{i,j=1}^n[|\p_iu^{k}-\p_iu^*|^2+|\p_{ij}u^{k}-\p_{ij}u^*|^2]\bigg)(x)<\infty, \q \textnormal{for a.e.~$x\in \Om$},$$
which leads us to the pointwise convergence of $u^k$ and its partial derivatives with respect to $k$.
\end{proof}

\begin{Remark}\l{rmk:H3/2}
We reiterate that  merely including the $L^2(\p\Om)$-norm of the boundary data in the cost functional \eqref{eq:J_error} in general cannot  guarantee  the  convergence of  the derivatives of the numerical solutions $\{u^k\}_{k=1}^\infty$, which can be seen from the following simple example. Let $\{g_k\}_{k=1}^\infty\subseteq H^{3/2}(\p\Om)$ be a sequence such that $g_k\to 0$ in $L^2(\p\Om)$ but not in $H^{1/2}(\p\Om)$, and for each $k\in \N$, let $h^k\in H^2(\Om)$ be the strong solution to $-\Delta h^k=0$ in $\Om$ and $h^k=g_k$ on $\p\Om$. 

The fact that $g_k\not\to 0$ in $H^{1/2}(\p\Om)$ implies that $h^k\not\to 0$ in $H^{1}(\Om)$ as $k\to\infty$. We now show $\lim_{k\to \infty}h^k=0$ in $L^2(\Om)$. Let $w\in H^2(\Om)$ be the solution to $-\Delta w=h^k$ in $\Om$ and $w=0$ on $\p\Om$, we can deduce from the integration by parts and the \textit{a priori} estimate $\|w\|_{H^2(\Om)}\le C\|h^k\|_{L^2(\Om)}$ that 
\begin{align*}
\|h^k\|^2_{L^2(\Om)}&=\int_\Om (-\Delta w)\, h^k\,dx=\int_\Om  w\, (-\Delta h^k)\,dx+\int_{\p\Om} w\p_n h^k\,d\sigma-\int_{\p\Om} h^k \p_n w\,d\sigma\\
&\le C\|g_k\|_{L^2(\p\Om)} \|w\|_{H^2(\Om)}\le C\|g_k\|_{L^2(\p\Om)} \|h^k\|_{L^2(\Om)},
\end{align*}
which shows that $\|h^k\|_{L^2(\Om)}\le C\|g_k\|_{L^2(\p\Om)}\to 0$ as $k\to \infty$.
Now let $\cF$ be a given family of trial functions, which is dense in $H^2(\Om)$. One can find $\{u^k\}_{k=1}^\infty\subseteq \cF$ satisfying $\lim_{k\to\infty}\|u^k-h^k\|_{H^2(\Om)}=0$, and consequently $u^k\not\to 0$ in $H^1(\Om)$ as $k\to \infty$. However, we have 
$$
\|-\Delta u^{k}\|^2_{L^2(\Om)}+\|u^k\|^2_{L^2(\p\Om)}= \|-\Delta (u^k-h^k)\|^2_{L^2(\Om)}+\|u^k-h^k+g_k\|^2_{L^2(\p\Om)} 
\to 0, \q \textnormal{as $k\to\infty$.}
$$
Similarly, one can construct functions $\{u^k\} _{k=1}^\infty\subseteq \cF$ such that $\|-\Delta u^{k}\|^2_{L^2(\Om)}+\|u^k\|^2_{H^{1/2}(\p\Om)}\to 0$ as $k\to \infty$, but $\{u^k\}_{k=1}^\infty$ does not converge to $0$ in $H^2(\Om)$.
\end{Remark}

We end this section with a convergent approximation of the optimal control strategies based on the iterates $\{ u^k\}_{k=1}^\infty$ generated by Algorithm \ref{alg:dpi_d}. For any given $u\in H^2(\Om)$, we denote by $\bA^u(x)$ and $\bB^u(x,\a)$ the set of optimal  control strategies for all $\a\in \bA$ and for a.e.~$x\in \Om$, such that
\begin{align*}
\bB^{u}(x,\a) &= \arg\min_{\b\in \bB}\big(b^i(x,\a,\b)\p_i u(x)+c(x,\a,\b)u(x)-f(x,\a,\b)\big), \\
\bA^{u}(x)&= \arg \max_{\a\in \A}\min_{\b\in \bB}\big(b^i(x,\a,\b)\p_i u(x)+c(x,\a,\b)u(x)-f(x,\a,\b)\big).
\end{align*}
As an important  consequence of the superlinear convergence of Algorithm \ref{alg:dpi_d}, we now conclude that the feedback control strategies $\{\a^k\}_{k=1}^\infty$ and $\{\b^k\}_{k=1}^\infty$ generated by Algorithm \ref{alg:dpi_d} are convergent to the optimal control strategies.

\begin{Corollary}\l{cor:dpi_ctrl}
Suppose the assumptions of Theorem \ref{thm:dpi_global} hold, and let $u^*\in H^2(\Om)$ be the solution to the Dirichlet problem  \eqref{eq:D}.
Assume further that 
there exist functions $\a^*:\Om\to \bA$ and $\b^*:\Om\to \bB$ such that
$\bA^{u^*}(x)=\{\a^*(x)\}$
and 
$\bB^{u^*}(x,\a^*(x))=\{\b^*(x)\}$ for a.e.~$x\in \Om$.
Then the measurable functions   $\a^k:\Om\to \bA$ and $\b^k:\Om\to \bB$, $k\in\N$,  generated by Algorithm \ref{alg:dpi_d} converge   to 
the optimal feedback control $(\a^*,\b^*)$ pointwise almost everywhere. 

\end{Corollary}
\begin{proof}
Let $\ell$ and $h$   be the Carath\'{e}odory  functions defined by \eqref{eq:l} and \eqref{eq:h}, respectively, and we consider the following set-valued mappings: 
\begin{align}\l{eq:Gamma}
\begin{split}
\Gamma_1:(x,\u)\in \Om\t \R^{n+1}&\rightrightarrows\Gamma_1(x,\u)\coloneqq\argmax_{\a\in \bA}h(x,\u,\a),\\
 \Gamma_2:(x,\u,\a)\in \Om\t \R^{n+1}\t \bA &\rightrightarrows \Gamma_2(x,\u,\a)\coloneqq \argmin_{\b\in \bB}\ell(x,\u,\a,\b).
 \end{split}
\end{align}
Theorem \ref{thm:uhc} implies that the set-valued mappings $\Gamma_1(x,\cdot):\R^{n+1}\rightrightarrows \bA$ and $\Gamma_2(x,\cdot,\cdot):  \R^{n+1}\t \bA \rightrightarrows\bB$ are  upper hemicontinuous. 
Then the result follows directly from  the pointwise convergence of $(u^k,\nabla u^k)_{k=1}^\infty$ in Theorem \ref{thm:dpi_global}, 
and the fact that  $\bA^{u^*}(x)=\{\a^*(x)\}$  and 
$\bB^{u^*}(x,\a^*(x))=\{\b^*(x)\}$ are singleton  for a.e.~$x\in \Om$.
\end{proof}
\color{black}

\begin{Remark}\l{rmk:dpi_ctrl}
If we assume in addition that $\bA\subset X_A$ and $\bB\subset Y_B$ for some Banach spaces $X_A$ and $Y_B$, then by using the compactness of $\bA$ and $\bB$ (see (H.\ref{assum:D})), we can conclude from the dominated convergence theorem that 
$\a^k\to \a^*$ in $L^p(\Om; X_A)$ and $\b^k\to \b^*$ in 
 $L^p(\Om; Y_B)$, for any $p\in [1,\infty)$.
\end{Remark}
 
\section{Inexact policy iteration for HJBI oblique derivative problems} \l{sec:oblique}
In this section, we extend  the algorithms introduced in previous sections  to more general boundary value problems. In particular, we shall propose a neural network based policy iteration algorithm with global $H^2$-superlinear convergence for solving HJBI  boundary value problems with oblique   derivative boundary conditions.
Similar arguments can be adapted to design superlinear convergent schemes for  mixed boundary value problems with both Dirichlet and oblique   derivative boundary conditions.

We consider the following HJBI oblique   derivative problem:
\begin{subequations}\l{eq:N}
\begin{align}
F(u)&\coloneqq -a^{ij}(x)\p_{ij}u+G(x,u,\nabla u)=0, \q\textnormal{a.e.~$x\in \Om$}, \l{eq:hjb_n}\\
Bu&\coloneqq  \gamma^i\tau(\p_iu)+\gamma^0 \tau u -g=0, \q \textnormal{on $\p\Om$.} \l{eq:bc_n}
\end{align}
\end{subequations}
where \eqref{eq:hjb_n} is the HJBI equation given in \eqref{eq:hjb}, and \eqref{eq:bc_n} is an oblique boundary condition.
Note that 
the boundary condition $Bu$ on $\p\Om$ involves 
 the traces of $u$ and its first partial derivatives, which exist  almost everywhere on $\p\Om$ (with respect to the surface measure). 
 
 The following conditions are imposed on the coefficients of \eqref{eq:N}:
 \begin{Assumption}\l{assum:N}
 Assume $\Om$, $\bA$, $\bB$, $(a^{ij})_{i,j=1}^n$, $(b^i)_{i=1}^n,c,f$ satisfy the same conditions as those in (H.\ref{assum:D}). 
Let $g\in H^{1/2}(\p\Om)$,  $\{\gamma^i\}_{i=0}^n\subseteq C^{0,1}(\p\Om)$, $\gamma^0\ge 0$ on $\p\Om$, and assume there exists a constant $\mu>0$, such that $c\ge \mu$ on $\Om\t\bA\t \bB$, and  $\sum_{i=1}^n\gamma^i\nu_i\ge \mu$ on $\p\Om$,  where $\{\nu_i\}_{i=1}^n$ are the components of the unit outer normal vector field on $\p\Om$.
\end{Assumption}
 
The next proposition establishes the well-posedness of the oblique   derivative problem. 
 
 \begin{Proposition}\l{prop:wp_n}
Suppose (H.\ref{assum:N}) holds. Then the oblique   derivative problem  \eqref{eq:D} admits a unique strong solution $u^*\in H^2(\Om)$. 
\end{Proposition}
\begin{proof}
We shall establish the uniqueness of strong solutions to \eqref{eq:N} in this proof, and then explicitly  construct the solution in Theorem \ref{thm:dpi_global_n} with the help of  policy iteration; see also Theorem \ref{thm:pi_global}. Suppose that  $u,v\in H^2(\Om)$ are two strong solutions to \eqref{eq:N}, then we can see $w=u-v$ is a strong solution to the following linear oblique   derivative problem:
\bb\l{eq:n_u-v}
-a^{ij}\p_{ij}w+\tilde{b}^{i}\p_i w+\tilde{c}w=0, \q \textnormal{a.e. in $\Om$}; \q \gamma^i\tau(\p_iw)+\gamma^0\tau w=0, \q \textnormal{on $\p\Om$},
\ee
where  $\tilde{b}^i$ is defined as in Proposition \ref{prop:wp_d}, and 
\begin{align*}
\tilde{c}(x)&=\begin{cases}
\f{G(x,u,\nabla u)-G(x,v,\nabla u)}{(u-v)(x)}, & \textnormal{on $\{x\in \Om\mid (u-v)(x)\not =0\}$,}\\
\mu, &\textnormal{otherwise.}
\end{cases}
\end{align*}
By following the same arguments as the proof of Proposition \ref{prop:wp_d}, we can  show that  $\tilde{b}^i, \tilde{c}\in L^\infty(\Om)$, and $\tilde{c}\ge \mu>0$ a.e.~in $\Om$,  which, along with Theorem \ref{thm:N_regularity}, implies that $w^*=0$ is the unique strong solution to \eqref{eq:n_u-v}, and consequently $u=v$ in $H^2(\Om)$.  
\end{proof} 
 
Now we present the neural network based  policy iteration algorithm for solving the oblique   derivative problem and establish its rate of convergence. 
%

 \begin{algorithm}[H]
\caption{Inexact policy iteration algorithm for oblique   derivative problems.}
\label{alg:dpi_n}
\vspace{-3mm}
\bn
  \setcounter{enumi}{0}
\item Choose a family of trial functions $\cF=\{\cF_M\}_{M\in\N}\subset H^2(\Om)$, an initial guess $u^0$ in $\cF$, a sequence $\{\eta_k\}_{k\in \N\cup\{0\}}$ of positive scalars, and set $k=0$.
\item Given the iterate  $u^{k}$, update the  control laws  $\a^{k}$ and $\b^{k}$ 
 by  \eqref{eq:a_k} and \eqref{eq:b_k}, respectively.
\item Find $u^{k+1}\in \cF$ such that
\bb\l{eq:J_error_n}
J_k(u^{k+1})=\|L_ku^{k+1}-f_k\|^2_{L^2(\Om)}+\|Bu^{k+1}\|^2_{H^{1/2}(\p\Om)}\le \eta_{k+1}\min(\|u^{k+1}-u^k\|^2_{H^2(\Om)},\eta_0),
\ee
where  $L_k$, $f_k$, and $B$  denote the linear  operator in \eqref{eq:linear},  the source term  in \eqref{eq:linear} and   the boundary operator in \eqref{eq:bc_n}, respectively.
\item  If $\|u^{k+1}-u^k\|_{H^2(\Om)}=0$, then terminate with outputs $u^{k+1},\a^k$ and $\b^k$, otherwise increment $k$ by one and go to step 2.
\en
\vspace{-3mm}
\end{algorithm}

Note that  the $H^{1/2}(\p\Om)$-norm of the boundary term is included in the cost function $J_k$, instead of the $H^{3/2}(\p\Om)$-norm as in Algorithm \ref{alg:dpi_d}.  It is straightforward to see that Algorithm \ref{alg:dpi_n} is well-defined under (H.\ref{assum:cF}) and (H.\ref{assum:N}). In fact, for each $k\in \N\cup\{0\}$, given the iterate $u^{k}\in \cF\subset H^2(\Om)$, 
Corollary  \ref{cor:semismooth_F} shows 
that one can select measurable control laws $(\a^k,\b^k)$ such that the following linear oblique boundary value problem has measurable coefficients:
 $$
L_k u-f_k=0, \; \textnormal{in $\Om$}; \q B u=0, \; \textnormal{on $\p\Om$},
$$
and hence admits a unique strong solution $\bar{u}^k$ in $H^2(\Om)$ (see  Theorem \ref{thm:N_regularity}). If $u^k=\bar{u}^k$, then $u^k$ solve the HJBI oblique   derivative problem \eqref{eq:N}, and we can select $u^{k+1}=u^k$ and terminate the algorithm. Otherwise, the facts that  $J_k(u^{k+1})\le C\|\bar{u}^k-u^{k+1}\|^2_{H^2(\Om)}$ and $\cF$ is dense in $H^2(\Om)$ allows us to choose   $u^{k+1}\in \cF$ sufficiently closed to $\bar{u}$ such that  the criterion \eqref{eq:J_error_n} is satisfied, and proceed to the next iteration.

 The following result  is analogue to Theorem \ref{thm:dpi_global}, and shows the global superlinear convergence of Algorithm \ref{alg:dpi_n} for solving the oblique   derivative problem \eqref{eq:N}. The proof follows precisely the lines given in Theorem \ref{thm:dpi_global},  hence we shall only present the main steps in Appendix \ref{sec:proof_N} for the reader's convenience. The convergence of feedback control laws can be concluded similarly to Corollary \ref{cor:dpi_ctrl} and Remark \ref{rmk:dpi_ctrl}.
 
 \begin{Theorem}\l{thm:dpi_global_n}
Suppose  \textcolor{blue}{(H.\ref{assum:discount})}, (H.\ref{assum:cF}) and (H.\ref{assum:N}) hold,   and $\lim_{k\to \infty}\eta_k=0$ in Algorithm  \ref{alg:dpi_n}. Let $u^*\in H^2(\Om)$ be the solution to the oblique   derivative problem  \eqref{eq:N}. Then for any initial guess $u^0\in \cF$, Algorithm  \ref{alg:dpi_n} either terminates with $u^k=u^*$ for some $k\in \N$, or generates a sequence  $\{u^k\}_{k\in \N}$ that converges $q$-superlinearly to $u^*$ in  $H^2(\Om)$, i.e., $\lim_{k\to \infty}\|u^{k+1}-u^*\|_{H^2(\Om)}/\|u^{k}-u^*\|_{H^2(\Om)}=0$. Consequently, we have $\lim_{k\to\infty}(u^k,\p_iu^k,\p_{ij}u^k)(x)= (u^*,\p_iu^*,\p_{ij}u^*)(x)$  for a.e.~$x\in \Om$, and for all $i,j=1,\ldots, n$.

\end{Theorem}

 \section{Numerical experiments: Zermelo's Navigation Problem}\l{sec:numerical}
 In this section, we illustrate the theoretical findings and demonstrate the effectiveness of the  schemes through numerical experiments. We present a two-dimensional convection-dominated HJBI Dirichlet boundary value problem in an annulus, which is related to stochastic minimum time problems.  
 

In particular, we consider the stochastic Zermelo navigation problem (see e.g.~\cite{mohajerin2016}), which is a time-optimal control problem where the objective is to find the optimal trajectories of a ship/aircraft navigating  a region of strong winds, modelled by a random vector field. Given a bounded open set $\Om\subset \R^n$ and an adaptive control strategy $\{\a_t\}_{t\ge 0}$ taking values in $\bA$, we assume the dynamics   $X^{x,\a}$ of the ship is governed by the following controlled dynamics:
$$
dX_t=b(X_t,\a_t) \, dt+\sigma \,dW_t,\q t\in [0,\infty); \q X_0=x\in \Om,
$$
where the drift coefficient $b:\Om\t \bA\to \R^n$ is the sum of the velocity of the wind and the relative velocity of the ship, the nondegenerate diffusion coefficient $\sigma:\Om\to \R^{n\t n}$ describes a random perturbation of the velocity field of the wind, and $W$ is an $n$-dimensional Brownian motion defined on a probability space $(\tilde{\Om}, \{\cF_t\}_{t\ge 0}, \bP)$.  

The aim of the controller is to minimize the  expected exit time of the region $\Om$, taking model ambiguity into account in the spirit of \cite{royer2006}.
More generally, we consider the following value function:
\bb\l{eq:value}
u(x)\coloneqq \inf_{\a\in \cA}\cE\bigg[\int_0^{\tau_{x,\a}}f(X^{x,\a}_t)\,dt+g(X^{x,\a}_{\tau_{x,\a}})\bigg]=\inf_{\a\in \cA}\sup_{\Q\in \cM}\ex_{\Q}\bigg[\int_0^{\tau_{x,\a}}f(X^{x,\a}_t)\,dt+g(X^{x,\a}_{\tau_{x,\a}})\bigg]
\ee
over all admissible choices of $\a\in \cA$, where $\tau_{x,\a}\coloneqq \inf\{t\ge 0\mid X^{x,\a}_t\not\in {\Om}\}$ denotes the first exit time of the controlled dynamics $X^{x,\a}$, 
the functions $f$ and $g$ denote the running cost and  the exit cost, respectively, which indicate the desired destinations,  and $\cM$ is a family of absolutely continuous probability measures with respect to $\bP$ with density $M_t=\exp\big(\int_0^t \b_t\,dW_t-\f{1}{2}\int_0^t \b_t^2\,dt\big)$, where $\{\b_t\}_{t\ge 0}$ is a  predictable process satisfying $\|\b_t\|_\infty=\max_{i}|\b_{t,i}| \le \kappa$ for all $t$
and a given parameter $\kappa\ge  0$. In other words, we would like to minimize a functional of the trajectory up to the exit time under the worst-case scenario, with uncertainty arising from the unknown law of the random perturbation. 

By using the dual representation of $\cE[\cdot]$ and the dynamic programming principle (see e.g.~\cite{royer2006,buckdahn2016}), we can characterize the value function $u$ as the unique viscosity solution to an HJBI Dirichlet boundary value problem of the form \eqref{eq:D}. Moreover, under  suitable assumptions, one can further show that $u$ is  the  strong (Sobolev) solution to this Dirichlet problem (see e.g.~\cite{krylov2018}).

For our numerical experiments, we assume that the  domain $\Om$ is an annulus, i.e., $\Om=\{(x,y)\in \R^2\mid r^2< x^2+y^2< R^2\}$,  the  wind blows along  the positive  $x$-axis with a magnitude $v_c$:
$$
v_c(x,y)=1-a\sin\bigg(\pi\f{x^2+y^2-r^2}{R^2-r^2}\bigg),\q \textnormal{for some constant $a\in [0,1)$,}
$$
which decreases in terms of the distance from the bank, and the random perturbation of the wind is given by the constant diffusion coefficient $\sigma=\textrm{diag}(\sigma_x,\sigma_y)$. We also assume that the ship moves with a constant velocity $v_s$, and the captain can control  the boat's direction  instantaneously, which leads to the following dynamics of the boat in the region:
$$
\begin{pmatrix} dX^{x,\a}_t \\ dY^{x,\a}_t\end{pmatrix}
=\begin{pmatrix} v_c(X^{x,\a}_t,Y^{x,\a}_t)+v_s\cos(\a_t)\\ v_s\sin(\a_t)\end{pmatrix}dt+\begin{pmatrix} \sigma_x & 0\\ 0 &\sigma_y\end{pmatrix}dW_t,\q t\ge 0; \q \begin{pmatrix} X^{x,\a}_0 \\ Y^{x,\a}_0\end{pmatrix}=x,
$$
where $\a_t\in \bA=[0,2\pi]$ represents the angle (measured counter-clockwise) between the positive $x$-axis and the direction of the boat. 
Finally, we assume the exit cost $g\equiv 0$ on $\p B_r(0)$ and  $g\equiv 1$ on $\p  B_R(0)$,
which represents that the controller prefers to exit the domain through the inner boundary instead of the outer one (see Figure \ref{fig:zermelo}).
 Then the corresponding Dirichlet problem for the value function $u$ in \eqref{eq:value} is given by: $u\equiv 0$ on $\p B_r(0)$,  $u\equiv 1$ on $\p  B_R(0)$, and 
\begin{align}\l{eq:navigation_d}
F(u)&=-\f{1}{2}(\sigma_x^2u_{xx}+\sigma_y^2u_{yy})-
v_cu_x-v_s\inf_{\a\in \bA}\big[(\cos(\a),\sin(\a))^T\nabla u\big]-\sup_{\|\b\|_\infty\le \kappa}\big[\b^T(\sigma\nabla u)\big]-f
\nb\\
&=-\f{1}{2}(\sigma_x^2u_{xx}+\sigma_y^2u_{yy})-
v_cu_x+v_s\|\nabla u\|_{\ell^2}-\kappa \|\sigma\nabla u\|_{\ell^1}-f
=0,\q\textnormal{in $\Om$},
\end{align}
where $\|\cdot\|_{\ell^1}$ and $\|\cdot\|_{\ell^2}$ denote the $\ell^1$-norm and $\ell^2$-norm on $\R^2$, respectively.
The optimal feedback control laws can be further computed as
\begin{align}\l{eq:ctrl}
\a^*=\pi+\theta, \q \b^*=\kappa(\textrm{sgn}\big(\sigma_xu_x\big),\textrm{sgn}\big(\sigma_yu_y\big))^T,\q \textnormal{a.e.~in $\Om$,}
\end{align}
where  $u\in H^2(\Om)$ is the strong solution to \eqref{eq:navigation_d}, and $\theta\in (-\pi,\pi]$ is the angle between $\nabla u$ and the positive $x$ direction. Note that the equation \eqref{eq:navigation_d} is neither convex nor concave in $\nabla u$.

\bigskip
\begin{figurehere}
    \centering
    \includegraphics[width=0.31\columnwidth,height=5cm]{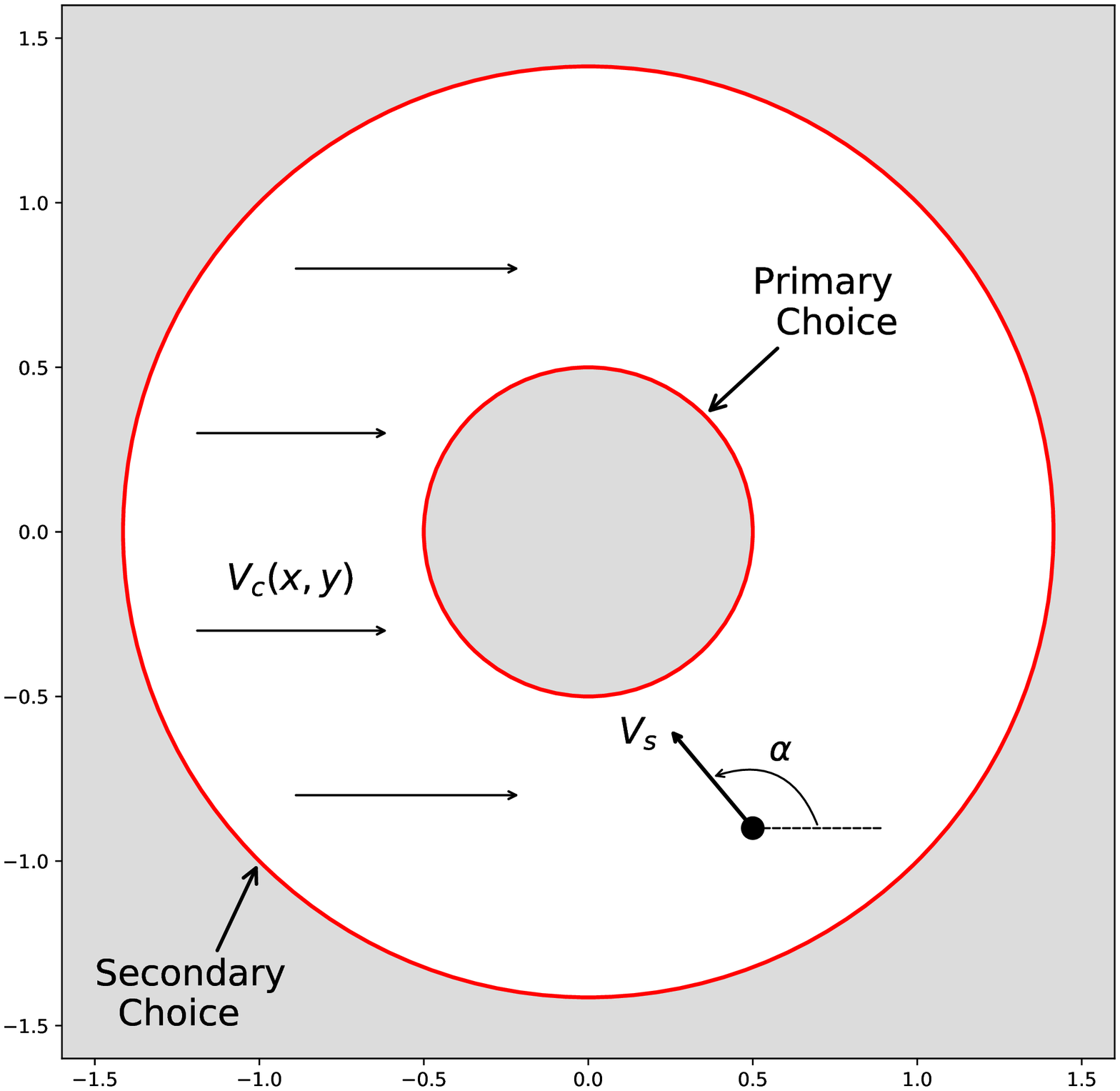}\hspace{1cm}
 \caption{Zermelo navigation problem in an annulus.}
    \label{fig:zermelo}
 \end{figurehere}
 \bigskip

\subsection{Implementation details}
In this section, we discuss the implementation details of Algorithm \ref{alg:dpi_d} for solving \eqref{eq:navigation_d}
with multi-layer neural networks (see Definition \ref{def:DNN}) as the trial functions.
We shall now introduce the architecture of the neural networks, the involved hyper-parameters, and various computational aspects  of the training process.

For simplicity, we shall adopt a fixed set of trial functions $\cF_M$ for all policy iterations, which contains   fully-connected networks  with the activation function  $\varrho(y)=\tanh(y)$,  the depth $L$, and  the dimension of each hidden layer $H$. The hyper-parameters $L$ and $H$ will be chosen depending on the complexity of the problem, which ensures that $\cF_M$ admits sufficient  flexibility to approximate the solutions within the desired accuracy. More complicated architectures of neural networks with shortcut connections can be adopted to further improve the performance of the algorithm (see e.g. \cite{e2018,sirignano2018}). 

We then proceed to discuss the computation of the  cost functional $J_k$ in \eqref{eq:J_error} for each policy iteration. It is well-known that  Sobolev norms of functions on sufficiently smooth boundaries can be explicitly computed  via local coordinate charts of the boundaries (see e.g.~\cite{grisvard1985}). In particular, due to the annulus shaped domain and the constant boundary conditions used in our experiment, we can express the cost functional $J_k$ as follows: for all $k\in \N$ and $u\in \cF_M$,
\begin{align}\l{eq:J_annu}
\begin{split}
J_k(u)=&\,\|L_ku-f_k\|^2_{L^2(\Om)}+\sum_{l=r,R}\bigg[\|u-g\|^2_{L^2(\p B_l(0))}+\gamma
\bigg(\int_{-\pi}^{\pi} |D_\theta(u\circ \Phi_l)|^2\,d\theta\\
&+\int_{(-\pi,\pi)^2}\f{|D_\theta(u\circ \Phi_l)(\theta_1)-D_\theta(u\circ \Phi_l)(\theta_2)|^2}{|\theta_1-\theta_2|^2}\,d\theta_1d\theta_2\bigg)\bigg], 
\end{split}
\end{align}
where we define the map $\Phi_l:\theta\in (-\pi,\pi)\to (l\cos(\theta),l\sin(\theta))\in \p B_l(0)$ for $l=r,R$.
Note that 
we introduce an extra weighting parameter $\gamma>0$ in \eqref{eq:J_annu}, which helps 
achieve the optimal balance between the residual of the PDE and the residuals of the boundary data. 
We set the  parameter $\gamma=0.1$ for  all the computations.

The  cost functional \eqref{eq:J_annu} is further approximated by an empirical cost via the collocation method (see \cite{lagaris1998,berg2018}), where we discretize  $\Om$ and 
$\Theta=(-\pi,\pi)^2$ 
by sets of collocation points $\Om_d=\{x_i\in\Om\mid 1\le i\le N_d\}$ and 
$\Theta_d=\{\theta=(\theta_{1,i},\theta_{2,i})\in \Theta\mid 1\le i\le N_b\}$, 
respectively, and write the discrete form of \eqref{eq:J_annu}  as follows: for all $k\in \N$ and $u\in \cF_M$,
\begin{align}\l{eq:J_d}
J_{k,d}(u)=&\,\f{|\Om|}{N_d}\sum_{x_i\in \Om_d}|L_ku(x_i)-f_k(x_i)|^2+
\sum_{l=r,R} \bigg[\f{|\p B_l(0)|}{N_b}\sum_{\theta\in \Theta_d}|(u-g)\circ \Phi_l(\theta_{1,i})|^2\\
&+\gamma\bigg(
\f{2\pi}{N_b}\sum_{\theta\in \Theta_d} |D_\theta(u\circ \Phi_l)|^2(\theta_{1,i})+
\f{(2\pi)^2}{N_b}\sum_{\theta\in \Theta_d}\f{|D_\theta(u\circ \Phi_l)(\theta_{1,i})-D_\theta(u\circ \Phi_l)(\theta_{2,i})|^2}{|\theta_{1,i}-\theta_{2,i}|^2}\bigg)\bigg],\nb
\end{align} 
 where $|\Om|=\pi(R^2-r^2)$, and $|\p B_l(0)|=2\pi l$ for $l=r,R$ are, respectively, the Lebesgue measures of the domain and boundaries. Note that the choice of the smooth activation function $\varrho(y)=\tanh(y)$ implies that every trial function $u\in \cF_M$ is smooth, hence all its derivatives are well-defined  at any given point. For simplicity, we take the same number of collocation points in the domain and on the boundaries, i.e., $N_d=N_b=N$.
 
It is clear that the choice of collocation points  is crucial for the accuracy and efficiency of the algorithm. 
Since the total number of points in a regular grid grows exponentially  with respect to the dimension, such a construction is infeasible for high-dimensional problems. Moreover, it is well-known that uniformly distributed pseudorandom points in high dimensions tend to cluster on hyperplanes 
and lead to a suboptimal distribution by relevant measures of uniformity
(see e.g.~\cite{cervellera2004,berg2018}).  Therefore, 
we shall generate collocation points by a quasi-Monte Carlo (QMC) method based on  low-discrepancy sequences. In particular, we first define 
points in  $[0,1]^2$ from the generalized Halton sequence (see \cite{faure2009}), and then  map those  points  into the annulus via the polar map $(x,y)\mapsto (l\cos(\psi),l\sin(\psi))$, where $l=\sqrt{(R^2-r^2)x+r^2}$ and $\psi=2\pi y$ for all $(x,y)\in [0,1]^2$. The above transformation preserves fractional area, which ensures that
a set of well-distributed points  on the square will map to a set of points spread evenly over the annulus.
We also use Halton points to approximate the (one-dimensional) boundary segments.
 
Now  we are ready to describe the training process, i.e., how to optimize \eqref{eq:J_d} over all trial functions in $\cF_M$. The optimization is performed by using the well-known Adam stochastic gradient descent (SGD) algorithm \cite{kingma2014} with a decaying learning rate schedule. At each SGD iteration, we randomly draw a mini-batch of points with size $B=25$ from the collection of collocation points, and perform gradient descent  based on these samples. 
We initialize the learning rate at $10^{-3}$ and 
decrease it by a factor of $0.5$ for every 2000 SGD iterations
for the examples with analytic solutions in Section \ref{sec:analytic}, while for  the examples without analytic solutions in Section \ref{sec:no_analytic} we decrease the learning rate by a factor of $0.5$ once the total number of iterations reaches one of the milestones $2000, 4000,6000,10000,20000,30000$. 

We implement Algorithm \ref{alg:dpi_d} using PyTorch and perform all computations on a NVIDIA Tesla K40 GPU with 12 GB memory. The entire algorithm can be briefly summarized as follows. Let $\{\eta_k\}_{k=0}^\infty$ be  a given  sequence, denoting the  accuracy requirement  for each policy iteration.  For each $k\in \N\cup\{0\}$, given the previous iterate $u^{k}$, we compute the feedback controls as  in \eqref{eq:ctrl} and obtain the controlled coefficients as defined in \eqref{eq:linear}. Then we apply the SGD method with analytically derived gradient to optimize $J_{k,d}$  over $\cF_M$ until we obtain a solution $u^{k+1}$ satisfying $J_{k,d}(u^{k+1}) \le \eta_{k}\min(\|u^{k+1}-u^k\|^2_{2,d},\eta_0)$, where $\|\cdot\|_{2,d}$ denotes the discrete $H^2$-norm evaluated based on the training samples in $\Om_d$. We then proceed to the next policy iteration, and terminate Algorithm \ref{alg:dpi_d} once the desired accuracy is achieved. 
 
\subsection{Examples with analytical solutions}\l{sec:analytic} 
In this section, we shall examine the convergence of Algorithm \ref{alg:dpi_d} for solving Dirichlet problems of the form \eqref{eq:navigation_d} with known solutions. In particular, we shall choose a running cost  $f$ such that the analytical solution to \eqref{eq:navigation_d}  is given by $u^*(x,y)=\sin(\pi r^2/2)-\sin(\pi(x^2+y^2)/2)$ for all $(x,y)\in \Om$.
To demonstrate the  generalizability and the superlinear convergence of the numerical solutions obtained by Algorithm \ref{alg:dpi_d}, we generate a different set of collocation points in $\Om$ of the size $N_{\textrm{val}}=2000$, and use them to estimate the relative error and the $q$-factor of the numerical solution $u^k$ obtained from the $k$-th policy iteration for all $k\in \N$:
$$
\textrm{Err}_k=\f{\|u^k-u^*\|_{2,\Om,\textrm{val}}}{\|u^*\|_{2,\Om,\textrm{val}}} \q \textnormal{and}\q q_k= \f{\|u^k-u^*\|_{2,\Om,\textrm{val}}}{\|u^{k-1}-u^*\|_{2,\Om,\textrm{val}}}.
$$
We use neural networks with depth $L=4$ and varying $H$ as trial functions, initialize Algorithm \ref{alg:dpi_d} with $u^0=0$,  and perform experiments with the following model parameters:
 $a=0.04$,  $\sigma_x = 0.5$, $\sigma_y = 0.2$,  $r=0.5$, $R=\sqrt{2}$, $\kappa=0.1$ and $v_s=0.6$.

\begin{figure}[!ht]
    \centering
    \includegraphics[width=0.32\columnwidth,height=5cm]{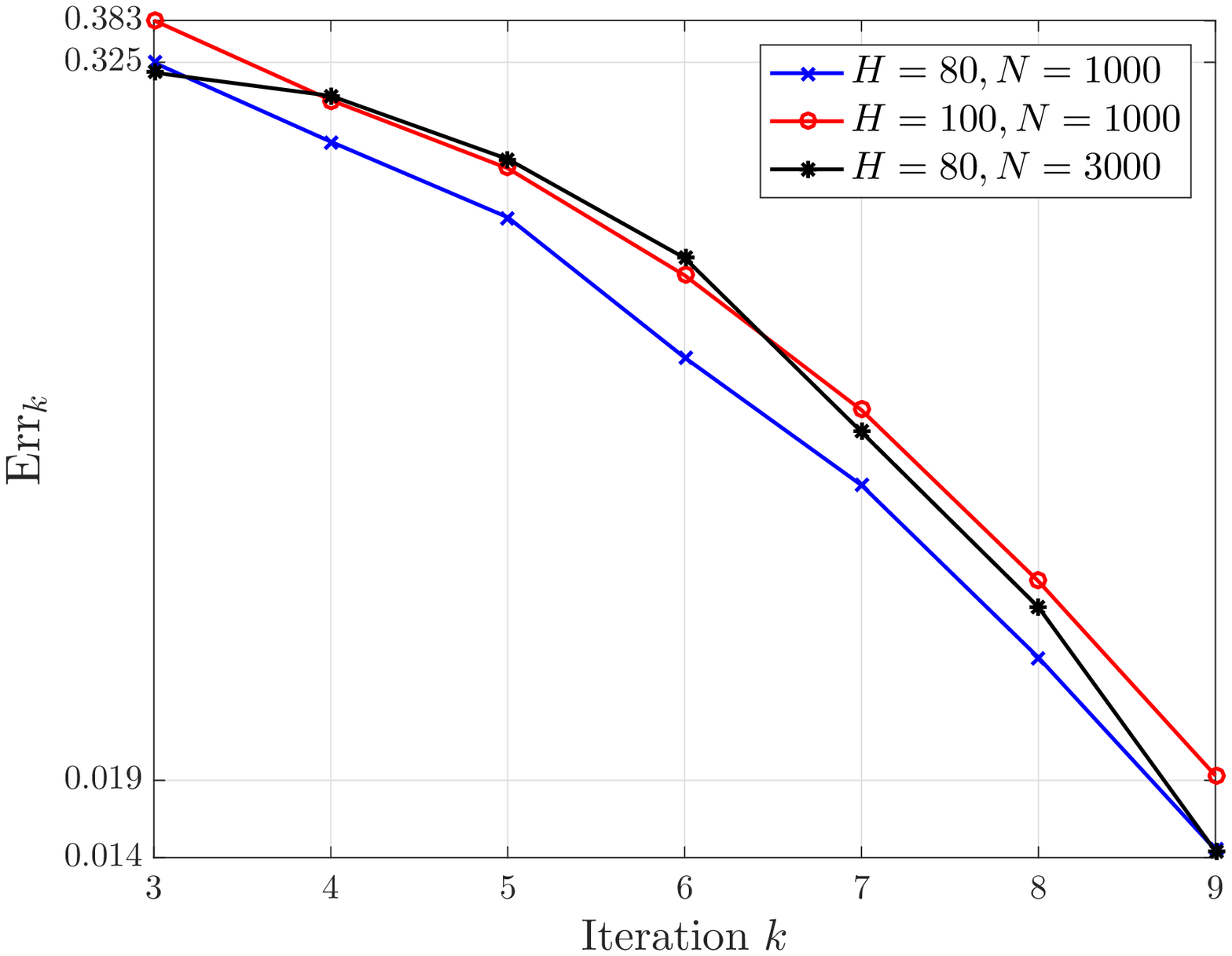}
        \includegraphics[width=0.32\columnwidth,height=5cm]{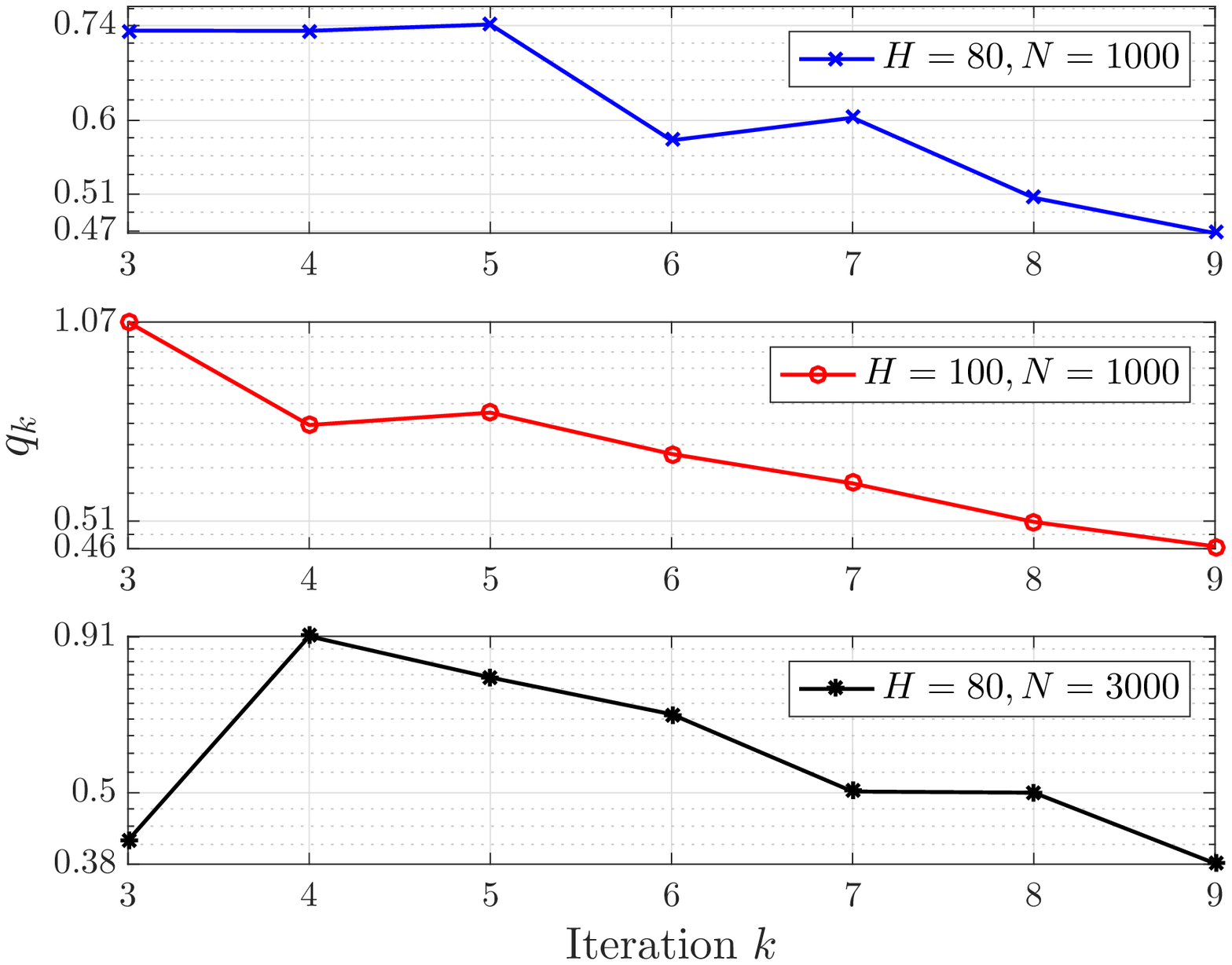}
    \includegraphics[width=0.32\columnwidth,height=5cm]{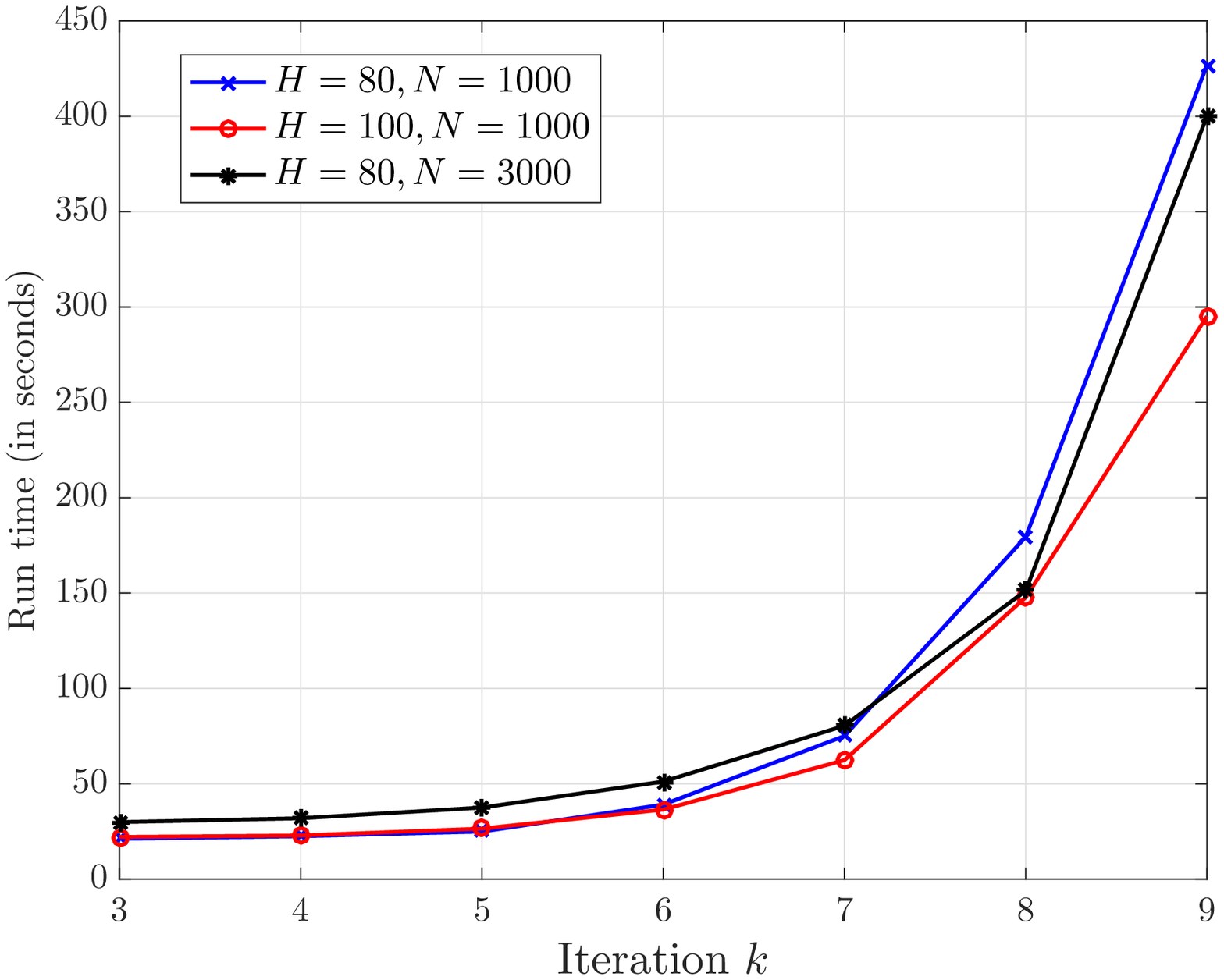}
    \caption{Impact of the training sample size $N$ and the  hidden width $H$ on the performance of  Algorithm \ref{alg:dpi_d}; from left to right: relative errors (plotted in a log scale), $q$-factors and the overall runtime for all policy iterations.}
    \label{fig:sample_size}
 \end{figure}

Figure \ref{fig:sample_size} depicts the performance of  Algorithm \ref{alg:dpi_d} with different  sizes of training samples and the dimensions of hidden layers, which are denoted by $N$ and $H$ respectively. 
The hyper-parameters $\{\eta_k\}_{k=0}^\infty$ are chosen as  $\eta_0=10$ and $\eta_k=2^{-k}$ for all $k\in \N$. One can clearly see from Figure \ref{fig:sample_size} (left) and (middle) that, despite the fact that Algorithm \ref{alg:dpi_d} is initialized with a relatively poor initial guess,  the  numerical solutions converge  superlinearly to the exact solution in the $H^2$-norm for all these combinations of $H$ and $N$, which  confirms the theoretical result in Theorem \ref{thm:dpi_global}. 
 It is interesting to  observe from Figure \ref{fig:sample_size} that,  even though increasing either the complexity of the networks (the red lines) or the size of training samples (the black lines) seems to accelerate the training process slightly (right), neither of them ensures a higher generalization accuracy on the testing samples (left). 
In all our computations, the accuracies of numerical solutions in the $L^2$-norm and the $H^1$-norm are in general higher than the accuracy  in the $H^2$-norm. For example, both the $L^2$-relative error and  
the $H^1$-relative error of  the numerical solution obtained at the 9th policy iteration with $H=80, N=1000$ are  0.0045.

\begin{figure}[!ht]
    \centering
    \includegraphics[width=0.33\columnwidth,height=5cm]{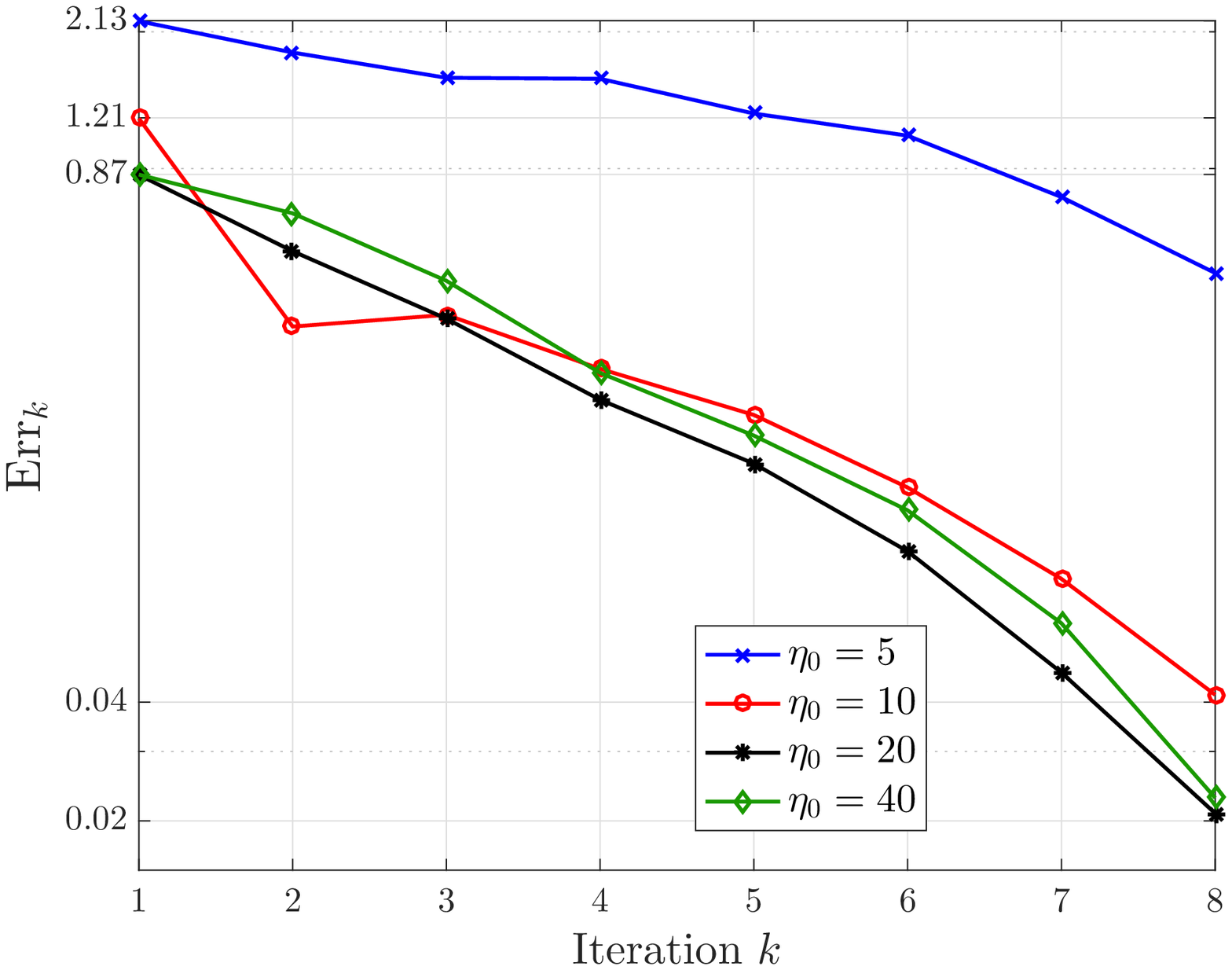}
        \includegraphics[width=0.33\columnwidth,height=5cm]{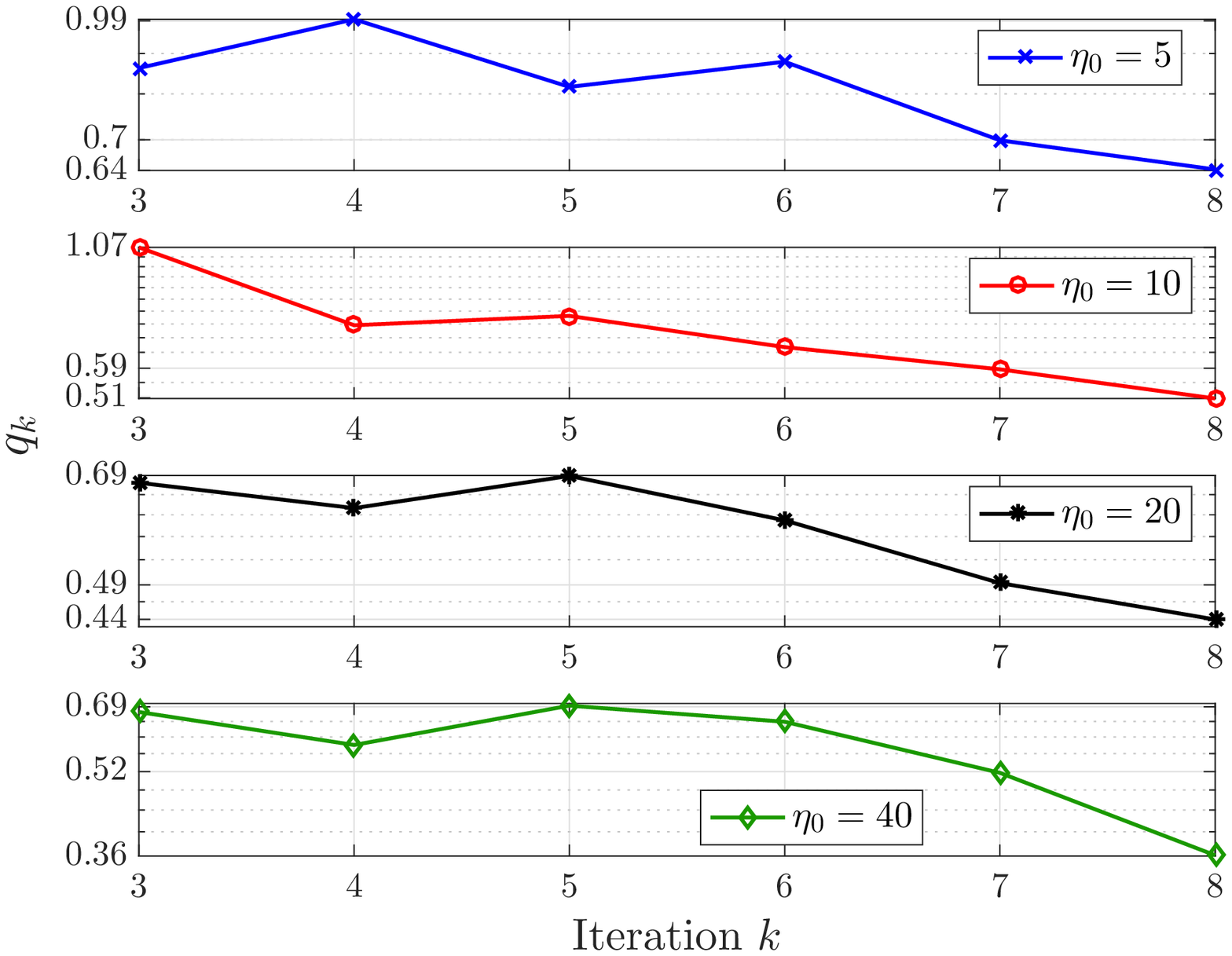}
    \includegraphics[width=0.33\columnwidth,height=5cm]{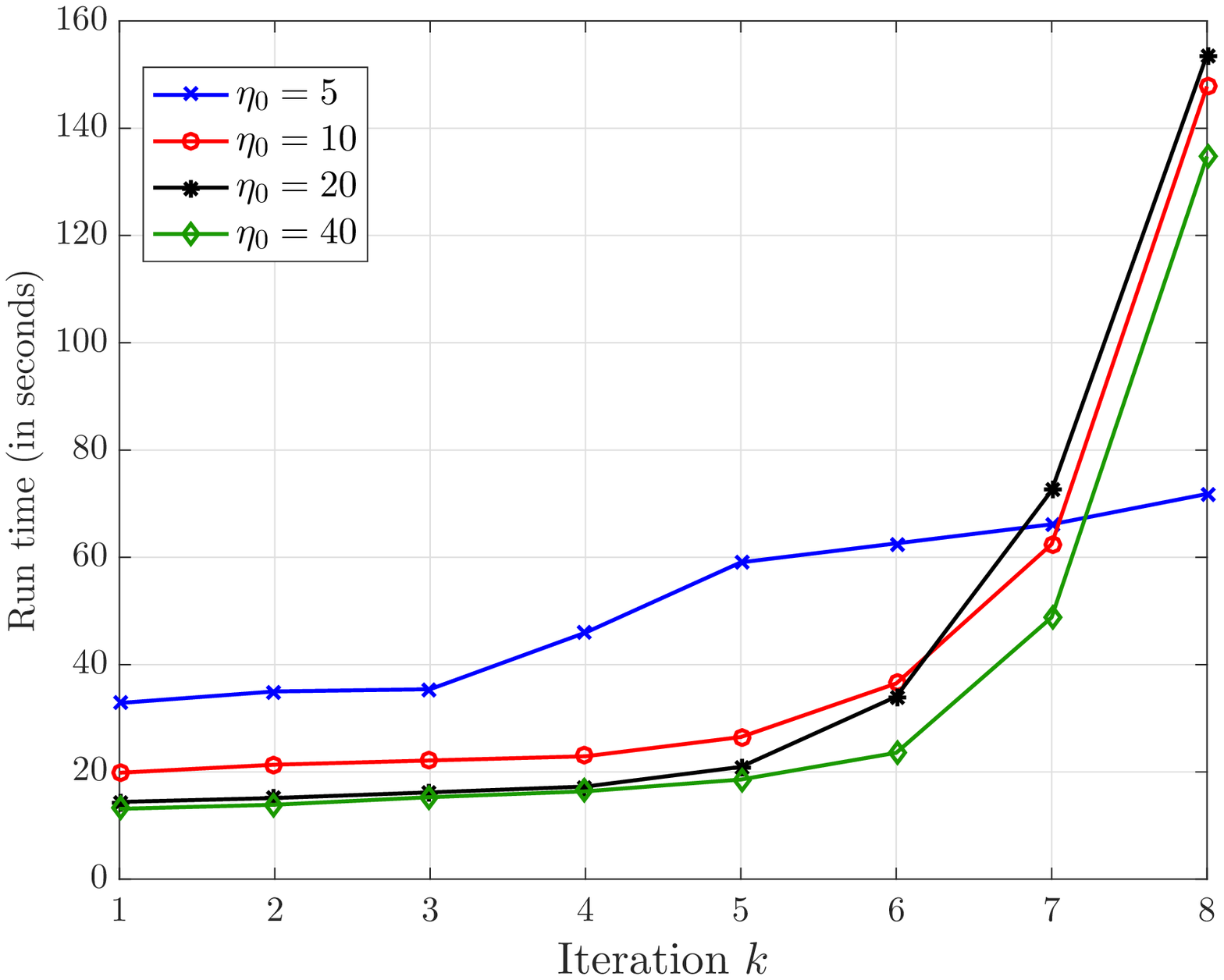}
    \caption{Impact of $\eta_0$ on the performance of  Algorithm \ref{alg:dpi_d}; from left to right: relative errors (plotted in a log scale), $q$-factors and the overall runtime for all policy iterations.}
    \label{fig:eta0}
 \end{figure}

We then proceed to analyze the effects of  the hyper-parameters $\{\eta_k\}_{k=0}^\infty$. Roughly speaking, the magnitude of $\eta_0$ indicates the accuracy of the iterates $\{u^k\}_{k=1}^\infty$ to the linear Dirichlet problems in the initial stage of Algorithm \ref{alg:dpi_d}, while the decay of $\{\eta_k\}_{k=1}^\infty$ determines the speed at which the $q$-factors $\{q_k\}_{k=1}^\infty$ converge to $0$, at an extra cost of solving the optimization problem in a given iteration more accurately for smaller $q_k$. Figure \ref{fig:eta0} presents the numerical results for different choices of $\eta_0$  with a fixed training sample size $N=1000$, hidden width $H=100$ and $\eta_k=2^{-k}$ for all $k\ge 1$.
Note that  solving each linear equation extremely accurate in the initial stage, i.e., by choosing  $\eta_0$ to be a  small value (the blue line), may not be beneficial for the overall performance of the algorithm in terms of both the accuracy and computational efficiency. 
This is due to the fact that the initialization of the algorithm is in general far from the exact solution to  the semilinear boundary value problem, and so are  the solutions of the linear equations arising from the first few policy iterations.
In fact, it appears in our experiments that  the choices of $\eta_0=20,40$ lead to the optimal performance of Algorithm \ref{alg:dpi_d}, which solves the initial equations sufficiently accurately, and  leverages the superlinear convergence of policy iteration to achieve a  higher accuracy  with a similar computational cost.

\begin{figure}[!ht]
    \centering
    \includegraphics[width=0.33\columnwidth,height=5cm]{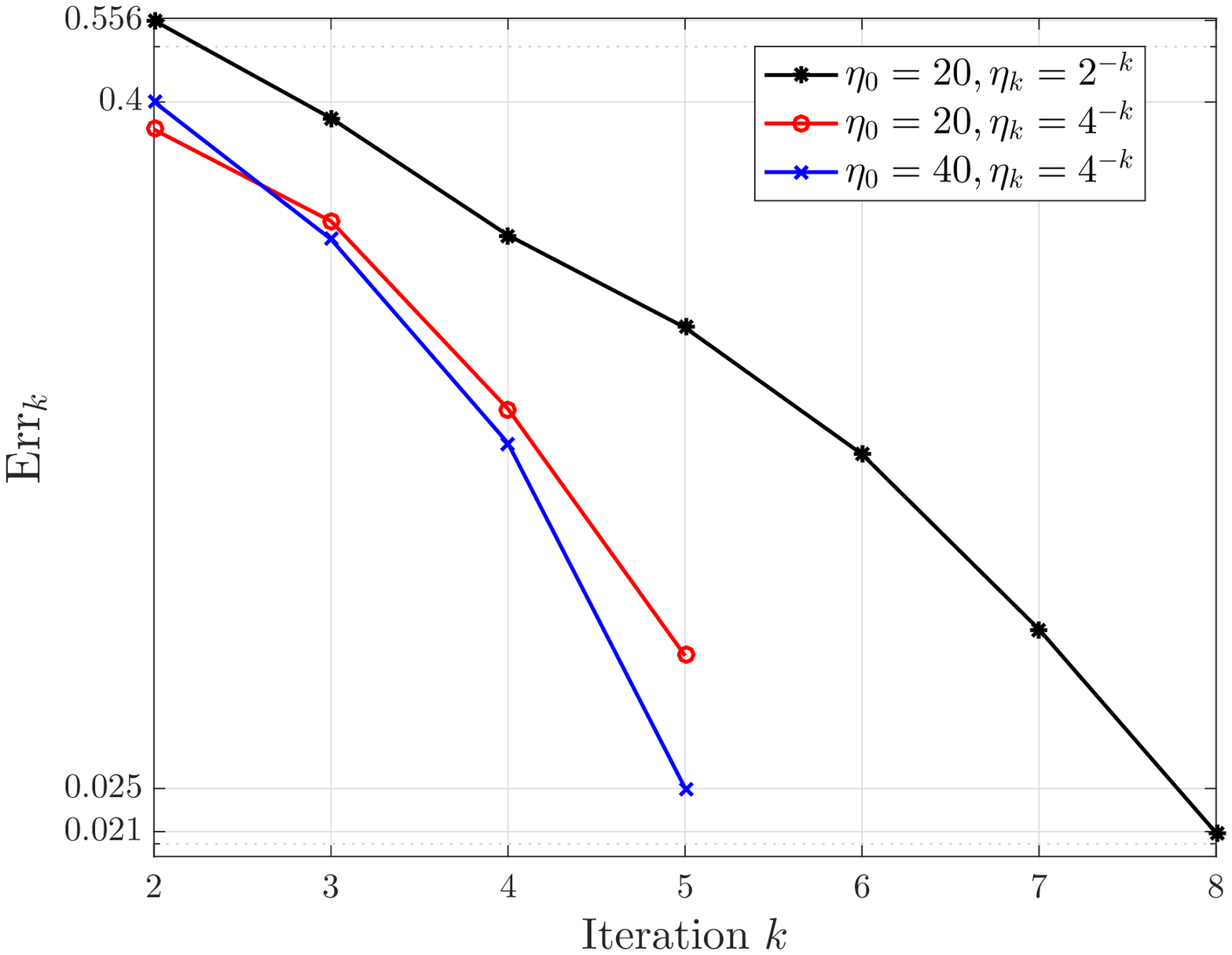}
      \includegraphics[width=0.33\columnwidth,height=5cm]{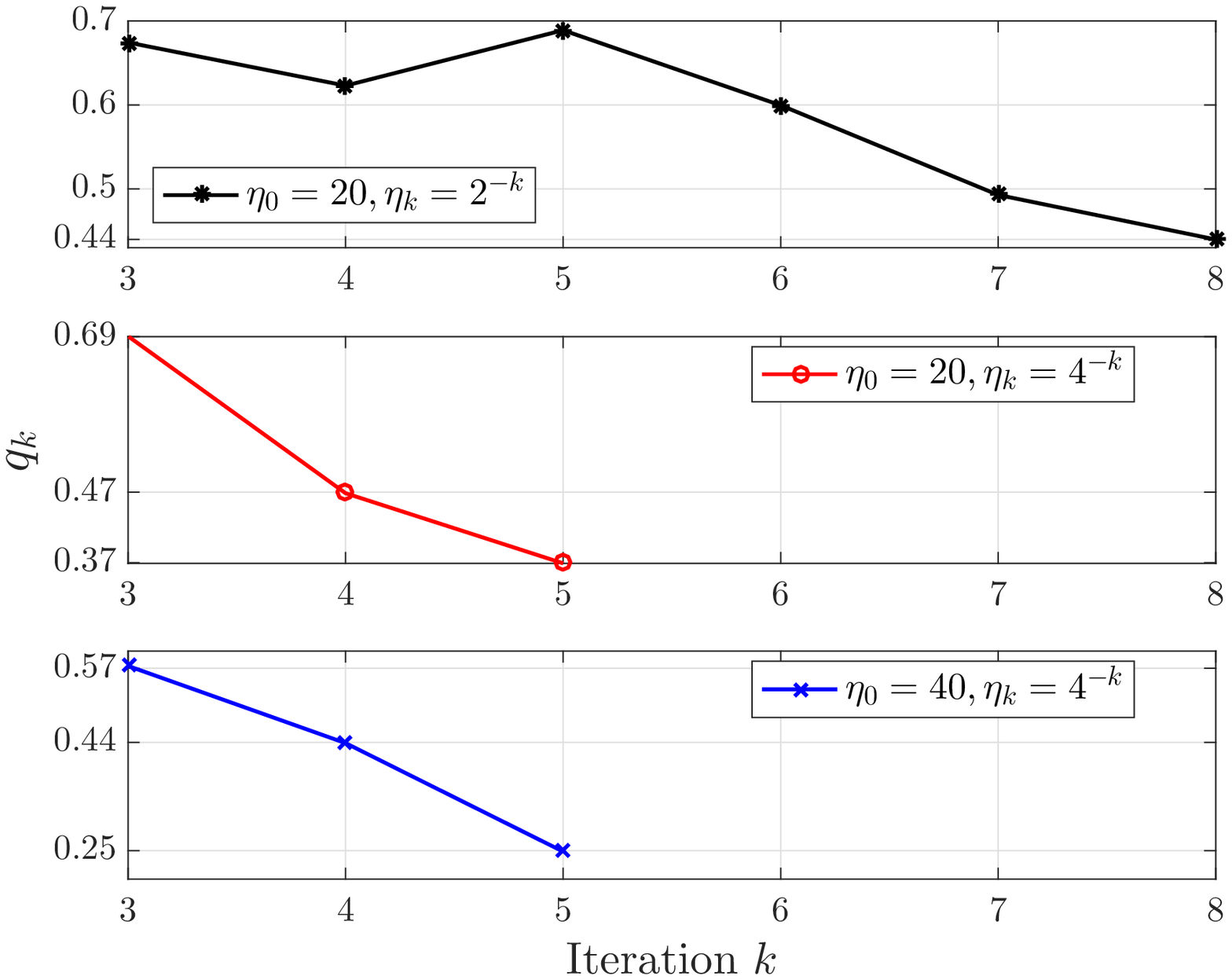}
   \includegraphics[width=0.33\columnwidth,height=5cm]{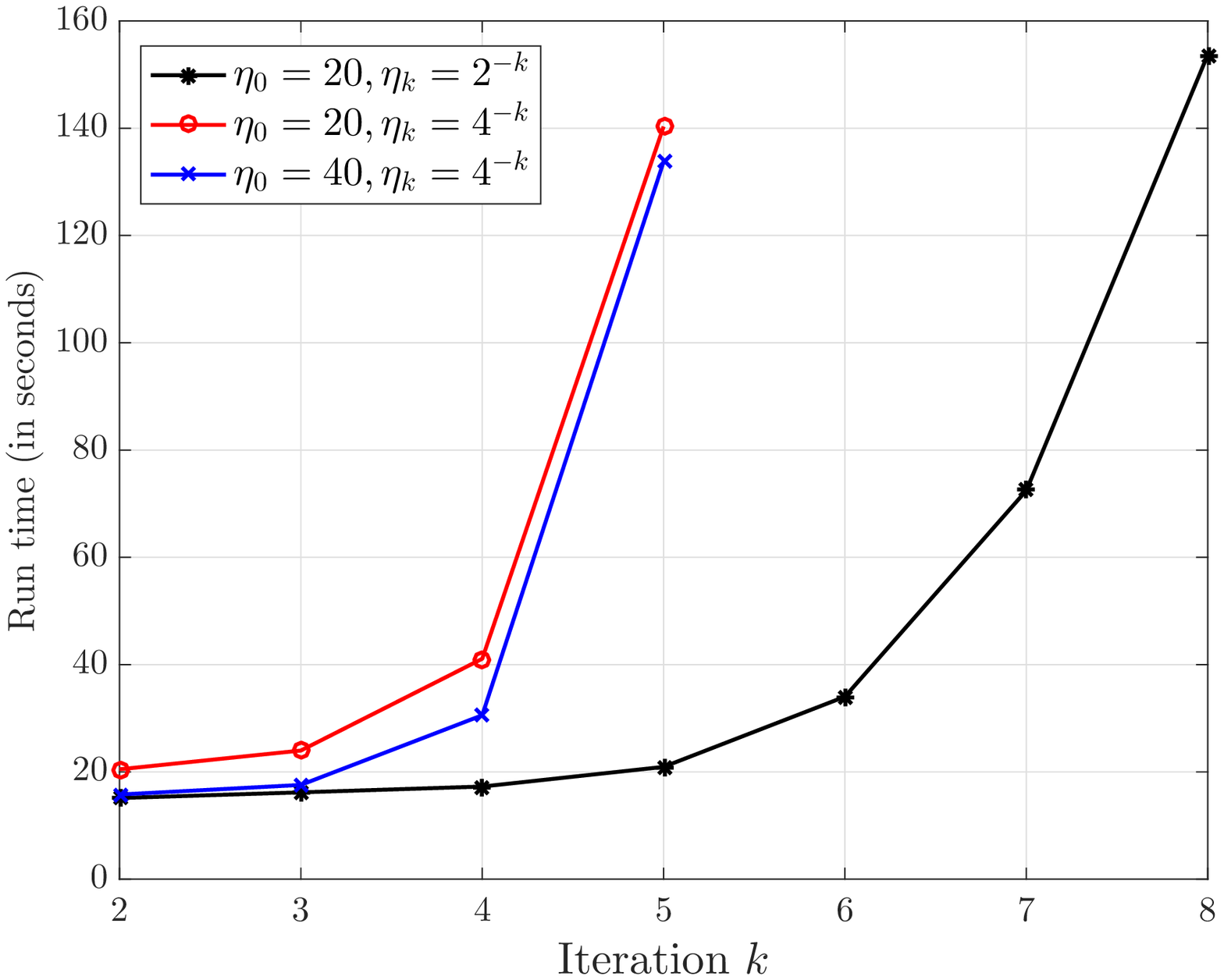}
    \caption{Impact of $\eta_k$ on the performance of  Algorithm \ref{alg:dpi_d}; from left to right: relative errors (plotted in a log scale), $q$-factors and the overall runtime for all policy iterations.}
    \label{fig:etak}
 \end{figure}
 
We further perform computations with different choices of $\{\eta_k\}_{k=1}^\infty$ by  fixing the   training sample size $N=1000$ and the hidden width $H=100$. Numerical results are shown in Figure \ref{fig:etak}, from which we can clearly observe that the iterates obtained with $\eta_k=4^{-k}$, $k\in \N$,  converge more rapidly to the exact solution. 
Note that  for $\eta_k=4^{-k}$, the optimal performance of the algorithm is achieved at $\eta_0=40$ instead of $\eta_0=20$. This is  due to the fact that we solve the first linear Dirichlet problem up to the accuracy $\eta_0\eta_1$ (if we ignore the requirement that $J_{0}(u^1)\le \eta_1\|u^1-u^0\|^2 _{H^2(\Om)}$ in \eqref{eq:J_error}), 
hence one needs to enlarge $\eta_0$ for a smaller $\eta_k$, such that   $\eta_0\eta_1$ is of the same magnitude as before. 
We  observe that 
the  rapid convergence of policy iteration indeed improves the efficiency of the algorithm, in the sense that, to achieve the same accuracy, Algorithm \ref{alg:dpi_d} with $\eta_k=4^{-k}$ requires slightly less  computational time than Algorithm \ref{alg:dpi_d} with $\eta_k=2^{-k}$, even though Algorithm \ref{alg:dpi_d} with $\eta_k=4^{-k}$ takes more time to solve the linear equations for each policy iteration;   see the last few iterations of the blue line and  the black line.
This efficiency improvement is more pronounced for the practical problems with  complicated solutions in Section \ref{sec:no_analytic}; see Figure \ref{fig:hjbi} and Table \ref{table:compare}. 
%

Finally, we shall compare the efficiency of Algorithm \ref{alg:dpi_d}
(with $\eta_0=40$, $\eta_k=4^{-k}$)
 to that of 
the Direct Methods (see e.g.~\cite{lagaris1998,berg2018,e2018,sirignano2018})
by fixing the 
trial functions  ($4$-layer networks with  hidden width $H=100$),
the training samples (with size $N=1000$) 
 and the learning rate of the SGD algorithm.
In the Direct Methods, we shall directly apply the SGD method to 
minimize the following (discretized) squared residual of the semilinear boundary value problem \eqref{eq:navigation_d}:\footnotemark
\footnotetext{
Strictly speaking,
the squared residual \eqref{eq:direct_X} is not differentiable (with respect to the network parameters) at the samples where 
one of the first partial derivatives of the current iterate $u$ is zero,
due to the nonsmooth functions $\|\cdot\|_{\ell^1},\|\cdot\|_{\ell^2}:\R^2\to [0,\infty)$ in 
the HJBI operator $F$ (see \eqref{eq:navigation_d}).
In practice, PyTorch will assign 0 as 
 partial derivatives of $\|\cdot\|_{\ell^1}$ and $\|\cdot\|_{\ell^2}$ functions at their nondifferentiable points,
and use it in the backward propagation.
}
\bb\l{eq:direct_X}
 \|F(u)\|^2_{0,\Om,\textrm{tra}}+\|u-g\|^2_{X,\p\Om,\textrm{tra}}, 
\ee
where $\|\cdot\|_{0,\Om,\textrm{tra}}$ is the discrete $L^2$ interior norm 
evaluated from the  training  samples in $\Om$,
and $\|\cdot\|_{X,\p\Om,\textrm{tra}} $ is a certain discrete boundary norm 
evaluated from samples on the boundary.
In particular, we shall perform computations by setting  
$\|\cdot\|^2_{X,\p\Om,\textrm{tra}}=\|\cdot\|^2_{3/2,\p\Om,\textrm{tra}}$   (defined  as in \eqref{eq:J_d})
and $\|\cdot\|^2_{X,\p\Om,\textrm{tra}}=\vartheta\|\cdot\|^2_{0,\p\Om,\textrm{tra}}$ with different choices of $\vartheta>0$
($\vartheta=1$  in \cite{sirignano2018}
and $\vartheta\in \{500,1000\}$ in \cite{e2018}),
which will be referred to as 
``DM with $\|\cdot\|^2_{3/2,\p\Om}$" 
and ``DM with $\vartheta\|\cdot\|^2_{0,\p\Om}$",
respectively, in the following discussion.
%
For both the Direct Methods and Algorithm \ref{alg:dpi_d},
we shall estimate  the  $H^2$-relative error of the numerical solution $\hat{u}_i$ obtained from the $i$-th SGD iteration 
by using the same testing samples in $\Om$ of the size $N_{\textrm{val}}=2000$ as follows:
$$
\textrm{SGD Err}_i={\|\hat{u}_i-u^*\|_{2,\Om,\textrm{val}}}/{\|u^*\|_{2,\Om,\textrm{val}}}, 
$$ 
where $u^*$ denotes  the analytical solution to \eqref{eq:navigation_d}.

Figure \ref{fig:direct_pi_analytic} (left)
depicts the  $H^2$-convergence of
``DM with $\|\cdot\|^2_{3/2,\p\Om}$" 
and 
``DM with $\vartheta\|\cdot\|^2_{0,\p\Om}$" (with various choices of $\vartheta>0$)
as the number of SGD iterations tends to infinity,
which clearly shows that,
compared with  using the $L^2$-boundary norm as in \cite{e2018,sirignano2018},
 incorporating the $H^{3/2}$-boundary norm in the loss function helps achieve a higher $H^2$-accuracy  
of the numerical solutions.
It is interesting to point out that,
even though  
penalizing the $L^2$-norm of the boundary term 
with a suitable parameter $\vartheta$
  helps improve the accuracy of ``DM with $\vartheta\|\cdot\|^2_{0,\p\Om}$" as suggested in \cite{e2018},
in our experiments, $\vartheta=10$ leads to the best $H^2$-convergence
of ``DM with $\vartheta\|\cdot\|^2_{0,\p\Om}$" (after $10^4$ SGD iterations) among 
other choices of $\vartheta\in \{0.1,1,5, 10,20,50,100,500,1000\}$.

Figure \ref{fig:direct_pi_analytic} (right)
presents the decay of $H^2$-relative errors with respect to the number of SGD iterations
used in
``DM with $\|\cdot\|^2_{0,\p\Om}$",
``DM with $\|\cdot\|^2_{3/2,\p\Om}$"
and  Algorithm \ref{alg:dpi_d}, 
 which clearly demonstrates that 
the superlinear convergence of policy iteration   
significantly accelerates the convergence of the algorithm.
In particular,
the accuracy enhancement of 
 Algorithm \ref{alg:dpi_d} over ``DM with $\|\cdot\|^2_{0,\p\Om}$" 
 (or equivalently the Deep Galerkin Method 
proposed in \cite{sirignano2018})
is of a factor of 20
with $10^4$ SGD iterations.
We remark that
the training time of Algorithm \ref{alg:dpi_d} is only slightly longer than
that of  ``DM with $\|\cdot\|^2_{0,\p\Om}$"
(the  runtimes 
of Algorithm \ref{alg:dpi_d} and ``DM with $\|\cdot\|^2_{0,\p\Om}$" with $10^4$ SGD iterations
are 333 and 308 seconds, respectively),
since
Algorithm \ref{alg:dpi_d} requires to  determine whether a given iterate
solves the policy evaluation equations sufficiently accurate (see \eqref{eq:J_error}), 
in order to proceed to the next policy iteration step.

\begin{figure}[!ht]
    \centering
    \includegraphics[width=0.43\columnwidth,height=6.2cm]{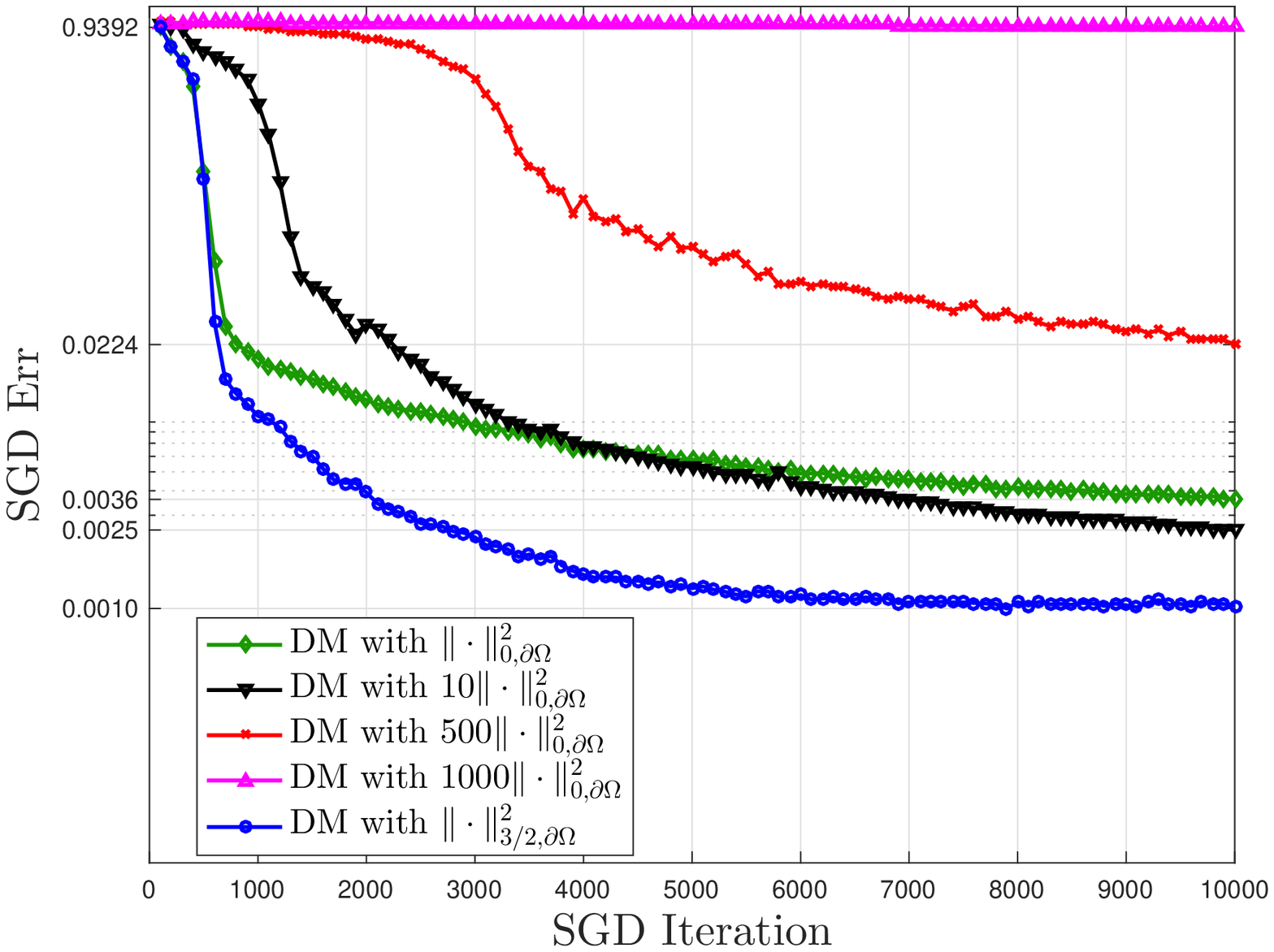}
    \includegraphics[width=0.43\columnwidth,height=6.2cm]{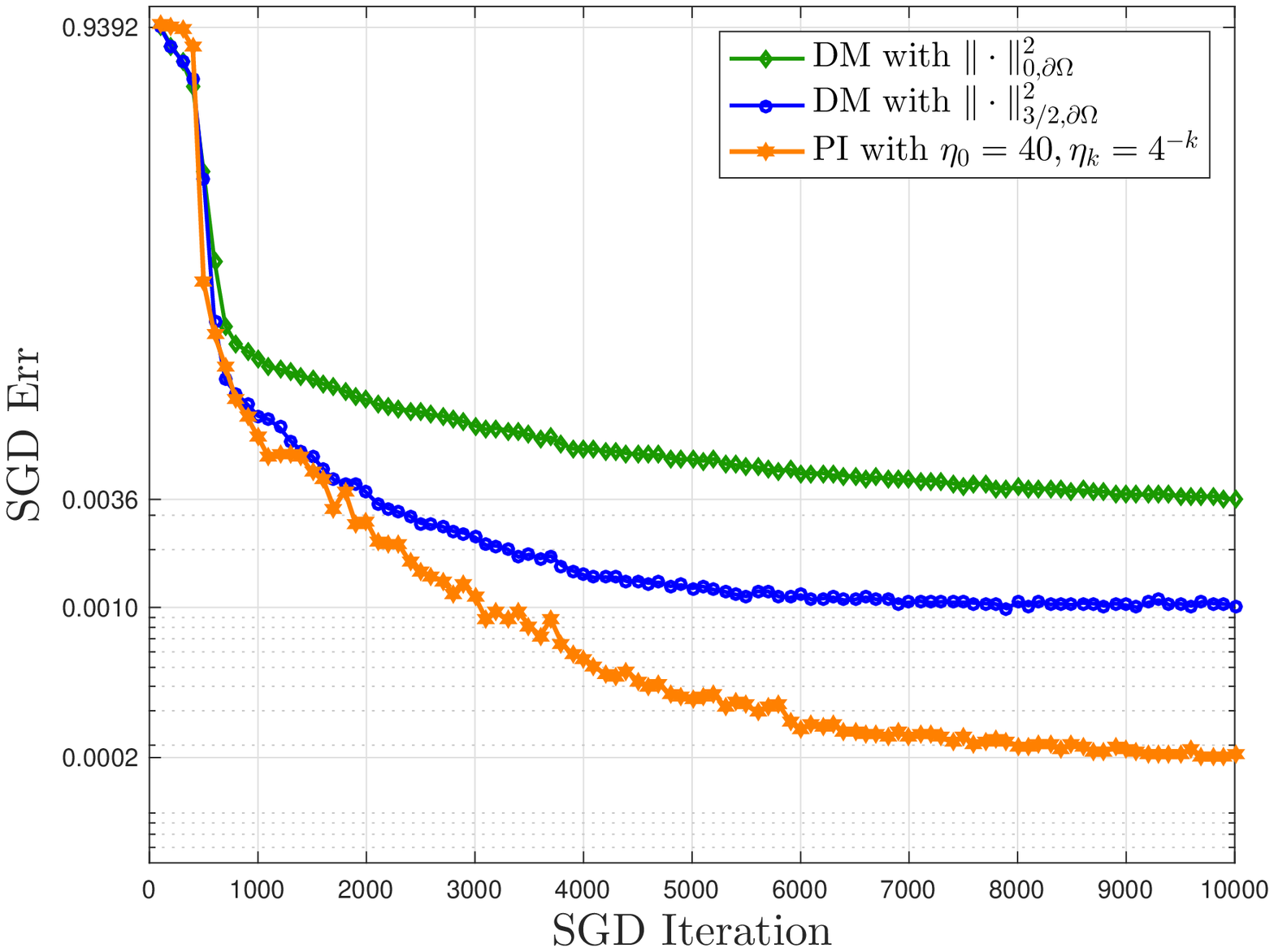}
    \caption{ Relative errors   of  the Direct Methods and Algorithm \ref{alg:dpi_d}
    with different numbers of SGD iterations (plotted in a log scale);
    from left to right: 
    improvements  caused by the $H^{3/2}$-boundary norm
    and 
    by policy iteration.}
    \label{fig:direct_pi_analytic}
 \end{figure}

\color{black}

\subsection{Examples without analytical solutions}\l{sec:no_analytic} 
 In this section, we shall demonstrate the performance of Algorithm \ref{alg:dpi_d} by solving \eqref{eq:navigation_d} with $f\equiv 1$. This corresponds to a minimum time problem with preferred targets, whose solution in general is not known analytically. Numerical simulations will be conducted with  the following model parameters:
 $a=0.2$,  $\sigma_x = 0.5$, $\sigma_y = 0.2$,  $r=0.5$, $R=\sqrt{2}$ and $\kappa=0.1$ but two different values of $v_s$, $v_s=0.5$ and  $v_s=1.2$, which are   associated with the two scenarios where the ship moves slower than and faster than the wind, respectively.  The algorithm is initialized with $u^0=0$.
 
We remark that this is a numerically challenging problem  due to the fact that the convection term in \eqref{eq:navigation_d} dominates the diffusion term, which leads to a sharp change of the solution and its derivatives near the boundaries. However, as we shall see, these boundary layers can be captured effectively by 
the numerical solutions of Algorithm \ref{alg:dpi_d}.
 
Figure \ref{fig:scenario} presents the numerical results for the two different scenarios obtained by Algorithm \ref{alg:dpi_d} with $N=2000$, $\eta_0=40$ and $\eta_k=1/k$ for all $k\in \N$. The set of trial functions consists of all fully-connected neural networks with depth $L=7$ and hidden width $H=50$ (the total number of  parameters in this network is  12951). We can clearly see from Figure \ref{fig:scenario} (left) and (middle) that for both scenarios, the numerical solution $\bar{u}$ and its derivatives are symmetric with respect to the axis $y=0$, and change rapidly near the boundaries.

The feedback control strategies, computed by \eqref{eq:ctrl}, are depicted in Figure \ref{fig:scenario} (right). If the ship starts from the left-hand side 
and travels toward the inner boundary, then the expected travel time to $\p B_r(0)$ is around $\f{R-r}{v_s+v_c}$, which is smaller than the exit cost along $\p B_R(0)$. Hence the ship would move in the direction of the positive $x$-axis for both cases, $v_s<v_c$ and $v_s>v_c$. However, the optimal control is different for the two scenarios if the  ship is on the right-hand side. 
For the case where  the ship's speed is less than the wind ($v_s=0.5$), if the ship is closed to  $\p B_r(0)$, then it would move in the direction of
the negative $x$-axis,
hoping the  random perturbation  of the wind would bring it to the preferred target, while if it is far from $\p B_r(0)$, then it has  less chance to reach $\p B_r(0)$, so  it would move along the positive $x$-axis. On the other hand, for the case where  the ship's speed is larger than the wind ($v_s=1.2$), the ship would in general try to reach the inner boundary. However, if the ship is sufficiently close to $\p B_R(0)$ in the right-hand half-plane, then the expected travel time to $\p B_r(0)$ is around $\f{R-r}{v_s-v_c}$, which is larger than the exit cost along $\p B_R(0)$. Hence the ship would choose to exit directly from  
the outer boundary.

\begin{figure}[!ht]
    \centering
    \includegraphics[width=0.35\columnwidth,height=5cm]{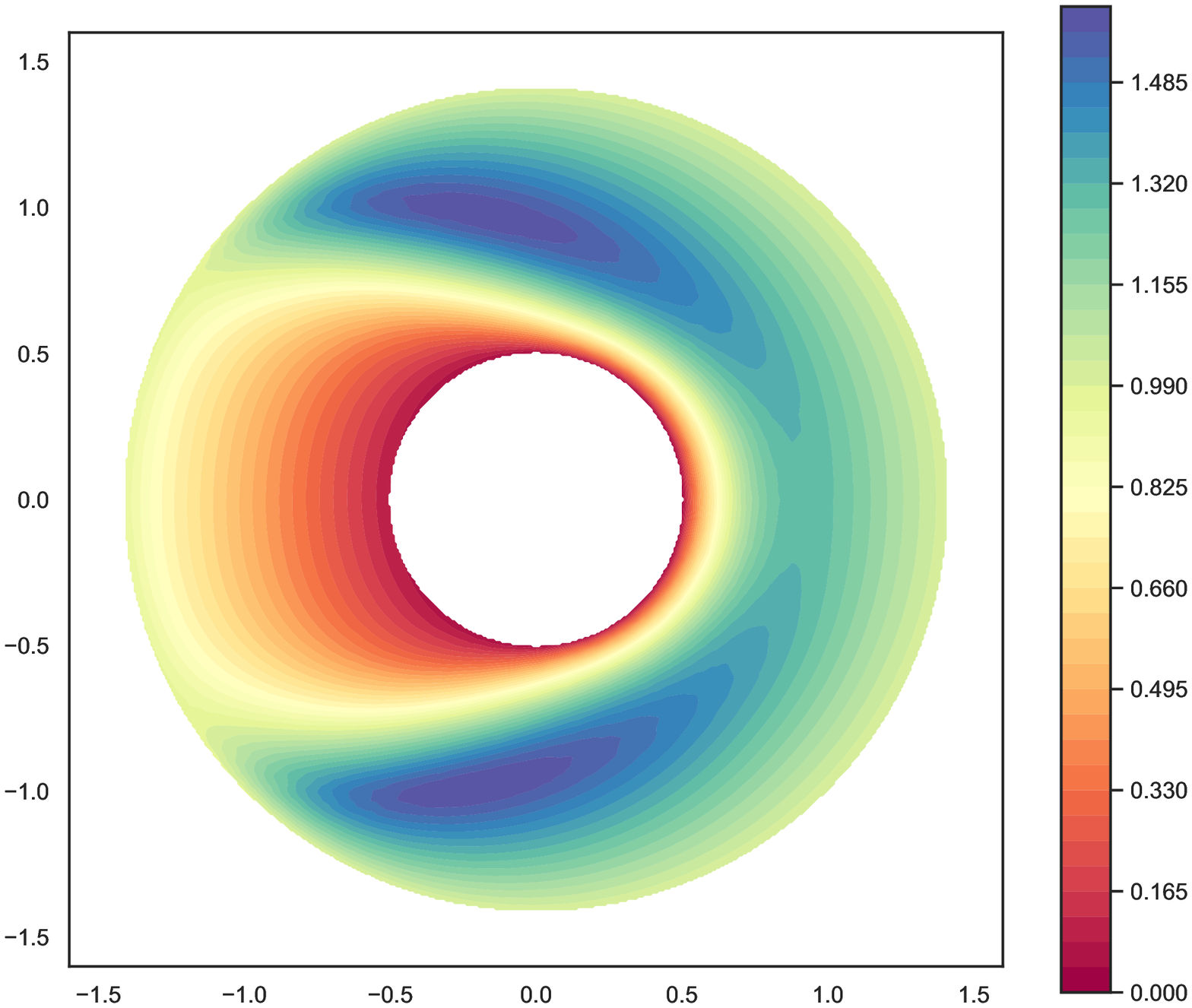}
        \includegraphics[width=0.33\columnwidth,height=5cm]{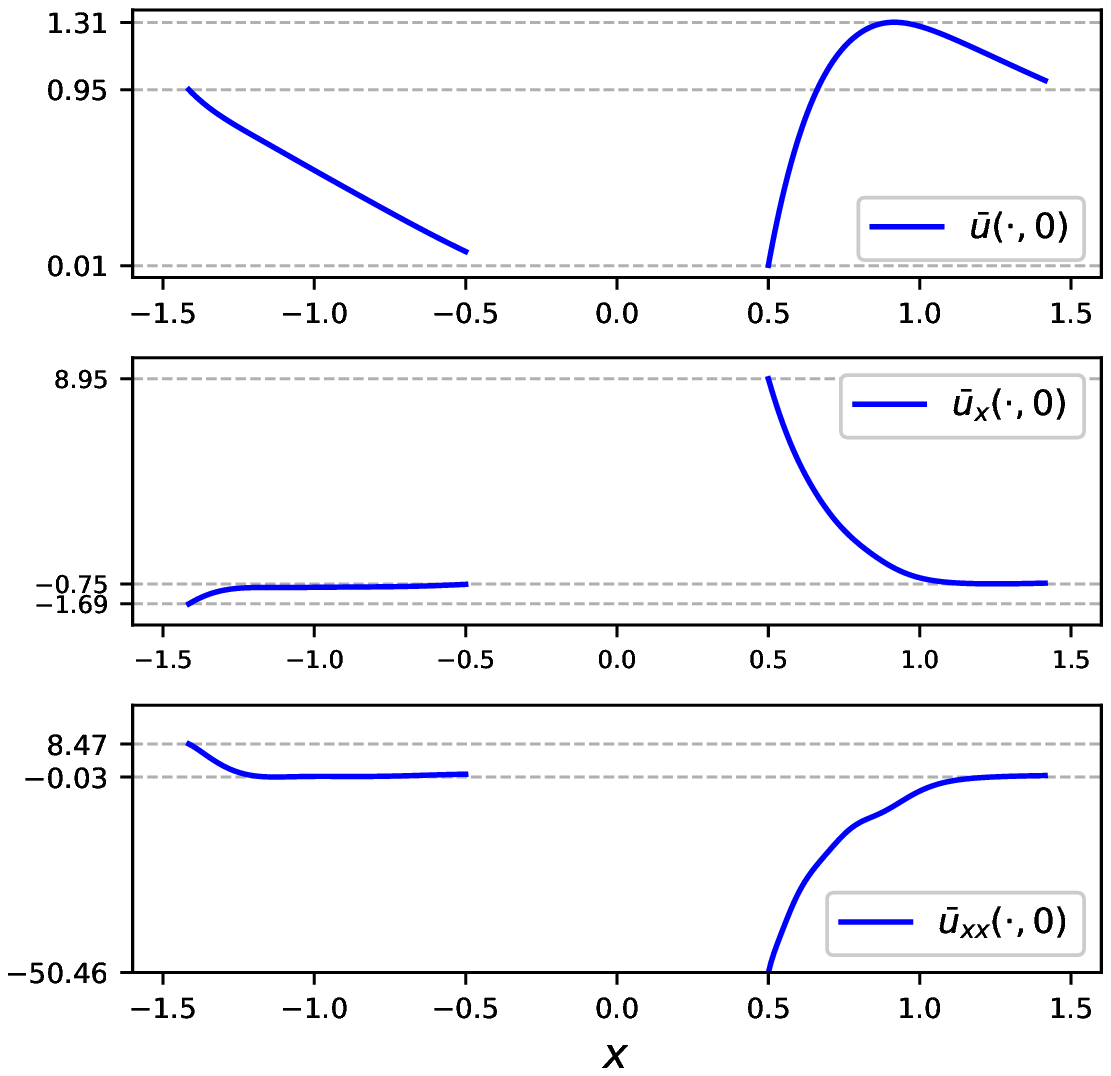}
    \includegraphics[width=0.3\columnwidth,height=5cm]{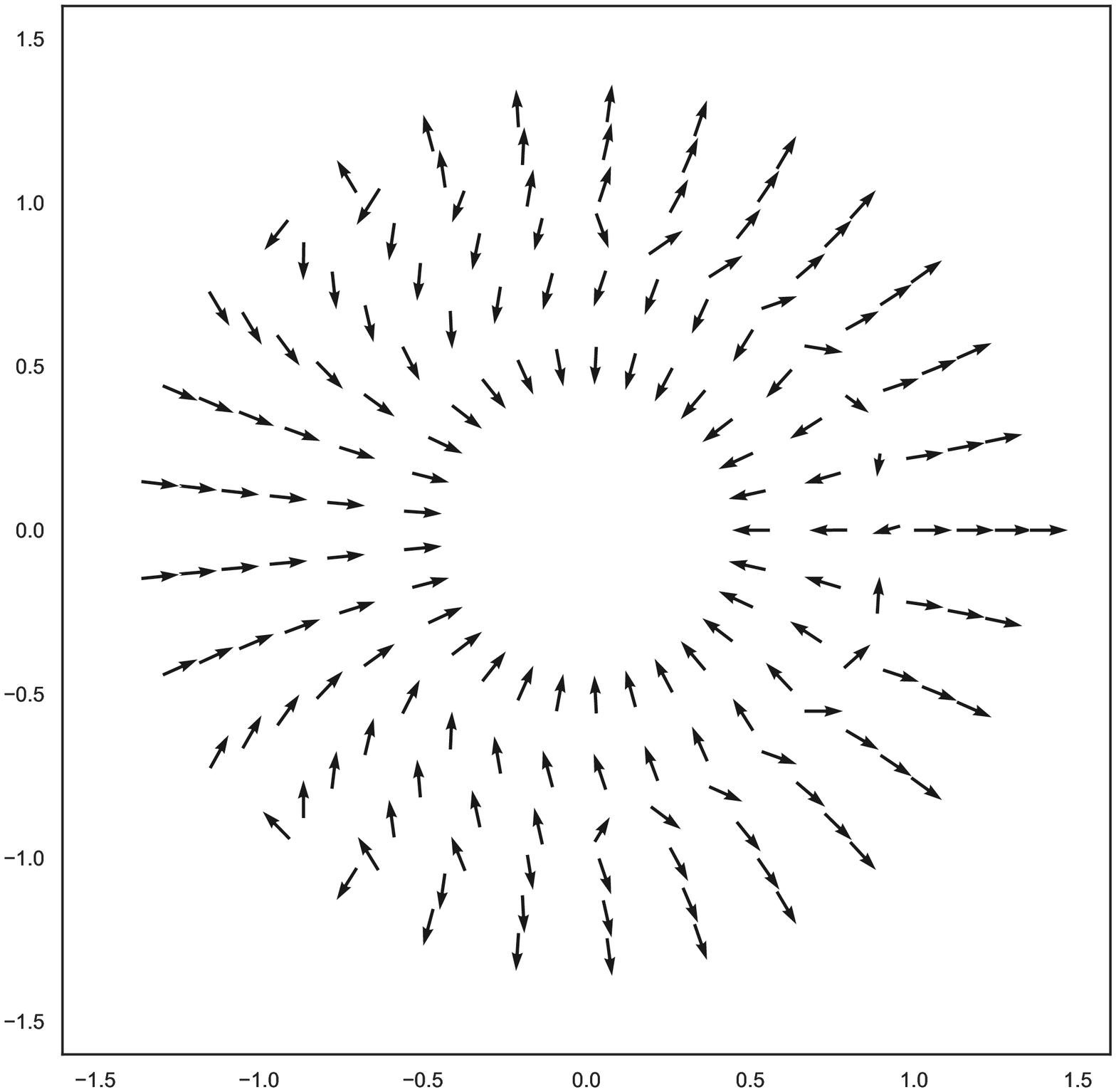}\\
     \includegraphics[width=0.35\columnwidth,height=5cm]{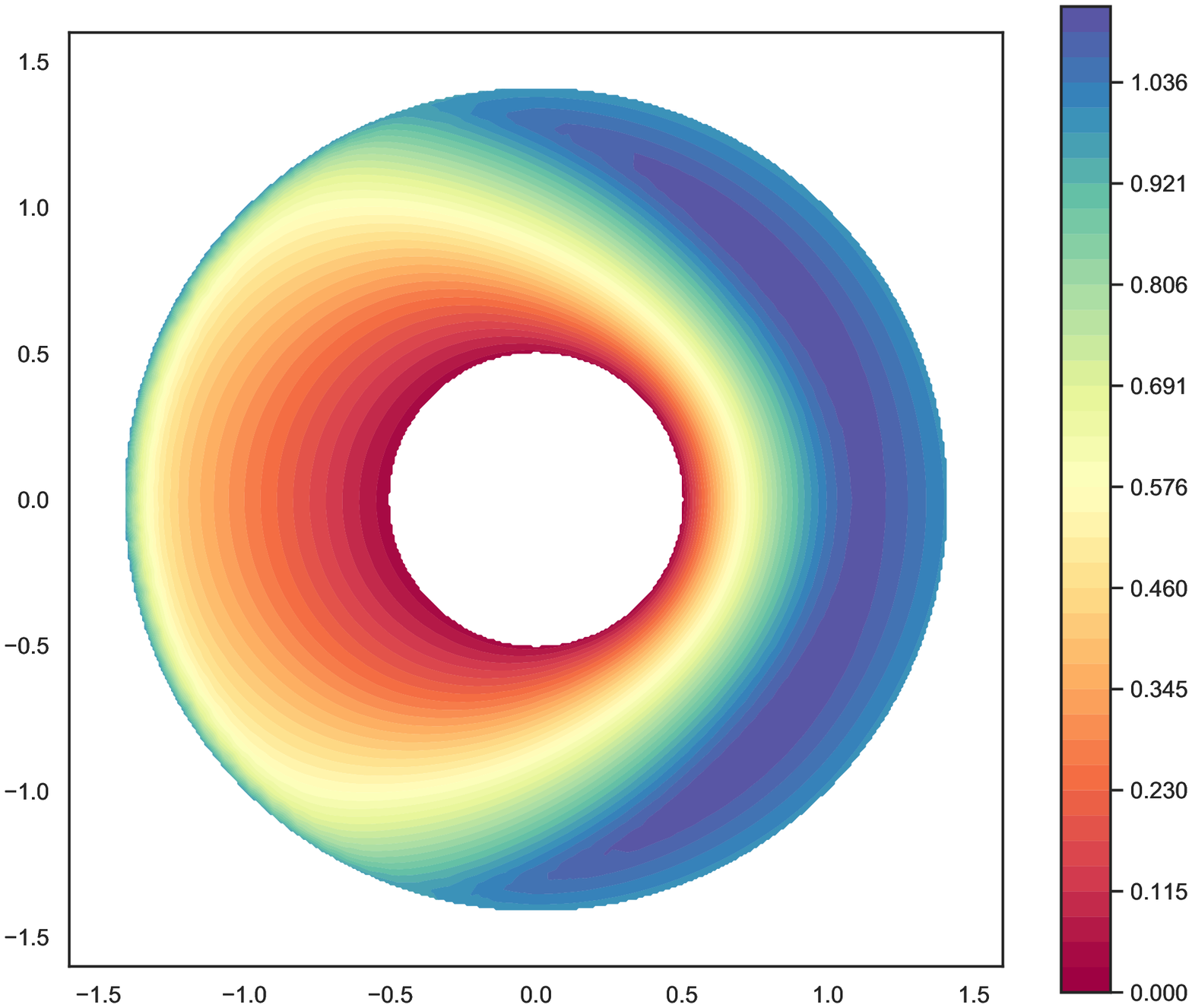}
     \includegraphics[width=0.33\columnwidth,height=5cm]{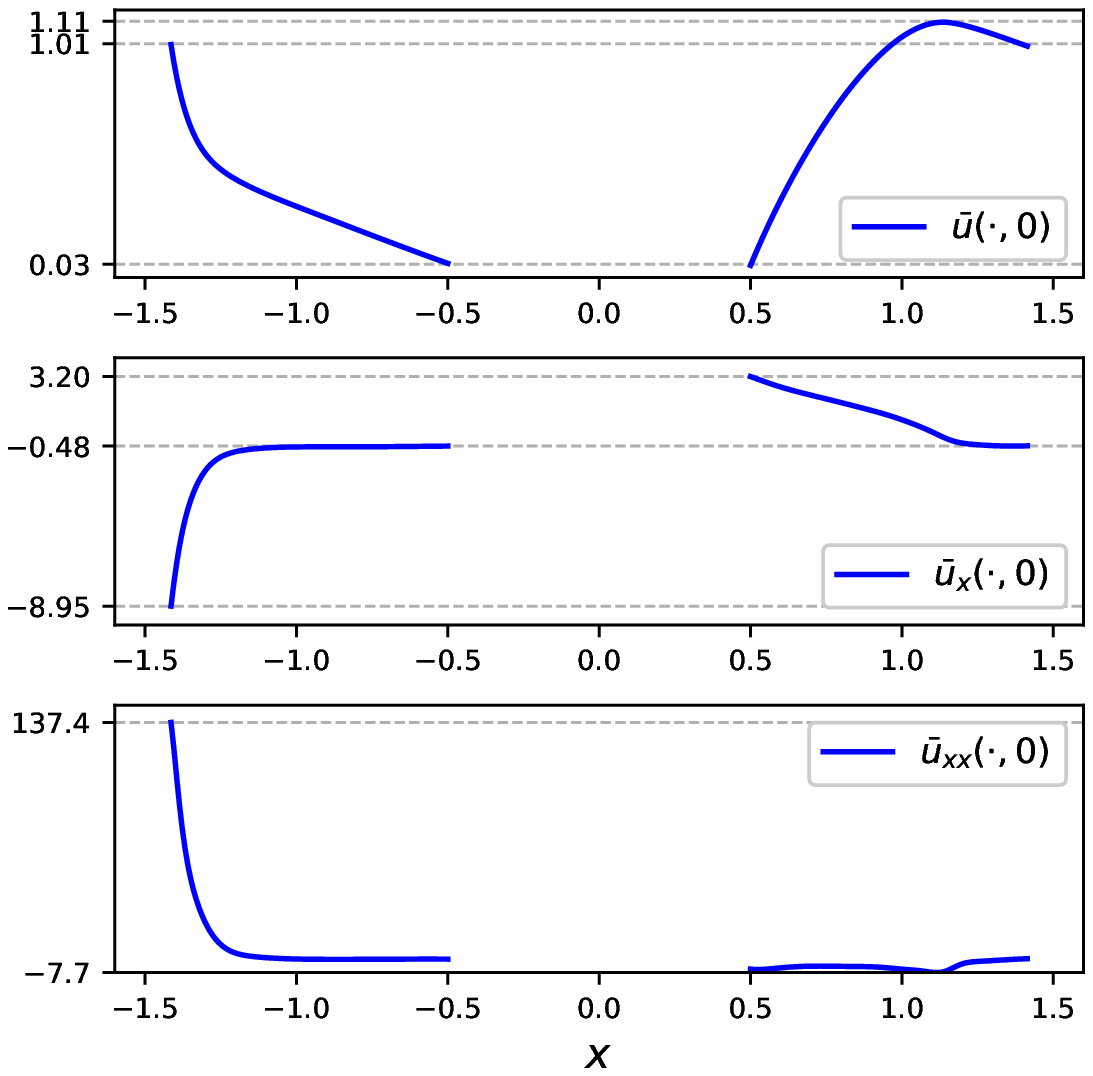}
    \includegraphics[width=0.3\columnwidth,height=5cm]{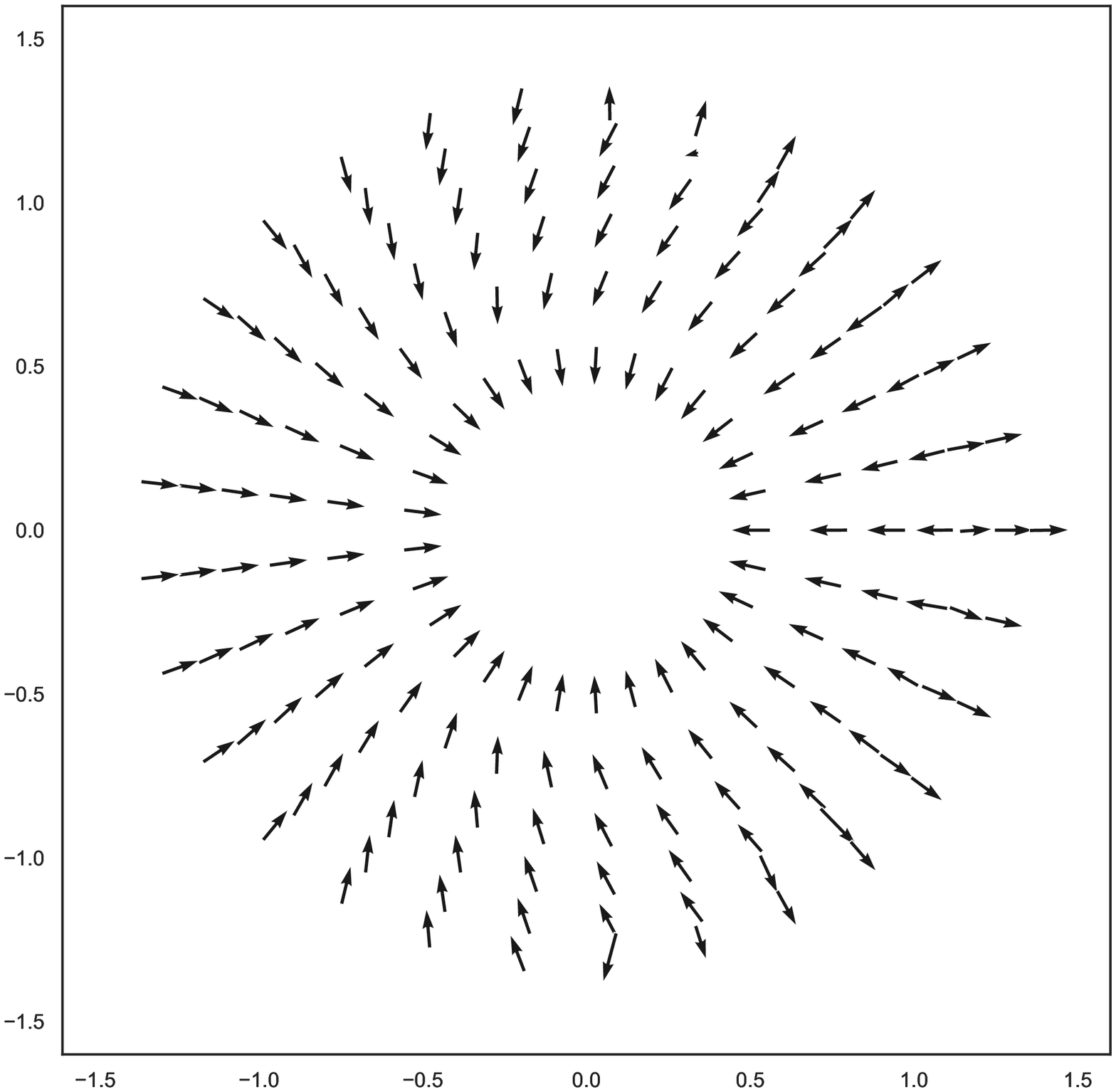}
    \caption{Numerical results for the two different scenarios;   from top to bottom: the ship moves slower than  the wind ($v_s=0.5$) and faster than the wind ($v_s=1.2$);  from left to right: the value function $\bar{u}$, the numerical solutions along $y=0$, and feedback control strategies.}   
     \label{fig:scenario}
 \end{figure}

We then analyze the convergence of Algorithm \ref{alg:dpi_d} by performing computations with 7-layer networks with different hidden width $H$
 (networks with wider hidden layers are employed such that  every linear Dirichlet problem can be solved more accurately)
 and parameters $\{\eta_k\}_{k=0}^\infty$. For any given iterate $u^k$, we shall consider the following (squared) residual of the semilinear boundary value problem \eqref{eq:navigation_d} :
\bb\l{eq:hjbi_residual}
\textrm{HJBI Residual}\coloneqq \|F(u^k)\|^2_{0,\Om,\textrm{val}}+\|u^k-g\|^2_{3/2,\p\Om,\textrm{val}},
\ee
which will be evaluated similar to \eqref{eq:J_d} based on testing samples in $\Om$ and on $(\p\Om)^2$ of the same size $N_{\textrm{val}}=2000$. Figure \ref{fig:hjbi} presents  the decay of the residuals in terms of  the number of policy iterations, which  suggests the $H^2$-superlinear convergence of the iterates $\{u^k\}_{k=0}^\infty$  (the $H^2$-norms of the last iterates  for $v_s=0.5$ and $v_s=1.2$ are 20.3 and 31.8, respectively).
Note that  the parameter $\eta_k=1/k$, $k\in \N$ leads to a  slower and more oscillating convergence of the  iterates $\{u^k\}_{k=0}^\infty$, due to the fact that we apply a mini-batch SGD method to optimize the discrete cost functional $J_{k,d}$ for each policy iteration.  A faster and smoother convergence can be achieved by choosing a more rapidly decaying $\{\eta_k\}_{k=1}^\infty$.

\begin{figure}[!ht]
    \centering
    \includegraphics[width=0.43\columnwidth,height=6.2cm]{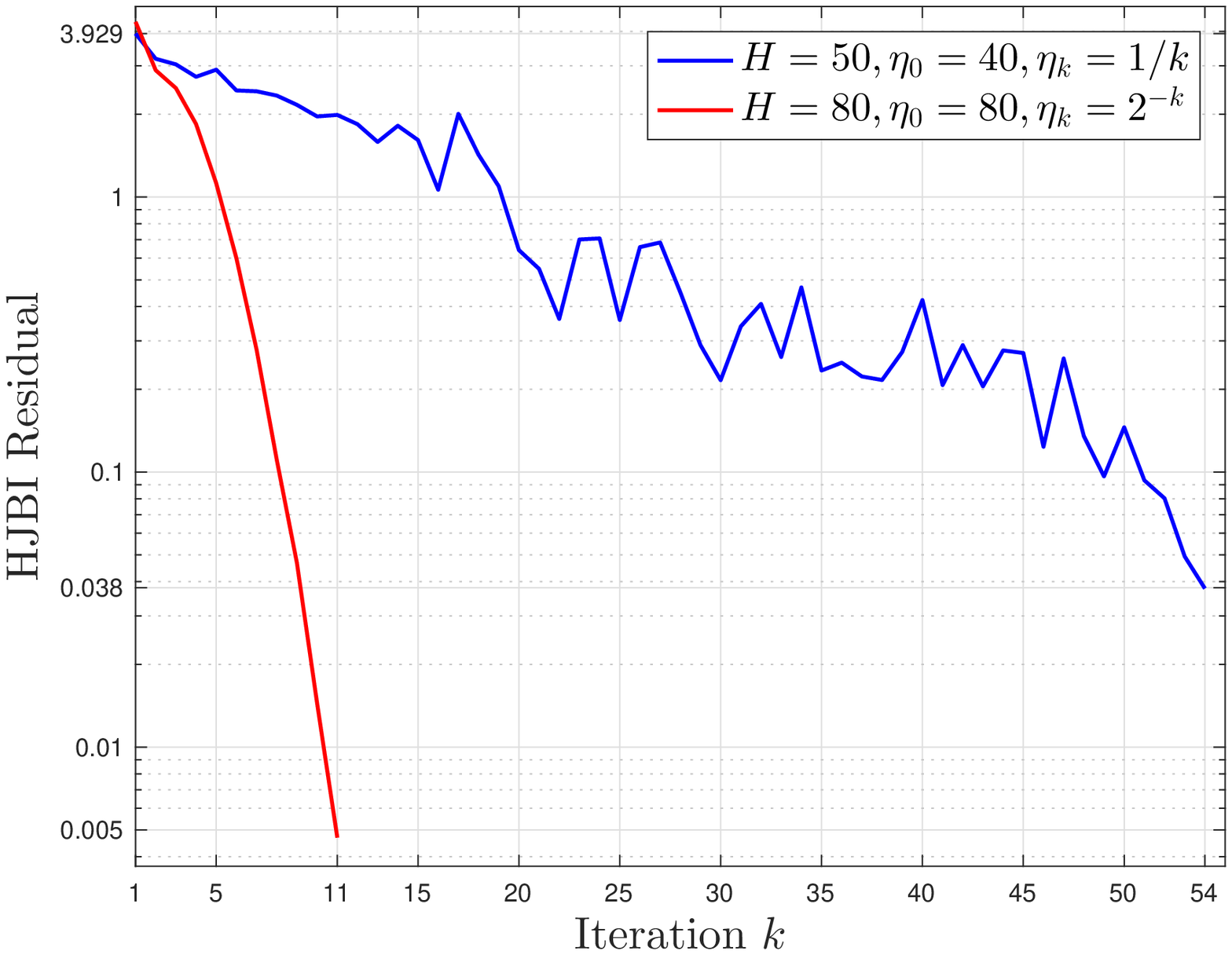}
    \includegraphics[width=0.43\columnwidth,height=6.2cm]{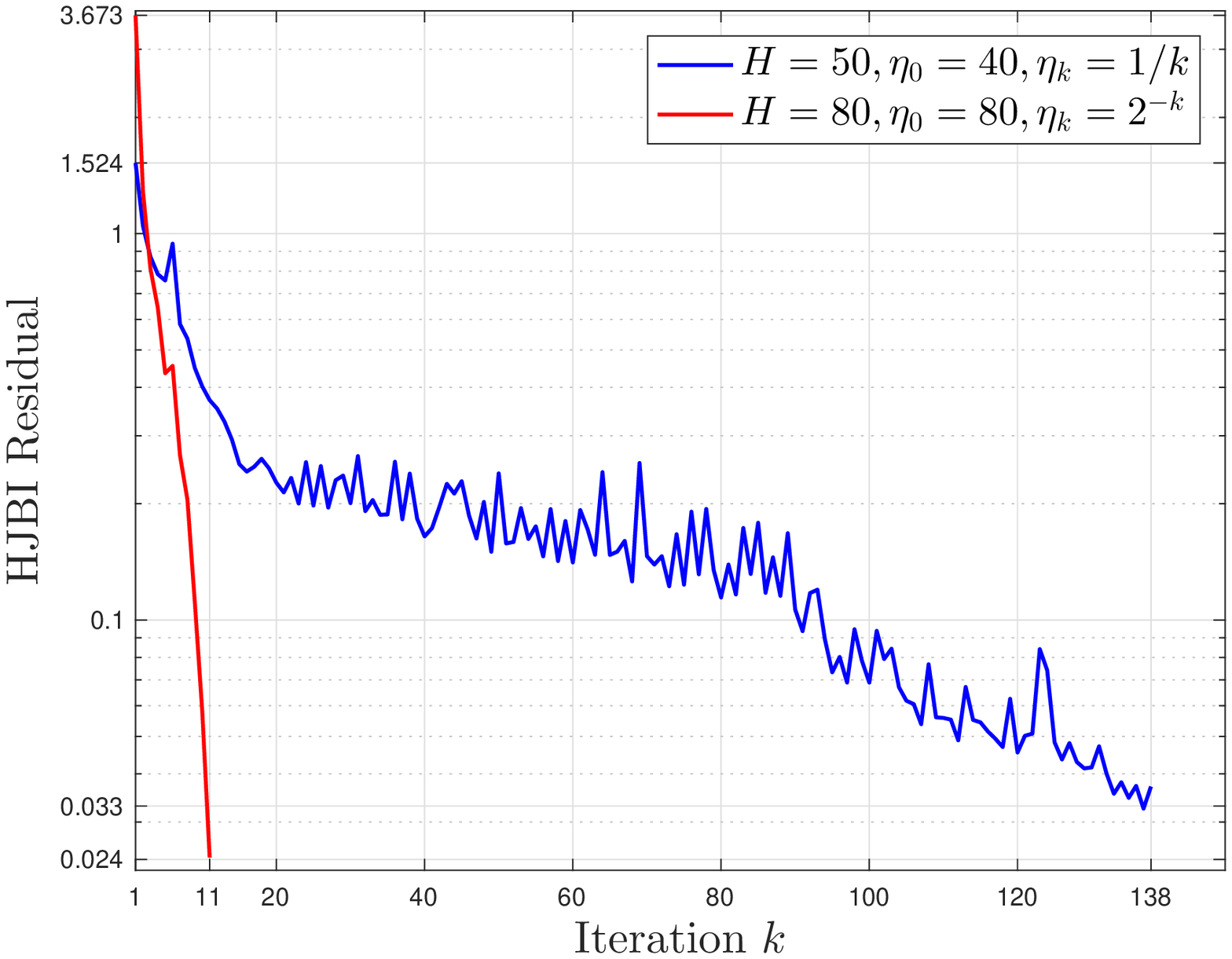}
    \caption{Residuals of the HJBI Dirichlet problems for the two different scenarios with respect to the number of policy iterations (plotted in a log scale);
    from left to right: the ship moves slower than  the wind ($v_s=0.5$) and faster than the wind ($v_s=1.2$).}
    \label{fig:hjbi}
 \end{figure}
 \bigskip

We further investigate the influence of the parameters $\{\eta_k\}_{k=1}^\infty$ on the accuracy and efficiency of Algorithm \ref{alg:dpi_d} in detail. Algorithm \ref{alg:dpi_d} is  carried out  with the same trial functions (7-layer networks with hidden width $H=80$ and complexity 32721) but different $\{\eta_k\}_{k=1}^\infty$ (we choose different $\eta_0$ to keep the quantity $\eta_0\eta_1$ constant), and  the numerical results are summarized in Table \ref{table:compare}. One can clearly see that a more rapidly decaying $\{\eta_k\}_{k=1}^\infty$ results in a better overall performance (in terms of accuracy and computational time), even though it requires more time to solve the linear equation for each policy iteration. The rapid decay of $\{\eta_k\}_{k=1}^\infty$ not only accelerates the superlinear convergence of the iterates $\{u^k\}_{k=1}^\infty$, but also helps to eliminate the oscillation caused by the randomness in the SGD algorithm (see Figure \ref{fig:hjbi}), which  enables us to achieve a higher accuracy with less total computational effort. However, we should keep in mind that smaller $\{\eta_k\}_{k=1}^\infty$ means that 
we need to solve  all linear equations with higher accuracy, which subsequently requires more complicated networks and more careful choices of the optimizers for $J_{k,d}$. Therefore, in general, we need to tune the balance between the superlinear convergence rate of policy iteration and the computational costs of the linear solvers, in order to achieve optimal  performance of the algorithm.

 \begin{table}[H]
 \renewcommand{\arraystretch}{1.05} 
\centering
\caption{Numerical results with different parameters $\{\eta_k\}_{k=0}^\infty$ for the scenario where the ship is slower than the wind ($v_s=0.5$).}
\label{table:compare}
\begin{tabular}[t]{@{}cccc c cccc@{}}
\toprule
 \multicolumn{4}{c}{$H=80, \eta_0=80,\eta_k=2^{-k}$} &  \phantom{a}&  
 \multicolumn{4}{c}{$H=80, \eta_0=40,\eta_k=1/k$} \\
 \cmidrule{1-4} \cmidrule{6-9}
 PI Itr & HJBI Residual & SGD Itr &  Run time & & PI Itr & HJBI Residual & SGD Itr & Run time\\ \midrule
 9 & 0.0466 & 11030 & 922s & 		&58 & 0.0252 & 32500  & 2729s \\ 
 10 & 0.0144 & 20330 & 1710s & 	&59 & 0.0201 & 39080  & 3275s \\ 
11 & 0.0046 & 45510 & 3820s &	&60 & 0.0156 & 45770 & 3836s\\ 
\bottomrule
\end{tabular}
\end{table}%

Finally,  we shall  compare the performance of 
Algorithm \ref{alg:dpi_d} (with $\eta_0=80,\eta_k=2^{-k}$) and 
the Direct Method (with $\|\cdot\|_{X,\p\Om,\textrm{tra}}=\|\cdot\|_{3/2,\p\Om,\textrm{tra}}$ in \eqref{eq:direct_X})
 by fixing the 
  trial functions (7-layer networks with hidden width $H=80$), 
  the training samples
and the learning rates of the SGD algorithms.
For both methods, we shall consider the following squared residual
for each  iterate $\hat{u}_i$ obtained from the $i$-th SGD iteration
(see \eqref{eq:hjbi_residual}):
$$
\textrm{HJBI Residual}\coloneqq \|F(\hat{u}_i)\|^2_{0,\Om,\textrm{val}}+\|\hat{u}_i-g\|^2_{3/2,\p\Om,\textrm{val}}.
$$
Figure \ref{fig:direct_pi_slow} (left) presents the decay of the residuals as the number of SGD iterations tends to infinity, which demonstrates 
the efficiency improvement of Algorithm \ref{alg:dpi_d} over the Direct Method.
The superlinear convergence of policy iteration helps to provide 
 better initial guesses of the SGD algorithm, 
which leads to 
a more rapidly decaying loss curve with smaller noise (on the validation samples); 
the HJBI residuals obtained  in the last 1000 SGD iterations of Algorithm \ref{alg:dpi_d} (resp.~the Direct Method)
oscillates around the value 0.0028 (resp.~0.0144) with a standard derivation 0.00096 (resp.~0.0093).
\color{black}

\begin{figure}[!ht]
    \centering
    \includegraphics[width=0.43\columnwidth,height=6.2cm]{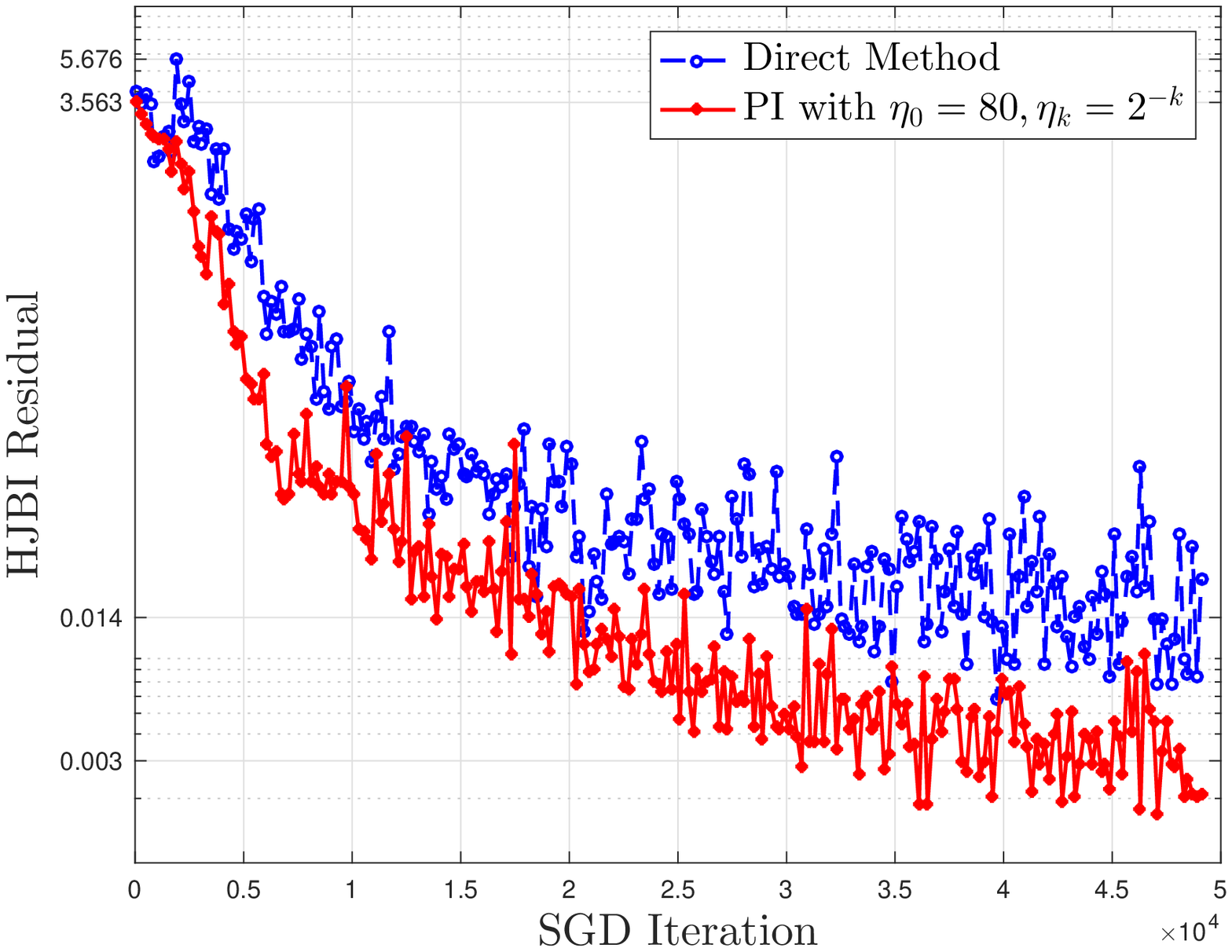}
    \includegraphics[width=0.43\columnwidth,height=6.2cm]{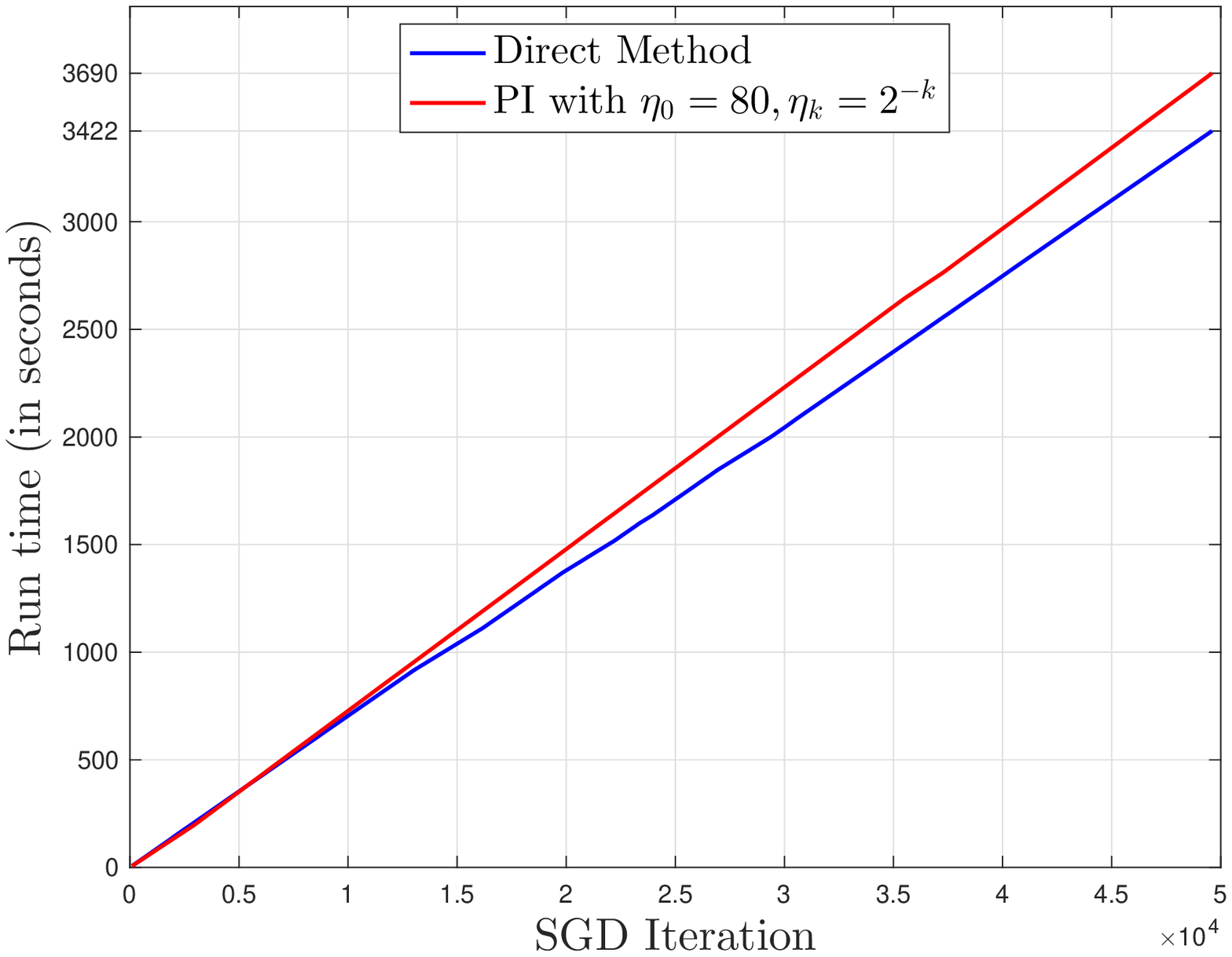}
    \caption{ Performance comparison  of the Direct Method and Algorithm \ref{alg:dpi_d}
    for the scenario where the ship is slower than the wind ($v_s=0.5$);
    from left to right:   residuals (plotted in a log scale) and  overall runtime for all SGD iterations.}
    \label{fig:direct_pi_slow}
 \end{figure}

\section{Conclusions}
This paper develops a neural network based policy iteration algorithm for solving HJBI boundary value problems arising in stochastic differential games  of diffusion processes with controlled drift and state constraints. 
We establish the $q$-superlinear convergence of the algorithm  in $H^2(\Om)$ with an arbitrary initial guess, and also the pointwise (almost everywhere)  convergence of the numerical solutions and their (first and second order) derivatives, which  subsequently leads to convergent approximations of optimal feedback controls. 
{The convergence results also hold for  
general trial functions, including
 kernel functions  and high-order separable polynomials      used in global spectral methods.}
Numerical examples for stochastic Zermelo navigation problems are presented to illustrate the theoretical findings.

To the best of our knowledge, this is the first paper which demonstrates the global superlinear convergence of policy iteration for  nonconvex HJBI equations in function spaces, and proposes convergent neural network based numerical methods for solving the  solutions of nonlinear boundary value problems and their  derivatives. 
 Natural next steps would be to extend the inexact policy iteration algorithm to parabolic HJBI equations, and to employ neural networks with tailored architectures to  enhance the efficiency of the algorithm for solving high-dimensional problems.

\appendix
\section{Some fundamental results}


Here, we collect some well-known results which are used frequently in the paper.

We start with the well-posedness of strong solutions to   Dirichlet boundary value problems. In the sequel, we shall denote by $\tau$ the trace operator.

\begin{Theorem}(\cite[Theorem 1.2.19]{garroni2002})\l{thm:D_regularity}
Let $\Om$ be  a bounded $C^{1,1}$ domain.
 Suppose that for all  $1\le i,j\le n$, $a^{ij}$ is in $C(\bar{\Om})$, and $b^i,c$ are in $L^\infty(\Om)$, satisfying  $c\ge 0$ and 
 \bb\l{eq:elliptic}
 \sum_{i,j=1}^na^{ij}(x)\xi_i\xi_j\ge \lambda|\xi|^2, \q \textnormal{for all $\xi\in \R^n$ and for almost every $x\in {\Om}$,}
 \ee
for some constant $\lambda>0$.
Then for every $f\in L^2(\Om)$ and  $g\in H^{3/2}(\p\Om)$, there exists a unique strong solution $u\in H^{2}(\Om)$ to the Dirichlet problem
$$
-a^{ij}\p_{ij}u+b^i\p_i u+cu=f, \; \textnormal{in $\Om$}; \q \tau u = g, \; \textnormal{on $\p\Om$},
$$
and the following estimate holds with a  constant $C$ independent of $f$ and $g$:
$$
\|u\|_{H^2(\Om)}\le C\big(\|f\|_{L^2(\Om)}+\|g\|_{H^{3/2}(\p\Om)}\big).
$$
\end{Theorem}

The next theorem shows the well-posedness of oblique  boundary value problems. 
\begin{Theorem}(\cite[Theorem 1.2.20]{garroni2002})\l{thm:N_regularity}
Let $\Om$ be  a bounded  $C^{1,1}$ domain.
Suppose that for all  $1\le i,j\le n$, $a^{ij}$ is in $C(\bar{\Om})$, and $b^i,c$ are in $L^\infty(\Om)$, satisfying $c\ge 0$ and the uniform elliptic condition \eqref{eq:elliptic}  for some  constant  $\lambda>0$.

Assume in addition, that $\{\gamma^j\}_{j=0}^n\subseteq C^{0,1}(\p\Om)$, $\gamma^0\ge 0$ on $\p\Om$, $\esssup_\Om c+\max_{\p\Om}\gamma^0>0$, and  $\sum_{j=1}^n\gamma^j\nu_j\ge \mu$ on $\p\Om$ for some constant $\mu>0$,  where $\{\nu_j\}_{j=1}^n$ are the components of the unit outer normal vector field on $\p\Om$.
Then for every $f\in L^2(\Om)$ and  $g\in H^{1/2}(\p\Om)$, there exists a unique strong solution $u\in H^{2}(\Om)$ to the following oblique   derivative problem:
$$
-a^{ij}\p_{ij}u+b^i\p_i u+cu=f, \; \textnormal{in $\Om$}; \q \gamma^j\tau(\p_ju)+\gamma^0\tau u = g, \; \textnormal{on $\p\Om$},
$$
and the following estimate holds with a  constant $C$ independent of $f$ and $g$:
$$
\|u\|_{H^2(\Om)}\le C\bigg(\|f\|_{L^2(\Om)}+\|g\|_{H^{1/2}(\p\Om)}\bigg).
$$
\end{Theorem}

We then recall several important measurability results. The  following  measurable selection theorem follows from Theorems 18.10 and 18.19 in \cite{aliprantis1999}, and ensures the existence of a measurable selector maximizing (or minimizing) a Carath\'{e}odory function.

\begin{Theorem}\l{thm:measurable_selection}
Let  $(S, \Sigma)$ a measurable space and $X$ be a separable  metrizable space. Let $\Gamma:S\rightrightarrows X$ be a  measurable set-valued mapping with nonempty compact values, and suppose $g: S \t X \to \R$ is a Carath\'{e}odory function. Define the value  function $m: S \to \R$ by $m(s)=\max_{x\in \Gamma(s)}g(s,x)$, and the set-valued map $\mu:S\rightrightarrows X$ by  $\mu(s)=\{x\in \Gamma(s)\mid g(s,x)=m(s)\}$.  
Then we have
\bn
\item The value function $m$ is measurable.
\item The set-valued mapping $\mu$ is measurable, has nonempty and compact values.  Moreover, there exists a measurable function $\psi:S\to X$ satisfying $\psi(s)\in \mu(s)$ for each $s\in S$.
\en
\end{Theorem}

The following theorem  shows the $\argmax$ set-valued mapping is upper hemicontinuous.

\begin{Theorem}(\cite[Theorem 17.31]{aliprantis1999})\l{thm:uhc}
Let $X,Y$ be topological spaces,  $\Gamma\subset Y$ be a nonempty compact subset, and $g:X\t \Gamma\to \R$ be a  continuous function. Define the value function $m: X \to \R$ by $m(x)=\max_{y\in \Gamma}g(x,y)$, and the set-valued map $\mu:X\rightrightarrows Y$ by  $\mu(x)=\{y\in \Gamma\mid g(x,y)=m(x)\}$.  
Then $\mu$ has nonempty and compact values. Moreover, 
if $Y$ is Hausdorff, then $\mu$ is  upper hemicontinuous, i.e., for every $x\in X$ and  every neighborhood $U$ of $\mu(x)$, there is a neighborhood $V$ of $x$ such that $z\in V$ implies $\mu(z)\subset  U$. 
\end{Theorem}



Finally, we present a  special case of \cite[Theorem 2]{dontchev2012}, which characterizes  $q$-superlinear convergence of quasi-Newton methods for a class of semismooth operator-valued equations.

\begin{Theorem}\l{thm:quasi-Newton}
Let  $Y,Z$ be two Banach spaces, and $F:Y\to Z$ be a given function with a zero $y^*\in Y$. Suppose there exists an open neighborhood $V$ of $y^*$ such that $F$ is semismooth with a generalized differential $\p^*F$  in $V$,  and 
there exists a constant $L>0$ such that
\begin{align*}
\|y - y^*\|_Y/L\le \|F(y) - F(y^*)\|_Z \le L\|y - y^*\|_Y, \q \fa y\in V.
\end{align*}
 For some starting point $y^0$ in $V$, let the sequence $\{y^k\}_{k\in \N}\subset V$ satisfy $y^k \not=y^*$ for all $k$, and be generated by the following quasi-Newton method:
$$ B_ks^k = -F(y^k),\q y^{k+1}=s^k+y^k,\q  k=0,1,\ldots$$
where $\{B_k\}_{k\in \N}$ is a sequence of bounded linear operators in $\cL(Y,Z)$. Let $\{A_k\}_{k\in \N}$ be a sequence of generalized differentials of $F$ such that  $A_k \in \p^*F(y^k)$ for all $k$, and let $E_k = B_k - A_k$. 
Then $y^k \to y^*$ $q$-superlinearly if and only if $\lim_{k\to\infty}y^k= y^*$ and $\lim_{k\to\infty}\|E_ks^k\|_Z/\|s^k\|_Y=0$.
 
\end{Theorem}

\section{Proof of Theorem \ref{thm:dpi_global_n}}\l{sec:proof_N}

Let $u^0\in \cF$ be an arbitrary initial guess, we shall assume  without loss of generality that Algorithm  \ref{alg:dpi_n}  runs infinitely, i.e., $\|u^{k+1}-u^k\|_{H^2(\Om)}>0$ and $u^k\not =u^*$ for all $k\in \N\cup\{0\}$.

 We first show $\{u^k\}_{k\in\N}$  converges to the unique solution $u^*$ in $H^2(\Om)$. For each $k\ge 0$,  we can deduce from \eqref{eq:J_error_n} that there exists  $f^e_k\in L^2(\Om)$ and $g^e_k\in H^{1/2}(\p\Om)$ such that 
\bb\l{eq:linear_inexact_n}
L_ku^{k+1}-f_k=f^e_k, \q \textnormal{in $\Om$}; \q B u^{k+1}=g^e_k, \q \textnormal{on $\p\Om$},
\ee
and $\|f^e_k\|^2_{L^2(\Om)}+\|g^e_k\|^2_{H^{1/2}(\Om)}\le \eta_{k+1}(\|u^{k+1}-u^k\|^2_{H^2(\Om)})$ with $\lim_{k\to\infty}\eta_{k}= 0$. 
Then, we can proceed as in the proof of Theorem  \ref{thm:dpi_global}, and conclude that if $c\ge \ul{c}_0$ with a sufficiently large $\ul{c}_0$, then  $\{u^k\}_{k\in\N}$ converges  to the solution $u^*$ of \eqref{eq:N}.

The $q$-superlinear convergence of Algorithm \ref{alg:dpi_n} can then be deduced by interpreting the algorithm as a quasi-Newton method for the operator equation $\bar{F}(u)=0$, with the operator 
$\bar{F}: u\in H^2(\Om)\to (F(u),B u)\in Z$,
where we introduce the Banach space $Z\coloneqq L^2(\Om)\t H^{1/2}(\p\Om)$  with the usual product norm $\|z\|_Z\coloneqq \|z_1\|_{L^2(\Om)}+\|z_2\|_{H^{1/2}(\p\Om)}$ for each $z=(z_1,z_2)\in Z$. Since $B\in \cL(H^2(\Om),H^{1/2}(\p\Om))$, we can directly infer from Corollary \ref{cor:semismooth_F} that $\bar{F}:H^2(\Om)\to Z$ is semismooth in $H^2(\Om)$,  with a generalized differential  $M_k=(L_k, \gamma^i\tau( \p_i)+\gamma^0\tau)\in \p^*\tilde{F}(u^k)\subset \cL(H^2(\Om),Z)$ for all $k\in \N\cup\{0\}$. Then, for each $k\ge 0$, by following the same arguments as in Theorem \ref{thm:dpi_global},  we can construct a perturbed operator $\delta M_k\in \cL(H^2(\Om),Z)$, such that  \eqref{eq:linear_inexact_n} can be  equivalently written as $(M_k+\delta M_k)s_k=-\bar{F}(u^k)$ with $s_k=u^{k+1}-u^k$, and $\|\delta M_ks_k\|/\|s^k\|_{H^2(\Om)}\le \sqrt{2\eta_0\eta_{k+1}}\to 0$, as $k\to\infty$. Finally, the regularity theory of elliptic oblique   derivative problems (see Theorem \ref{thm:N_regularity}) shows that $M_k$ is nonsingular for each $k$, and  $\|M_k^{-1}\|_{\cL(Z,H^2(\Om))}\le C$ for some constant $C$ independent of $k$. Hence we can verify that 
there exists a neighborhood $V$ of $u^*$ and a constant $L>0$, such that 
$$\|u - u^*\|_{H^2(\Om)}/L\le \|\bar{F}(u) - \bar{F}(u^*)\|_Z \le L\|u - u^*\|_{H^2(\Om)},\q \fa u\in V,$$
which allows us to conclude from Theorem  \ref{thm:quasi-Newton} the $q$-superlinear convergence of $\{u^k\}_{k\in \N}$.


\end{document}